\documentclass[runningheads]{svjour2}
\smartqed  
\usepackage{graphicx}
\usepackage{amsmath}
\usepackage{amssymb}
\usepackage{epsfig}
\usepackage{hyperref}
\usepackage{verbatim}
\allowdisplaybreaks
\newtheorem{Proposition}{Proposition}[section]
  \newtheorem{Remark}[Proposition]{Remark}
  \newtheorem{Corollary}[Proposition]{Corollary}
  \newtheorem{Lemma}[Proposition]{Lemma}
  
  \newtheorem{Theorem}{Theorem}[section]
 
 \newtheorem{Definition}[Proposition]{Definition}

\journalname{Archive for Rational Mechanics and Analysis}
\begin{document}

\title{Existence, Uniqueness, Analyticity, and Borel Summability of 
Boussinesq and Magnetic B$\acute{\bf{e}}$nard Equations
}

\titlerunning{Borel Summability of Boussinesq and MHD Equations}        

\author{H. Rosenblatt          \and
        S. Tanveer 
}


\institute{H. Rosenblatt  \at
              Department of Mathematics: The Ohio State University \\
              222 Math Tower 231 West 18th Avenue Columbus, OH 43210-1174\\
              Tel.: 614-292-7648\\
              Fax: 614-292-1479\\
              \email{rosenblatt@math.ohio-state.edu}           
           \and
           S. Tanveer \at
              402 Math Tower 231 West 18th Avenue Columbus, OH 43210-1174\\
              Tel.: 614-292-5710\\
              \email{tanveer@math.ohio-state.edu}
}

\date{\notused}

\maketitle

\begin{abstract}
Through Borel summation methods, 
we analyze two different variations of the Navier-Stokes equation
--the Boussinesq
equation for fluid motion and temperature field and the
the magnetic B$\acute{\textnormal{e}}$nard equation which approximates electro-magnetic
effects on fluid flow under some simplifying assumptions. In the Boussinesq equation,
\begin{align}\label{star}
u_t-\nu\Delta u&=-P[ u\cdot \nabla u - a e_2\Theta] +f\\ \nonumber
\Theta_t-\mu \Delta \Theta&= -u\cdot \nabla \Theta,
\end{align}
where $d=2$ or $3$ is the dimension, $u:\mathbb{R}^d\times\mathbb{R}^+\rightarrow \mathbb{R}^d$, and $\Theta:\mathbb{R}^d\times\mathbb{R}^+\rightarrow \mathbb{R}$. For 
the magnetic B$\acute{\textnormal{e}}$nard equation,
\begin{align}\label{star2}
v_t-\nu\Delta v&=-P[ v\cdot \nabla v - \frac{1}{\mu \rho}B\cdot \nabla B] +f\\ \nonumber
B_t-\frac{1}{\mu\sigma} \Delta B&=-P[v\cdot \nabla B-B\cdot \nabla v],
\end{align}
where $v,B:\mathbb{R}^d\times\mathbb{R}^+\rightarrow \mathbb{R}^d$.

This method has previously been applied 
to the Navier-Stokes equation in \cite{smalltime}, \cite{scripta},  
and \cite{longtime}. We show that this approach 
can be used to show local existence for the Boussinesq and 
magnetic B$\acute{\textnormal{e}}$nard equation, either for $d=2$ or $d=3$.
We prove that 
an equivalent system of integral equations in each case
has a unique solution, which is exponentially bounded for 
$p\in \mathbb{R}^{+}$, $p$ being the Laplace dual variable of $1/t$. 
This implies the local 
existence of a classical solution to (\ref{star}) and (\ref{star2}) in
a complex $t$-region that includes a real positive time ($t$)-axis segment.
Further, it is shown that within this real time
interval,  for analytic initial data and forcing,
the solution remains analytic and has the same analyticity strip width. Further,
under these conditions, the solution is 
Borel summable, implying that
that formal series in time is Gevrey-1 asymptotic for small $t$. We also determine conditions
on the integral equation solution in each case over a finite interval $[0, p_0]$
that result in a better estimate for existence time of the PDE solution.
\end{abstract}

\section{Introduction}
\label{intro}
We consider two variations of the incompressible Navier-Stokes equation.
In the first case, we consider the coupling of temperature field with
fluid flow under the assumption that the temperature induced changes in density have
negligible effects on momentum, but
cause a significant buoyant force. The corresponding Boussinesq equation
for $u:\mathbb{R}^d\times\mathbb{R}^+\rightarrow \mathbb{R}^d$ and $\Theta:\mathbb{R}^d\times\mathbb{R}^+\rightarrow \mathbb{R}$ with $d=2,\, 3$ are     
\begin{align}
\label{B}
u_t-\nu\Delta u&=-P[ u\cdot \nabla u - a e_2\Theta] +f ~~,~~~~~u(x,0) = u_0 (x)
\\ \nonumber
\Theta_t-\mu \Delta \Theta&= -u\cdot \nabla \Theta ~~~,~~~\Theta (x, 0) =\Theta_0 (x) 
\end{align}
where $P=I-\nabla \Delta^{-1}(\nabla\cdot )$ is the Hodge projection operator to 
the space of divergence free vector fields and $e_2$ is the unit vector aligned opposite to gravity and parameter $a$ is proportional to gravity. Here $(u, \Theta)$ corresponds
to the fluid velocity and temperature field. 
Using standard energy methods, see for instance \cite{Temam}, existence of Leray type solutions in $L^{\infty}(0,T,L^2(\mathbb{R}^d))\cap L^2(0,T,H^1(\mathbb{R}^d))$ follows easily for any $T>0$. In $\mathbb{R}^2$ a unique classical global solution can be shown to exist for all time. Further, in $\mathbb{R}^3$ there is a unique solution under the additional assumption that the solution lies in $L^{\infty}(0,T,H^1(\mathbb{R}^3)$. In \cite{Cannon}, local existence and uniqueness for Boussinesq equation are shown in $L^p(0,T,L^q(\mathbb{R}^d))$ for $d<p<\infty$ and $\frac{d}{p}+\frac{2}{q}\leq 1$.

For the second problem, we study the 
the viscous magnetic B$\acute{\textnormal{e}}$nard equation, or MHD equation, which arises in the
motion of a magnetic fluid in situations where displacement current and 
charge density variations are negligible \cite{Chandra}.
The equations for $v, B:\mathbb{R}^d\times\mathbb{R}^+\rightarrow \mathbb{R}^d$ are
\begin{align}\label{Ber}
v_t-\nu\Delta v&=-P[ v\cdot \nabla v - \frac{1}{\mu \rho}B\cdot \nabla B] +f ~~~,~~~v(x, 0) =v_0 (x) \\ \nonumber
B_t-\frac{1}{\mu\sigma} \Delta B&=-P[v\cdot \nabla B-B\cdot \nabla v] ~~~~,~~~B(x, 0) = B_0 (x) 
\end{align}
\noindent where $d=2, 3$ as before, $v$ is the fluid velocity, 
$B$ is the magnetic field, while $\nu$, $\rho$, $\mu$ and $\sigma$ are constants related to
fluid viscosity, density, magnetic permeability and electric conductivity respectively. The question of regularity of solutions to the MHD equation in two and three dimensions has been well studied. Duraut and Lions \cite{DurautandLions} constructed a class of global weak solutions and a local class of strong solutions using energy methods in both two and three dimensions. In the two dimensional case, uniqueness and smoothness were established for all time. More generally, Sermange and Temam \cite{TemamMHD} showed existence in three dimensions in the class $L^{\infty}(0,T,L^2(\mathbb{R}^d))\cap L^2(0,T,H^1(\mathbb{R}^d))$ and uniqueness assuming the solution lies in $L^{\infty}(0,T,H^1(\mathbb{R}^d))$. Many others \cite{reg-of-weak-solution1}, \cite{reg-of-weak-solution2}, \cite{reg-of-weak-solution3}, and \cite{reg-of-weak-solution4} have a variety of results improving regularity.

In both the problems above, the existence of classical solutions, globally in time, remains an open problem as it is for the 3-D Navier
Stokes equation. Control of a higher order energy norm (like the $H^1$ norm of velocity) has remained 
a serious impediment for a long time. This
motivates one to look at other formulations of the existence problem that do not rely on energy bounds.

The primary purpose of this paper is to show that the 
Borel transform methods, developed earlier in \cite{smalltime} and \cite{longtime} in the context of  Navier-Stokes equation, can be extended to determine classical PDE solutions for the Boussinesq and magnetic B$\acute{\textnormal{e}}$nard equations. This provides an alternate existence and uniqueness
theory for a class of nonlinear PDEs for which the question of global existence
of solution to the PDE becomes one of asymptotics for known solution to the associated nonlinear integral equations. While this asymptotics problem is difficult and yet to be resolved, it is shown (Thm \ref{improved existence}) how information about solution on a finite interval in the dual variable for specific initial condition and forcing may be used for obtaining better exponential bounds in the Borel plane and therefore better existence time for classical solutions to the PDEs.

Further, many analyticity properties readily follow from this representation. Time analyticity for $\Re \frac{1}{t}>\alpha$ follow from the solution representation. We also prove that the classical $H^2(\mathbb{R}^d)$ solution, which is unique, has the Laplace transform representation given here, provided initial data and forcing are in $L^1\cap L^{\infty}$ in Fourier space. Furthermore, for analytic 
initial data and forcing, we prove that the formal expansion in powers of
$t$ is Borel summable and hence Gevrey asymptotic for small $t$. 
In the latter case, it is also shown that the associated power series 
in the Borel plane has a radius of convergence independent of size 
of initial data and forcing when initial data and 
forcing have a fixed number of Fourier modes, 
this is useful in computing the solution in the Borel plane.  

\section{Main Results}\label{sec1}

We first write the equations as integral equations in Fourier space. We denote by $\hat{f}$ the Fourier transform of $f$ and $\hat{*}$ the Fourier convolution. The Fourier transform operator is denoted by $\mathcal{F}$. As usual, a repeated index $j$ denotes the sum over $j$ from $1$
to $d$. 
$P_k$ is the Fourier transform of the Hodge projection and has the representation
\begin{equation*}
P_k\equiv \left( 1-\frac{k(k\cdot)}{|k|^2}\right).
\end{equation*}
Moreover $u$, $v,$ and $B$ are divergence free. Formal derivation\footnote{While derivation is formal, 
in the space of functions where existence
is proved, it will be clear the integral and differential formulations are equivalent} based on inversion
of the heat operator in Fourier space in (\ref{B}) leads to the following integral equations:
\begin{align}\label{FB2}
\hat{u}(k,t)=- \int_0^t& e^{-\nu|k|^2(t-\tau)}\left(i k_j P_k[\hat u_j\hat *\hat u-a e_2\hat{\Theta}](k,\tau)-\hat f(k)\right)d\tau\\ \nonumber
&+e^{-\nu|k|^2t}\hat u_0(k)\\ \nonumber
\hat{\Theta}(k,t)=- \int_0^t& e^{-\mu|k|^2(t-\tau)}\left( i k_j[\hat u_j\hat *\hat{\Theta}](k,\tau)\right)d\tau+e^{-\mu|k|^2t} \hat{\Theta}_0(k)
\end{align}
and, for the magnetic B$\acute{\textnormal{e}}$nard equation (\ref{Ber}), one obtains
\begin{align}\label{MB2}
\hat{v}(k,t)=- \int_0^t &e^{-\nu|k|^2(t-\tau)}\left(i k_j P_k\left[\hat v_j\hat *\hat v-\frac{1}{\mu \rho}\hat B_j\hat *\hat B\right](k,\tau)-\hat f(k)\right)d\tau\\ \nonumber
&+e^{-\nu|k|^2t}\hat v_0(k)\\ \nonumber
\hat B(k,t)=- \int_0^t &e^{\frac{-|k|^2(t-\tau)}{\mu \sigma}}\left(i k_j P_k\left[\hat v_j\hat *\hat B-\hat B_j\hat *\hat v\right](k,\tau)\right)d\tau+e^{\frac{-|k|^2t}{\mu \sigma}}\hat B_0(k).
\end{align}

\begin{Remark}
We may assume the initial conditions $u_0$ in the Boussinesq equation
and $v_0$, $B_0$ for the B$\acute{\textnormal{e}}$nard equation, as well
as the
forcing $f$ are divergence free, since any non-zero divergence part of $f$ can be included
in a gradient term, which has been projected away. 
We assume $f=f(x)$ to be time independent for simplicity 
although a time dependent $f$ with some restrictions may be treated in a similar manner. Additional forcing terms on the temperature and magnetic equations can be accommodated in the formalism here.
\end{Remark}

\begin{Definition}\label{gamma,beta} We introduce the norm $||\cdot||_{\gamma,\beta}$ for some $\beta \geq 0$ 
and $\gamma >d$ by
\begin{equation*}
|| \hat f ||_{\gamma,\beta}=\sup_{k\in \mathbb{R}^d} (1+|k|)^{\gamma}e^{\beta |k|}|\hat f(k)|, \textnormal{ where }\hat{f}(k)=\mathcal{F}[f(\cdot)](k).
\end{equation*}
\end{Definition}
\begin{Definition}\label{L1}
We also use the space $L^1\cap L^{\infty}$ with the norm defined by 
\begin{equation*}
||\hat{f}||_{L^1\cap L^{\infty}}=\max\left\{\int_{\mathbb{R}^d} |\hat{f}(k)|dk, \sup_{k\in\mathbb{R}^d}|\hat{f}(k)|\right\}.
\end{equation*}
\end{Definition}

In the case when results hold either for $\| \cdot \|_{\gamma, \beta}$ or $\| \cdot \|_{L^1\cap L^{\infty}}$ norm,  we will use $||\cdot||_{N}$ for brevity of notation. 

We assume 
$||(1+|k|)^2( \hat{u}_0,\hat{\Theta}_0)||_{N}<\infty$, 
$||(1+|k|)^2(\hat{v}_0,\hat{B}_0)||_{N}<\infty$, 
and $||\hat{f}||_{N}<\infty$ in what follows. 
If $\| \cdot \|_N= \| \cdot \|_{\gamma,\beta}$ and $\beta>0$ then the initial condition and forcing are 
real analytic in $x$ in a strip of width at least $\beta$.
 
\begin{Theorem}\label{existence} (Boussinesq Existence and Uniqueness) 

If $||(1+|\cdot|)^2(\hat{u}_0, \hat{\Theta}_0)||_{N}
<\infty$ and $||\hat{f}||_{N}<\infty$, then the following hold.

i) The Boussinesq equation (\ref{FB2}) has a solution 
$(\hat{u}, \hat{\Theta})(k,t)$ such that $||(\hat{u}, \hat{\Theta})(\cdot,t)||_{N}<\infty$ 
for $\Re \frac{1}{t} >\omega$ for $\omega$ sufficiently large. 
Specifically, (\ref{1-2.39}) holds, where $(\hat{u}_1,\hat{\Theta}_1)$, defined in (\ref{u1}), depends on the initial data and forcing. 

ii) The solution has the Laplace transform representation
\begin{equation}\label{Laplacetransform1}
(\hat{u}, \hat{\Theta})(k,t)=(\hat{u}_0, \hat{\Theta}_0)(k)+\int_0^{\infty}(\hat{H}, \hat{S})(k,p)e^{-p/t}dp
\end{equation}
where $(\hat{H},\hat{S})$ satisfies a set of integral equations that has a unique solution for $||(\hat{H},\hat{S})(\cdot,p)||_{N}e^{-\omega p}\in L^1(0,\infty)$. From this representation $(u, \Theta)(x,t)=\mathcal{F}^{-1}[(\hat{u},\hat{\Theta})(k)](x,t)$ is 
analytic in $t$ for $\Re \frac{1}{t}>\omega$. This implies that if $\beta>0$ then $(u, \Theta)$ is analytic in $x$ in a strip of the same width as the analyticity strip for the initial data and forcing for any $t\in [0,\omega^{-1})$. 

iii) Further for this solution, $||(1+|\cdot|)^2(\hat{u}, \hat{\Theta})(\cdot,t)||_{N} < \infty$ for $t\in (0,\omega^{-1})$. 
Moreover, $(u, \Theta)(x,t)$ solves (\ref{B}) and is the unique solution in $L^{\infty}(0,T,H^2(\mathbb{R}^d))$. In other words, given any solution in $H^2(\mathbb{R}^d)$ to the Boussinesq equation for which the initial data and forcing satisfy the given assumption then the solution has the representation (\ref{Laplacetransform1}).

iv) A sufficient condition for global existence of smooth solution to the Boussinesq equation is that $e^{-\omega p}||(\hat{H},\hat{S})(\cdot,p)||_{N}\in L^1(0,\infty)$ for any $\omega> 0$.
\end{Theorem}

\begin{Theorem}\label{existenceMHD}(MHD Existence and Uniqueness) If $||(1+|\cdot|)^2(\hat{v}_0, \hat{B}_0)||_{N}<\infty$ and 
$||\hat{f}||_{N}<\infty$, then the following hold.
 
i) The magnetic B$\acute{\textnormal{e}}$nard equation (\ref{MB2}) has a solution 
$(\hat{v}, \hat{B})(k,t)$ such that $||(\hat{v}, \hat{B})(\cdot,t)||_{N}<\infty$ for $\Re \frac{1}{t}
>\alpha$ for $\alpha$ sufficiently large.
Specifically, (\ref{2-2.39}) holds, 
where $(\hat{v}_1,\hat{B}_1)$, defined in (\ref{v1}), depends on the initial data and forcing.
 
ii) The solution has the Laplace transform representation
\begin{equation}\label{Laplacetransform2}
(\hat{v}, \hat{B})(k,t)=(\hat{v}_0, \hat{B}_0)(k)+\int_0^{\infty}(\hat{W}, \hat{Q})(k,p)e^{-p/t}dp
\end{equation}
where $(\hat{W},\hat{Q})$ satisfies a set of integral equations that has a unique solution for $||(\hat{W},\hat{Q})(\cdot,p)||_{N}e^{-\alpha p}\in L^1(0,\infty)$. From this representation $(v, B)(x,t)=\mathcal{F}^{-1}[(\hat{v},\hat{B})(k)](x,t)$ is analytic in $t$ 
for $\Re \frac{1}{t}>\alpha$. This implies that if $\beta>0$ then 
$(v, B)$ is analytic in $x$ in a strip of the same width as the analyticity strip for the initial data and forcing for any $t\in [0,\alpha^{-1})$.

iii) Further for this solution, $||(1+|\cdot|)^2(\hat{v}, \hat{B})(\cdot,t)||_{N} <\infty$ 
for $t\in (0,\alpha^{-1})$. Moreover, $(v,B)(x,t)$ is the unique solution to (\ref{Ber}) in $L^{\infty}(0,T,H^2(\mathbb{R}^d))$. In other words, given any solution in $H^2(\mathbb{R}^d)$ to the MHD equation for which the initial data and forcing satisfy the given assumption then the solution has the representation (\ref{Laplacetransform2}).

iv) A sufficient condition for global existence of smooth solution to the magnetic B$\acute{\textnormal{e}}$nard equation is that $e^{-\alpha p}||(\hat{W},\hat{Q})(\cdot,p)||_{N}\in L^1(0,\infty)$ for any $\alpha> 0$.
\end{Theorem}

\begin{Remark}
If the initial condition and forcing are known to be in $L^1$ in Fourier space but not necessarily to be in $L^{\infty}$, then we have a unique solution to (\ref{FB2}) or (\ref{MB2}) for which $||(\hat{u},\hat{\Theta})||_{L^1(\mathbb{R}^d)}<\infty$ for $t\in (0,\omega^{-1})$, respectively $||(\hat{v},\hat{B})||_{L^1(\mathbb{R}^d)}<\infty$ for $t\in (0,\alpha^{-1})$. This solution is smooth pointwise by instantaneous smoothing and solves the corresponding equation (\ref{B}) or (\ref{Ber}). What is not known is 
whether the corresponding $(u, \Theta)$ or $(v, B)$ in the physical space is in $L^\infty(0,T,H^2(\mathbb{R}^d))$.
\end{Remark}

\begin{Remark} The guaranteed existence time $T=\omega^{-1}$ 
or $\alpha^{-1}$,
depending on the equation being considered, depends on $||(1+|\cdot|)^2(\hat{u}_0,\hat{\Theta}_0)(\cdot)||_{N}$ 
or $||(1+|\cdot|)^2(\hat{v}_0,\hat{B}_0)(\cdot)||_{N}$. This condition 
can be weakened using an accelerated version of the Borel transform as in \cite{longtime}, {\it i.e.} using
an alternate representation for $n > 1$:
\begin{equation}\label{accelerated}
(\hat{u},\hat{\Theta})(k,t)=(\hat{u}_0,\hat{\Theta}_0)(k)+\int_0^{\infty}(\hat{H},\hat{S})(k,q)e^{-q/(t^n)}dq
\end{equation}
\end{Remark}

\begin{Remark}
Using an accelerated
variable instead of $p$, as in (\ref{accelerated})
for $n$ sufficiently large, we expect to be able to prove that in the case without forcing 
for the periodic case $x \in \mathbb{T}^d$, global solutions of the PDEs implies that the growth
rate $\alpha$ for associated integral equation solution is arbitrarily small, a result already shown for 3-D Navier-Stokes \cite {longtime}.
\end{Remark}

\begin{Theorem} \label{Borel summability} (Borel Summability) i) For analytic initial data and forcing and $\beta >0$ the solution to the Boussinesq equation, $(u, \Theta)$, and the solution to the magnetic B$\acute{\textnormal{e}}$nard equation, $(v, B)$, are Borel summable in $t$. 
That is there exists $(H, S)(x,p)$ and $(W, Q)(x,p)$ analytic in a neighborhood of $
\{ 0 \} \cup \mathbb{R}^{+}$, exponentially bounded and analytic in $x$ for $|$Im$(x)|<\beta$ such that
\begin{equation}\label{Laplacetransform3}
(u, \Theta)(x,t)=(u_0, \Theta_0)(x)+\int_0^{\infty}(H, S)(x,p)e^{-p/t}dp
\end{equation}
and
\begin{equation}\label{Laplacetransform4}
(v, B)(x,t)=(v_0, B_0)(x)+\int_0^{\infty}(W, Q)(x,p)e^{-p/t}dp.
\end{equation}
\noindent In particularly by Watson's Lemma, as $t\rightarrow 0^+$
\begin{equation*} 
(u, \Theta)(x,t)\sim (u_0, \Theta_0)(x)+
\sum_{m=1}^{\infty}(u_m, \Theta_m) (x)t^m
\end{equation*}
and
\begin{equation*}
(v, B)(x,t)\sim (v_0, B_0)(x)+\sum_{m=1}^{\infty}(v_m, B_m)(x)t^m,
\end{equation*}
\noindent where $|(u_m, \Theta_m)(x)|\leq m! A_0 D_0^m$ and 
$|(v_m, B_m) (x)|\leq m!\tilde A_0\tilde D_0^m$ with constants
$A_0$, $\tilde A_0$, $D_0$, and $\tilde D_0$ generally dependent on the initial condition and forcing through Lemma \ref{3.2.}.

ii) Further, if analytic initial data and forcing have only a finite number of Fourier modes and $\beta>0$, the solutions $(\hat{H},\hat{S})(k,p)$ and $(\hat{W},\hat{Q})(k,p)$ have radii of convergence independent of the size of the initial data and forcing.
\noindent In particularly, constants
$A_0$, $\tilde A_0$ depend on the initial condition and forcing and constants $D_0$, and $\tilde D_0$ depend on the number of Fourier modes of the initial condition and forcing but are independent of the size of initial data and forcing. 
\end{Theorem}

\begin{Remark}
In the case $\beta >0$, we do not need the restriction $\gamma>d$. If $||\hat{u}||_{\gamma,\beta}<\infty$, then for $\beta'\in (0,\beta)$ we have for any $n\in \mathbb{N}$,  $||\hat{u}||_{\gamma+n,\beta'}<\infty$.
\end{Remark}

\begin{Remark}
Besides the nature of early time asymptotics, the finite radius of convergence of the series in $p$ being independent of size of initial condition, at least for data with finite Fourier modes, helps determine the solution in $[0, p_0]$. Knowledge of the solution on $[0, p_0]$ can be exploited (as in the following Theorem \ref{improved existence}) to compute a revised estimate on $\omega$ and $\alpha$ for specific initial data and forcing.
\end{Remark}

Let $(\hat{H},\hat{S})(k,p)$ be the solution to (\ref{1-2.18}) provided by Lemma \ref{2.8.}. Define
\begin{equation*}
(\hat{H}, \hat{S})^{(a)}(k,p)=\begin{cases}(\hat{H}, \hat{S})(k,p) \mbox{ for } p\in (0,p_0]\subset \mathbb{R}^{+}\\
0 \mbox{ otherwise}\end{cases}
\end{equation*}
\noindent and
\begin{align*}
\hat{H}^{(s)}(k,p)&=\frac{i k_j\pi}{2|k|\sqrt{\nu p}}\int_0^{\min(p,2p_0)} \mathcal{G}(z,z')\hat{G}_j^{[1],(a)}(k,p')dp'+2\hat u_1(k)\frac{J_1(2|k|\sqrt{\nu p})}{2|k|\sqrt{\nu p}}\\ 
&\quad+\frac{a\pi}{2|k|\sqrt{\nu p}}\int_0^{\min(p,p_0)} \mathcal{G}(z,z')P_k[e_2\hat{S}^{(a)}(k,p')]dp'\\ 
\hat{S}^{(s)}(k,p)&=\frac{i k_j\pi}{2|k|\sqrt{\mu p}}\int_0^{\min(p,2p_0)} \mathcal{G}(\zeta ,\zeta ')\hat{G}_j^{[2],(a)}(k,p')dp'+2\hat \Theta_1(k)\frac{J_1(2|k|\sqrt{\mu p})}{2|k|\sqrt{\mu p}}
\end{align*}
\noindent where 
\begin{align}\nonumber
\hat{G}_j^{[1],(a)}(k,p)&=-P_k[\hat{u}_{0,j}\hat{*}\hat{H}^{(a)}+\hat{H}_j^{(a)}\hat{*}\hat{u}_0+\hat{H}_j^{(a)}\, ^{\ast}_{\ast}\hat{S}^{(a)}]\\ \nonumber
\hat{G}_j^{[2],(a)}(k,p)&=-[\hat{u}_{0,j}\hat{*}\hat{S}^{(a)}+\hat{H}_j^{(a)}\hat{*}\hat{\Theta}_0+\hat{S}_j^{(a)}\, ^{\ast}_{\ast}\hat{S}^{(a)}].
\end{align}
Notice if $(\hat{H}, \hat{S})^{(a)}(k,p)$ is known, then $\hat{H}^{(s)}(k,p)$, $\hat{S}^{(s)}(k,p)$, $G_j^{[1],(a)}(k,p)$, and $G_j^{[2],(a)}(k,p)$ are also known functions given by (\ref{G_j}). Also, recall $\hat{u}_1$ and $\hat{\Theta}_1$ are quantities based on the initial condition and forcing given in (\ref{u1}).

\begin{Theorem} \label{improved existence} (Improved exponential estimates) Assume $\epsilon_1$, $B_3$ and $b$ are
functionals of the forcing $f$, initial condition $(\hat{u}_0,\hat{\Theta}_0)$, and 
the solution $({\hat H}, {\hat S} )$ to the
set of integral equations (\ref{N}) on
a finite interval $[0, p_0]$,
determined from the following equations for any chosen $\omega_0 \ge 0$: 
\begin{equation*}
b=\omega_0 \int_{p_0}^{\infty}e^{-\omega_0 p}||(\hat{H}, \hat{S})^{(s)}(\cdot,p)||_{N}dp
\end{equation*}
\begin{equation*}
\epsilon_1=\mathcal{B}_1+\mathcal{B}_4+\int_0^{p_0}e^{-\omega_0 p}\mathcal{B}_2(p)dp,
\end{equation*}
\noindent where
\begin{multline}\nonumber
\mathcal{B}_0(k)=C_0\sup_{p_0\leq p'\leq p}|\mathcal{G}(z,z')/z|,\hspace{.5 in} \mathcal{B}_1=2\sup_{k\in\mathbb{R}^d}|k|\mathcal{B}_0(k)||(\hat{u}_0, \hat{\Theta}_0)||_{N},\\ \nonumber
\mathcal{B}_2=2\sup_{k\in\mathbb{R}^d}|k|\mathcal{B}_0(k)||(\hat{H}, \hat{S})^{(a)}(\cdot,p)||_{N},\,\, \mathcal{B}_3=\sup_{k\in\mathbb{R}^d}|k|\mathcal{B}_0(k), \,\,\mathcal{B}_4=a\sup_{k\in\mathbb{R}^d}\mathcal{B}_0(k).
\end{multline}
Then, over an extended interval, the solution
satisfies the relation
\begin{equation}\nonumber
\left \| \left ( {\hat H} (\cdot, p), {\hat S} (\cdot, p) \right ) 
\right \|_{N} \in L^1 \left (e^{-\omega p} dp \right )
\end{equation}
for any $\omega\geq \omega_0$ satisfying 
\begin{equation} \nonumber
\omega>\epsilon_1+2\sqrt{\mathcal{B}_3 b}, 
\end{equation}
where $f\in L^1(e^{-\omega p}dp)$ means $\int_0^{\infty}|f(p)|e^{-\omega p}dp<\infty$.
\end{Theorem}

\begin{Remark} This means that if the solution $(\hat{H},\hat{S})$, 
restricted to $[0,p_0]$, to the integral equation 
equivalent to the Boussinesq equation 
is known, through computation of power series in $p$ or otherwise, 
and the corresponding functionals $\epsilon$ and $\mathcal{B}_3 b$ are
small, as is the case for sufficiently rapidly decaying $({\hat H}, {\hat S})$
over a large enough interval $[0, p_0]$, then a    
long time interval of existence $(0,\omega^{-1})$ 
for classical solutions to Boussinesq equation
is guaranteed. A specific choice of $\omega_0$ may be made
to optimize the lower bound on $\omega$ in the above calculations.
The point of Theorem \ref{improved existence} is that solutions to the integral equation over a finite 
interval in $p$ (either in the form of a Taylor series in $p$, as appropriate for analytic data
and initial conditions, or in the form of numerical calculations, where rigorous error control are expected 
similar to 3-D Navier-Stokes \cite{longtime}) can lead to a revised asymptotic bounds on $\omega$ which translates into
a longer existence time for the PDE.
\end{Remark}
\begin{Remark}
A similar result holds for the magnetic B$\acute{\textnormal{e}}$nard equation with the obvious changes.
\end{Remark}

\section{Formulation of Integral Equation: Borel Transform}\label{sec2}

Our goal is to take the Borel transform and create equivalent 
integral equations. To ensure decay in $1/t$ and 
avoid dealing with delta distribution when applying the Borel transform in $1/t$, it is convenient to define $\hat{h}$, $\hat{w}$, $\hat{s}$, 
and $\hat{q}$ so that 
\begin{align}
\hat{u}(k,t)=\hat{u}_0(k)+\hat{h}(k,t)\\ \nonumber
\hat{\Theta}(k,t)=\hat{\Theta}_0(k)+\hat{s}(k,t)\\ \nonumber
\hat{v}(k,t)=\hat{v}_0(k)+\hat{w}(k,t)\\ \nonumber
\hat{B}(k,t)=\hat{B}_0(k)+\hat{q}(k,t).
\end{align}

\noindent For (\ref{FB2}), we define 
\begin{align}\label{1-2.7}
\hat{g}^{[1]}_j&:=- P_k[\hat{h}_j\hat{*}\hat{h}+\hat{h}_j\hat{*}\hat{u}_0+
\hat{u}_{0,j}\hat{*}\hat{h}]\\ \nonumber
\hat{g}^{[2]}_j&:=-[\hat{h}_j\hat{*}\hat{s}+\hat{h}_j\hat{*}\hat{\Theta}_0+
\hat{u}_{0,j}\hat{*}\hat{s}]\\ \nonumber
\end{align}
and
\begin{align}\label{u1}
\hat{u}_1(k)&:=-\nu|k|^2\hat{u}_0-i k_j P_k[\hat{u}_{0,j}\hat{*}\hat{u}_0]+a P_k[e_2\hat{\Theta}_0]+\hat{f}\\ \nonumber
\hat{\Theta}_1(k)&:=-\mu|k|^2\hat{\Theta}_0-i k_j(\hat{u}_{0,j}\hat{*}\hat{\Theta}_0).
\end{align}

\noindent Similarly, for ($\ref{MB2}$), we define
\begin{align}\label{2-2.7}
\hat{g}^{[3]}_j&:=- P_k[\hat{v}_{0,j}\hat{*}\hat{w}+\hat{w}_j\hat{*}\hat{v}_0+\hat{w}_j\hat{*}\hat{w}]+\frac{1}{\mu \rho}P_k[\hat{B}_{0,j}\hat{*}\hat{q}+\hat{q}_j\hat{*}\hat{B}_0+\hat{q}_j\hat{*}\hat{q}]\\ \nonumber
\hat{g}^{[4]}_j&:=- P_k[\hat{v}_{0,j}\hat{*}\hat{q}+\hat{w}_j\hat{*}\hat{B}_0+\hat{w}_j\hat{*}\hat{q}]+ P_k[\hat{B}_{0,j}\hat{*}\hat{w}+\hat{q}_j\hat{*}\hat{v}_0+\hat{q}_j\hat{*}\hat{w}]\\ \nonumber
\end{align}
and
\begin{align}\label{v1}
\hat{v}_1(k)&:=-\nu|k|^2\hat{v}_0-i k_j P_k[\hat{v}_{0,j}\hat{*}\hat{v}_0-\frac{1}{\mu\rho}\hat{B}_{0,j}\hat{*}\hat{B}_0]+\hat{f}\\ \nonumber
\hat{B}_1(k)&:=-\frac{1}{\mu \sigma}|k|^2\hat{B}_0-i k_j P_k[\hat{v}_{0,j}\hat{*}\hat{B}_0-\hat{B}_{0,j}\hat{*}\hat{v}_0].
\end{align}
Using these definitions in (\ref{FB2}) and (\ref{MB2}) and integrating terms whose $\tau$ dependence appears only in the exponential, we obtain the integral equations
\begin{align}\label{IE-t-1}
\hat{h}(k,t)&=-ik_j\int_0^te^{-\nu|k|^2(t-s')}\left(\hat{g}^{[1]}_j(k,s')-P_k[ae_2\hat{s}](k,s')\right)ds'+\left(\frac{1-e^{-\nu|k|^2t}}{\nu|k|^2}\right)\hat{u}_1\\ \nonumber
\hat{s}(k,t)&=-ik_j\int_0^te^{-\mu|k|^2(t-s)}\hat{g}^{[2]}_j(k,s)ds+\left(\frac{1-e^{-\mu|k|^2t}}{\mu|k|^2}\right)\hat{\Theta}_1
\end{align}
and 
\begin{align}\label{IE-t-2}
\hat{w}(k,t)&=-ik_j\int_0^te^{-\nu|k|^2(t-s)}\hat{g}^{[3]}_j(k,s)ds+\left(\frac{1-e^{-\nu|k|^2t}}{\nu|k|^2}\right)\hat{v}_1\\ \nonumber
\hat{q}(k,t)&=-ik_j\int_0^te^{-(\mu\sigma)^{-1}|k|^2(t-s)}\hat{g}^{[4]}_j(k,s)ds+\left(\frac{1-e^{-(\mu\sigma)^{-1}|k|^2t}}{(\mu\sigma)^{-1}|k|^2}\right)\hat{B}_1.
\end{align}

In both systems, we seek a solution as a Laplace transform, 
\begin{align*} 
(\hat{h},\hat{s})(k,t)&=\int_0^{\infty}\left(\hat{H}, \hat{S}\right)(k,p)e^{-p/t}dp\\
(\hat{w}, \hat{q})(k,t)&=\int_0^{\infty}\left(\hat{W}, \hat{Q}\right)(k,p)e^{-p/t}dp.
\end{align*}
\noindent With this goal, we take the formal
\footnote[1]{While the derivation of the integral equation is formal, we prove later (Lemma \ref{2.9.}) that the unique solution to the integral equation in the Borel plane generates a solution to the Boussinesq/magnetic B$\acute{\textnormal{e}}$nard equation through Laplace transform.}  
inverse Laplace transform in $1/t$ of our two equations. The inverse Laplace transform of $f$ is given as usual by 
\begin{equation}\nonumber
[\mathcal{L}^{-1}f](p)=\frac{1}{2\pi i}\int_{c-i\infty}^{c+i\infty}f(s)e^{sp}ds,
\end{equation}
where $c$ is chosen so that for $\mathrm{Re} \,s\geq c$, $f$ is analytic and has suitable asymptotic decay. We define 
\begin{equation}\nonumber
\mathcal{H}^{(\nu)}(p,p',k):=\int_{p'/p}^1\left\{\frac{1}{2\pi i}\int_{c-i\infty}^{c+i\infty}\tau^{-1}exp[-\nu|k|^2\tau^{-1}(1-s)+(p-p's^{-1})\tau]d\tau\right\}ds.
\end{equation}
Then (\ref{IE-t-1}) becomes
\begin{align}\label{IE-p-1}
\hat{H}(k,p)&=-ik_j\int_0^p\mathcal{H}^{(\nu)}(p,p',k)\hat{G}^{[1]}_j(k,p')dp'+\int_0^p\mathcal{H}^{(\nu)}(p,p',k)P_k[ae_2\hat{S}](k,p)dp\\ \nonumber
&\qquad\qquad +\hat{u}_1(k)\mathcal{L}^{-1}\left(\frac{1-e^{-\nu|k|^2t}}{\nu|k|^2}\right)(p)\\ \nonumber
\hat{S}(k,p)&=-ik_j\int_0^p\mathcal{H}^{(\mu)}(p,p',k)\hat{G}^{[2]}_j(k,p')dp'+\hat{\Theta}_1(k)\mathcal{L}^{-1}\left(\frac{1-e^{-\mu|k|^2t}}{\mu|k|^2}\right)(p)
\end{align}
and (\ref{IE-t-2}) becomes
\begin{align}\label{IE-p-2}
\hat{W}(k,p)&=-ik_j\int_0^p\mathcal{H}^{(\nu)}(p,p',k)\hat{G}^{[3]}_j(k,p')dp'+\hat{v}_1(k)\mathcal{L}^{-1}\left(\frac{1-e^{-\nu|k|^2t}}{\nu|k|^2}\right)(p)\\ \nonumber
\hat{Q}(k,p)&=-ik_j\int_0^p\mathcal{H}^{(\mu\sigma)^{-1}}(p,p',k)\hat{G}^{[4]}_j(k,p')dp'+\hat{B}_1(k)\mathcal{L}^{-1}\left(\frac{1-e^{-(\mu\sigma)^{-1}|k|^2t}}{(\mu\sigma)^{-1}|k|^2}\right)(p).
\end{align}
\noindent In the above, $\hat{G}_j^{1,2,3,4}=\mathcal{L}^{-1}[g_j^{1,2,3,4}]$. Specifically,
\begin{align}\label{G_j}
 \hat{G}^{[1]}_j&=P_k[\hat u_{0,j}\hat{*}\hat H+\hat H_j\hat{*}\hat u_0+\hat H_j \, ^*_* \hat H],\\ \nonumber
 \hat{G}^{[2]}_j&=[\hat u_{0,j}\hat{*}\hat S+\hat H_j\hat{*}\hat{\Theta}_0+\hat H_j\, ^*_* \hat S],\\ \nonumber
 \hat{G}^{[3]}_j&=P_k[\hat v_{0,j}\hat{*}\hat W+\hat W_j\hat{*}\hat v_0+\hat W_j \, ^*_* \hat W]-\frac{1}{\mu \rho}P_k[\hat{B}_{0,j}\hat{*}\hat{Q}+\hat{Q}_j\hat{*}\hat{B}_0+\hat{Q}_j\, ^*_* \hat{Q}],\\ \nonumber
 \hat{G}^{[4]}_j&=P_k[\hat v_{0,j}\hat{*}\hat Q+\hat W_j\hat{*}\hat{B}_0+\hat W_j\, ^*_* \hat Q]-P_k[\hat{B}_{0,j}\hat{*}\hat{W}+\hat{Q}_j\hat{*}\hat{v}_0+\hat{Q}_j\, ^*_* \hat{W}]
\end{align}
where $\, ^* _*$ denotes the Laplace convolution followed by Fourier convolution (order is unimportant). 
We now make the observation that our kernel $\mathcal{H}^{(\nu)}(p,p',k)$ has a representation in terms of Bessel functions. Namely,
\begin{equation}\nonumber
\mathcal{H}^{(\nu)}(p,p',k)=\frac{\pi}{z}\mathcal{G}(z,z'):=\frac{\pi z'}{z}\left\{-J_1(z)Y_1(z')+Y_1(z)J_1(z')\right\}
\end{equation}
where $J_1$ and $Y_1$ are the Bessel functions of order 1, $z=2|k|\sqrt{\nu p}$, and $z'=2|k|\sqrt{\nu p'}$. In similar spirit, we have
\begin{equation}
\frac{2J_1(z)}{z}=\mathcal{L}^{-1}\left(\frac{1-e^{-\nu|k|^2\tau^{-1}}}{\nu|k|^2}\right)(p).
\end{equation}
These assertions are proved in the appendix in Lemma \ref{kernel} and Lemma \ref{U_0}.
Thus, our integral Boussinesq equation becomes 
\begin{align}\label{1-2.18}
\hat{H}(k,p)=&\frac{i k_j\pi}{2|k|\sqrt{\nu p}}\int_0^p \mathcal{G}(z,z')\hat{G}_j^{[1]}(k,p')dp'+a\pi\int_0^p \frac{\mathcal{G}(z,z')}{z}P_k[e_2\hat T(k,p')]dp'\\ \nonumber
&\qquad+2\hat u_1(k)\frac{J_1(z)}{z}\\ \nonumber
\hat{S}(k,p)=&\frac{i k_j\pi}{2|k|\sqrt{\mu p}}\int_0^p \mathcal{G}(\zeta ,\zeta ')\hat{G}_j^{[2]}(k,p')dp'+2\hat \Theta_1(k)\frac{J_1(\zeta)}{\zeta},
\end{align}
where $\zeta=2|k|\sqrt{\mu p}$, and $\zeta'=2|k|\sqrt{\mu p'}$. Abstractly, we may write the set of equations (\ref{1-2.18}) as
\begin{equation}\label{N}
(\hat{H},\hat{S})(k,p)=\mathcal{N}[(\hat{H},\hat{S})](k,p).
\end{equation}
Similarly, our integral MHD equation becomes
\begin{align}\label{2-2.18}
\hat{W}(k,p)=&\frac{i k_j\pi}{2|k|\sqrt{\nu p}}\int_0^p \mathcal{G}(\tilde{z},\tilde{z}')\hat{G}_j^{[3]}(k,p')dp'+2\hat v_1(k)\frac{J_1(\tilde{z})}{\tilde{z}}\\ \nonumber
\hat{Q}(k,p)=&\frac{i k_j\pi\sqrt{\mu \sigma}}{2|k|\sqrt{p}}\int_0^p \mathcal{G}(\tilde{\zeta} ,\tilde{\zeta} ')\hat{G}_j^{[4]}(k,p')dp'+2\hat B_1(k)\frac{J_1(\tilde{\zeta})}{\tilde{\zeta}},
\end{align}
where $\tilde{z}=2|k|\sqrt{\nu p}$, $\tilde{z}'=2|k|\sqrt{\nu p'}$, $\tilde{\zeta}=2|k|\sqrt{\frac{p}{\mu\sigma}}$, and $\tilde{\zeta}'=2|k|\sqrt{\frac{p'}{\mu\sigma}}$. Abstractly, we will denote the set of integral equations in (\ref{2-2.18}) as
\begin{equation}\label{M}
(\hat{W},\hat{Q})(k,p)=\mathcal{M}[(\hat{W},\hat{Q})](k,p).
\end{equation}

\begin{Remark} By properties of Bessel functions $|\mathcal{G}(z,z')|$ is bounded for all real nonnegative $z'\leq z$. (The approximate bound is $0.6$).
\end{Remark}

\begin{Remark} By properties of Bessel functions $|\mathcal{G}(z,z')/z|$ is bounded for all real nonnegative $z'\leq z$.
\end{Remark}

To prove Theorem \ref{existence} and \ref{existenceMHD}, we will show $\mathcal{N}$ and $\mathcal{M}$ are contractive in a suitable space, so $(\hat H,\hat S)$ and $ (\hat W,\hat Q)$ are Laplace transformable in $1/t$. Then Lemma \ref{kernel} tells us that $(\hat h,\hat s)$ and $(\hat w,\hat q)$ the Laplace transforms satisfy ($\ref{IE-t-1}$) and ($\ref{IE-t-2}$) for $\Re (1/t)$ large enough. 
This means that at least for small enough $t$,
\begin{equation}\nonumber
(\hat u,\hat{\Theta})(k,t)=(\hat u_0,\hat{\Theta}_0)+\int_0^{\infty}(\hat H,\hat{S})(k,p)e^{-p/t}dp
\end{equation}
\noindent solves the Boussinesq equation ($\ref{FB2}$) in the Fourier space with given initial condition and
\begin{equation}\nonumber
(\hat v,\hat{B})(k,t)=(\hat v_0,\hat{B}_0)+\int_0^{\infty}(\hat W,\hat Q)(k,p)e^{-p/t}dp
\end{equation}
\noindent solves the magnetic B$\acute{\textnormal{e}}$nard equation ($\ref{MB2}$) in the Fourier space with given initial condition. 
Furthermore, we show $(u, \Theta)(x,t)=\mathcal{F}^{-1}[(\hat{u},\hat{\Theta})(\cdot,t)](x)$ (respectively, $(v, B)(x,t)=\mathcal{F}^{-1}[(\hat{v},\hat{B})(\cdot,t)](x)$) is a classical solution to the Boussinesq (magnetic Bernard) problem. 

\section{Norms in $p$}\label{sec3}
Recall the norm $|| \cdot ||_{N}$ in $k$ is either the $(\gamma,\beta)$ norm given in Definition \ref{gamma,beta} for some $\beta \geq 0$ and $\gamma >d$ or the $L^1\cap L^{\infty}$ norm.
\begin{Definition}
For $\alpha \geq 1$, we define 
\begin{equation}\nonumber
||\hat f||^{(\alpha)}=\sup_{p\geq 0}(1+p^2)e^{-\alpha p}||\hat{f}(\cdot, p)||_{N}.
\end{equation}
\end{Definition}

\begin{Definition}
We define $\mathcal{A}^{\alpha}$ to be the Banach space of continuous function of $(k,p)$ for $k\in \mathbb{R}^d$ and $p\in \mathbb{R}^+$ for which $||\cdot ||^{\alpha}$ is finite. In similar spirit, we define the space $\mathcal{A}_1^{\alpha}$ of locally integrable functions for $p\in [0,L)$, and continuous in $k$ such that
\begin{equation}\nonumber
||\hat f||_1^{\alpha}=\int_0^Le^{-\alpha p}||\hat f(\cdot, p)||_{N}dp < \infty.
\end{equation}
\end{Definition}
\begin{Definition}
Finally, we also define $\mathcal{A}_L^{\alpha}$ to be the Banach space of continuous functions in $(k,p)$ for $k$ in $\mathbb{R}^d$ and $p\in [0,L]$ such that
\begin{equation}\nonumber
||\hat f||^{\infty}_L=\sup_{p\in[0,L]}||\hat f(\cdot,p)||_{N}<\infty.
\end{equation}
\end{Definition}
These norms are used in the analysis of the solutions to (\ref{1-2.18}) and (\ref{2-2.18}). The norms are used to guarantee the solutions have the properties necessary to insure their Laplace transforms satisfy the corresponding integral equations, (\ref{FB2}) and (\ref{MB2}). Furthermore, to show Borel summability for analytic data and forcing, more regularity in $p$ is required than provided by $||\cdot||_1^{\alpha}$. By proving the solution is unique in the spaces $\mathcal{A}_1^{\alpha}$ and $\mathcal{A}_L^{\alpha}$, where one clearly contains the other for finite $L$, we are assured of regularity in $p$.

\section{Existence of a Solution to (\ref{1-2.18}) and (\ref{2-2.18})}\label{sec4}
We need some preliminary lemmas. Recall, $d=2$ or $d=3$ denotes the dimension in $x$ or its dual $k$. Often constants appearing in subalgebra bounds will depend on dimension. We will explicitly state the dependence when defining them and suppress the dependence elsewhere.

\begin{Lemma}
If $||\hat{v}||_{\gamma, \beta}$ and $||\hat{w}||_{\gamma, \beta}<\infty$ for $\gamma >d$ and $k\in \mathbb{R}^d$, then
\begin{equation*}
||\hat{v}\hat{*}\hat{w}||_{\gamma,\beta}\leq \tilde{C}_0(d)||\hat{v}||_{\gamma, \beta}||\hat{w}||_{\gamma, \beta},
\end{equation*}
\noindent where 
\begin{align*}
\tilde{C}_0(2)=2^{\gamma+1}\int_{k'\in \mathbb{R}^2}\frac{1}{(1+|k'|)^{\gamma}}dk'=\frac{\pi 2^{\gamma +2}}{(\gamma -1)(\gamma -2)}  \textnormal{ and }\\
\tilde{C}_0(3)=2^{\gamma+1}\int_{k'\in \mathbb{R}^3}\frac{1}{(1+|k'|)^{\gamma}}dk'=\frac{\pi 2^{\gamma +4}}{(\gamma -1)(\gamma -2)(\gamma -3)}.
\end{align*} 
\end{Lemma}

\begin{proof} The $d=3$ {\bf case}
can be found in \cite{smalltime} and 
the $d=2$ {\bf case} is basically the same. For a detailed proof see \cite{Thesis}. From the definition of $||\cdot||_{\gamma,\beta}$ and the fact that $e^{-\beta(|k'|+|k-k'|)}\leq e^{-\beta|k|}$, we have 
\begin{equation}\nonumber
|\hat{v}\hat{*}\hat{w}|\leq e^{-\beta|k|}||\hat{v}||_{\gamma,\beta}||\hat{w}||_{\gamma,\beta}\int_{k'\in\mathbb{R}^2}\frac{1}{(1+|k'|)^{\gamma}(1+|k-k'|)^{\gamma}}dk'.
\end{equation}
\noindent Split the integral into two domains $|k'|\leq|k|/2$ and its compliment to show
\begin{align}\nonumber
\int_{k'\in\mathbb{R}^2}\frac{1}{(1+|k'|)^{\gamma}(1+|k-k'|)^{\gamma}}dk'&\leq \frac{2^{\gamma+1}}{(1+|k|)^{\gamma}}\int_{k'\in\mathbb{R}^2}\frac{1}{(1+|k'|)^{\gamma}}dk'\\ \nonumber
&=\frac{2^{\gamma+2}\pi}{(1+|k|)^{\gamma}(\gamma-1)(\gamma-2)},
\end{align} 
where polar coordinates and integration by parts are used to evaluate the last integral.
\end{proof}

\begin{Corollary}\label{2.2.} If $||\hat{v}||_{N}$, $||\hat{w}||_{N}<\infty$, then for $C_0=C_0(d)$ chosen such that $C_0=\tilde{C}_0 $ for $N = (\gamma, \beta) $, $\gamma>d$ and $C_0=1$ for $N=L^1\cap L^{\infty}$, we have
\begin{equation}\nonumber
||\hat{v}\hat{*}\hat{w}||_N\leq C_0||\hat{v}||_N||\hat{w}||_N.
\end{equation}
\end{Corollary}

\begin{Lemma}\label{2.3.} Also, notice that
\begin{equation}\nonumber
\left\|\left(P_k(\hat{f}), P_k(\hat{g})\right)\right\|_{N}\leq ||(\hat{f} , \hat{g})||_{N}
\end{equation}
\end{Lemma}
\begin{proof} $P_k$ is the projection of a vector onto $k^{\bot}$.
\end{proof}

\begin{Lemma}\label{2.4.} With $C_0$ as defined in Corollary \ref{2.2.}, appropriately modified for $d=2$ or $3$, and constants
\begin{align*}
C_2&=\frac{\pi C_0}{\min (\sqrt{\nu}, \sqrt{\mu})} \sup_{z\in\mathbb{R}^+,0\leq z'\leq z}|\mathcal{G}(z,z')|,\\ 
C_4&=2\pi\max (\frac{1}{\sqrt{\nu}}, 
\sqrt{\mu \sigma})\max (1,\frac{1}{\mu \rho}) C_0 \sup_{z\in\mathbb{R}^+,0\leq z'\leq z}|\mathcal{G}(z,z')|,\\ 
C_3&=\pi a\sup_{z\in\mathbb{R}^+,0\leq z'\leq z}|\mathcal{G}(z,z')/z|,
\end{align*} 
we have the following bounds on the norm in $k$, for operators $\mathcal{N}$ and $\mathcal{M}$ defined in (\ref{N}) and (\ref{M}) respectively: 
\begin{multline}\label{1-2.23}
||\mathcal{N}[(\hat H, \hat S)(\cdot, p)]||_{N} \leq  \frac{C_2}{\sqrt{p}}\int_0^p\left( ||(\hat H,\hat{S}) (\cdot, p')||_{N}*||(\hat H, \hat S)(\cdot,p')||_{N}\right. \hspace{2 in}\\ 
+\left. ||(\hat{u}_0, \hat \Theta_0)||_{N}||(\hat H, \hat S)(\cdot,p')||_{N} \right) dp'
+||(\hat u_1, \hat \Theta_1)||_{N}+C_3\int_0^p||\hat S(\cdot,p')||_{N}dp'
\end{multline}
\begin{multline} \label{2-2.23}
||\mathcal{M}[(\hat W, \hat Q)(\cdot, p)]||_{N} \leq  \frac{C_4}{\sqrt{p}}\int_0^p\left( ||(\hat W , \hat{Q})(\cdot, p')||_{N}*||(\hat W, \hat Q)(\cdot,p')||_{N}\right. \hspace{1.5 in}\\ 
+\left. ||(\hat{v}_0, \hat B_0)||_{N}||(\hat W, \hat Q)(\cdot,p')||_{N} \right) dp'
+||(\hat v_1, \hat B_1)||_{N}
\end{multline}
and
\begin{align}\label{1-2.24}
||\mathcal{N}[(\hat H^{[1]}, \hat S^{[1]})](\cdot, p)-\mathcal{N}[(\hat H^{[2]}, \hat S^{[2]})](\cdot, p)||_{N} \leq\qquad& \\ \nonumber
\frac{C_2}{\sqrt{p}}\int_0^p \left( ||(\hat H^{[1]},\hat{S}^{[1]}) (\cdot, p')||_{N}+||(\hat H^{[2]},\hat{S}^{[2]} )(\cdot, p')||_{N}\right)& *\left\|(\hat H^{[1]}, \hat S^{[1]})(\cdot,p')\right.\\ \nonumber
\left.-(\hat H^{[2]}, \hat S^{[2]})(\cdot,p')\right\|_{N} +||(\hat u_0,\hat \Theta_0)||_{N}||(\hat H^{[1]}, \hat S^{[1]})(\cdot,p')&-(\hat H^{[2]}, \hat S^{[2]})(\cdot,p')||_{N} dp'\\ \nonumber
+C_3\int_0^p||\hat S^{[1]}-\hat S^{[2]}(\cdot, p')||_{N}dp'&
\end{align}
\begin{align} \label{2-2.24}
||\mathcal{M}[(\hat W^{[1]}, \hat Q^{[1]})(\cdot, p)]-\mathcal{M}[(\hat W^{[2]}, \hat Q^{[2]})(\cdot, p)]||_{N} \leq \qquad&\\ \nonumber 
\frac{C_4}{\sqrt{p}}\int_0^p \left( ||(\hat W^{[1]},\hat{Q}^{[1]}) (\cdot, p')||_{N}+||(\hat W^{[2]} , \hat{Q}^{[2]})(\cdot, p')||_{N}\right)&*\left\|(\hat W^{[1]}, \hat Q^{[1]})(\cdot,p') \right. \\ \nonumber
\left.-(\hat W^{[2]}, \hat Q^{[2]})(\cdot,p')\right\|_{N} +||(\hat v_0,\hat B_0)||_{N}||(\hat W^{[1]}, \hat Q^{[1]})(\cdot,p')&-(\hat W^{[2]}, \hat Q^{[2]})(\cdot,p')||_{N} dp'. 
\end{align}
\end{Lemma}

\begin{proof} We will give the proof for (\ref{2-2.23}) and (\ref{2-2.24}). The two inequalities for $(\hat{H},\hat{S})$ are very similar. From \cite{handbook}, $|J_1(z)/z|\leq 1/2$ for $z\in \mathbb{R}^+$ and
\begin{equation}
\left\|2\left(\hat v_1(k)\frac{J_1(\tilde{z})}{\tilde{z}}, \hat B_1(k)\frac{J_1(\tilde{\zeta})}{\tilde{\zeta}}\right)\right\|_{N}\leq ||(\hat v_1, \hat B_1)||_{N}.
\end{equation}

\noindent From Corollary \ref{2.2.}, we have
\begin{multline}\nonumber
|||\hat v_{0}|\hat{*}(\hat W, \hat{Q})+|\hat W|\hat{*}(\hat v_0, \hat{B}_0)+|\hat W| \, ^*_* (\hat W,\hat{Q})||_{N}\leq \hspace{2 in}\\ \nonumber
\left[2C_0||(\hat{v}_0,\hat{B}_0)||_{N}||(\hat{W},\hat{Q})(\cdot, p)||_{N}+C_0||\hat{W}(\cdot, p)||_{N}*||(\hat{W},\hat{Q})(\cdot, p)||_{N}\right].
\end{multline}
\noindent Similarly,
\begin{multline}\nonumber
\left\||\hat{B}_{0}|\hat{*}\left(\frac{\hat{Q}}{\mu \rho}, \hat{W}\right)+|\hat{Q}|\hat{*}\left(\frac{\hat{B}_0}{\mu \rho}, \hat{v}_0\right)+ |\hat{Q}|\, ^*_* \left(\frac{\hat{Q}}{\mu \rho},\hat{W}\right)\right\|_{N}\leq \max\left(1,\frac{1}{\mu \rho}\right)\cdot \\ \nonumber
\left[2C_0||(\hat{v}_0,\hat{B}_0)||_{N}||(\hat{W},\hat{Q})(\cdot, p)||_{N}+C_0||\hat{Q}(\cdot, p)||_{N}*||(\hat{W},\hat{Q})(\cdot, p)||_{N}\right].
\end{multline}
\noindent Then using Lemma \ref{2.3.}, the two inequalities above, and Schwartz inequality we obtain
\begin{align*}
||k_j(\hat{G}_j^{[3]}, \hat{G}_j^{[4]})||_{N}\leq& 4C_0|k|\max \left(1,\frac{1}{\mu \rho}\right) \left(||(\hat W , \hat{Q}) (\cdot, p')||_{N}*||(\hat W, \hat Q)(\cdot,p')||_{N}\right.\\ \nonumber
&\qquad+\left. ||(\hat{v}_0, \hat B_0)||_{N}||(\hat W, \hat Q)(\cdot,p')||_{N} \right).
\end{align*}
\noindent Now, noticing that 
\begin{align*}
\left|k_j\left(\frac{\mathcal{G}(z,z')}{\sqrt{\nu}}\right.\right.\hat{G}^{[3]}_j,&\left.\left.\sqrt{\mu \sigma}\mathcal{G}(\zeta, \zeta ')\hat{G}^{[4]}_j\right)\right|\leq \\
&\max (\frac{1}{\sqrt{\nu}}, \sqrt{\mu \sigma})|k_j(\hat{G}^{[3]}_j, \hat{G}^{[4]}_j)|\sup_{z\in\mathbb{R}^+,0\leq z'\leq z}|\mathcal{G}(z,z')|
\end{align*}
\noindent (\ref{2-2.23}) follows directly.  To obtain (\ref{2-2.24}) notice that
\begin{align*}
\hat W_j^{[1]} \, ^*_*(\hat W^{[1]},\hat{Q}^{[1]})&-\hat W_j^{[2]} \, ^*_*(\hat W^{[2]},\hat{Q}^{[2]})=\\
&\hat W_j^{[1]} \, ^*_* \left((\hat W^{[1]}, \hat{Q}^{[1]})-(\hat W^{[2]},\hat{Q}^{[2]})\right)+(\hat W_j^{[1]}-\hat W_j^{[2]}) \, ^*_*(\hat W^{[2]},\hat{Q}^{[2]}).
\end{align*}
\noindent From which we get 
\begin{multline} \nonumber
\left\|\hat W_j^{[1]} \, ^*_*(\hat W^{[1]},\hat{Q}^{[1]})-\hat W_j^{[2]} \, ^*_*(\hat W^{[2]},\hat{Q}^{[2]})\right\|_{N}\leq C_0 \left\|(\hat W^{[1]},\hat{Q}^{[1]})-(\hat W^{[2]},\hat{Q}^{[2]})\right\|_{N}\\ \nonumber
\qquad *\left( ||(\hat W^{[1]},\hat{Q}^{[1]})||_{N} +||(\hat W^{[2]}, \hat{Q}^{[2]})||_{N} \right).
\end{multline}
\noindent Similarly,
\begin{multline} \nonumber
\left\|\hat Q_j^{[1]} \, ^*_*(\hat Q^{[1]},\hat{W}^{[1]})-\hat Q_j^{[2]} \, ^*_*(\hat Q^{[2]},\hat{W}^{[2]})\right\|_{N}\leq C_0 \left\|(\hat W^{[1]},\hat{Q}^{[1]})-(\hat W^{[2]},\hat{Q}^{[2]})\right\|_{N}\\ \nonumber
\qquad *\left(||(\hat W^{[1]},\hat{Q}^{[1]})||_{N} +||(\hat W^{[2]}, \hat{Q}^{[2]})||_{N}\right).
\end{multline}
\noindent Combining this bound and bounds using Lemma \ref{2.3.} as in the first part of the proof, we get (\ref{2-2.24}). 
\end{proof}

\begin{Lemma}\label{2.6.} For $\hat{f},\hat{g} \in \mathcal{A}^{\alpha}, \mathcal{A}^{\alpha}_1$ or $\mathcal{A}^{\infty}_L$
\begin{align}\nonumber
||\hat{f}\,^*_*\hat{g}||^{(\alpha)}&\leq M_0 C_0 ||\hat{f}||^{(\alpha)}||\hat{g}||^{(\alpha)}\\ \nonumber
||\hat{f}\,^*_*\hat{g}||^{(\alpha)}_1&\leq C_0||\hat{f}||^{(\alpha)}_1||\hat{g}||^{(\alpha)}_1\\ \nonumber
||\hat{f}\,^*_*\hat{g}||^{(\infty)}_L&\leq L C_0 ||\hat{f}||^{(\infty)}_L||\hat{g}||^{(\infty)}_L,
\end{align} 
\noindent where $M_0\approx 3.76\cdots$ is large enough so
\begin{equation}\nonumber
\int_0^p\frac{(1+p^2)ds}{(1+s^2)(1+(p-s)^2)}\leq M_0.
\end{equation}
\end{Lemma}
\noindent This means the Banach spaces listed in the norms section form subalgebras under the operation $\,^*_*$. The properties listed are independent of dimension except for a change in $C_0$ showing up due to the Fourier convolution. The proof is in \cite{smalltime}. The basic idea is that $k$ and $p$ act separately in the norm. So, we need only consider how the $p$ portion of the norm effects $\int_0^p u(p)v(p-s)ds$.

\begin{Lemma}\label{2.7.} (This lemma expands the bounds in Lemma \ref{2.4.} to bounds in $p$ in some of our other norms).
On $\mathcal{A}_1^{\alpha}$, the operators $\mathcal{M}$ and $\mathcal{N}$ satisfy the following inequalities 
\begin{multline}\label{1-2.30}
||\mathcal{N}(\hat H, \hat S)||_1^{\alpha}\leq C_2 \sqrt{\pi}\alpha^{-1/2}\left\{ (||(\hat H, \hat S)||_1^{\alpha})^2+||(\hat u_0, \hat{\Theta}_0)||_{N}||(\hat H, \hat S)||_1^{\alpha}\right\} \\ 
+\alpha ^{-1}||(\hat u_1, \hat \Theta_1)||_{N}+\alpha^{-1}C_3||\hat S||_1^{\alpha},
\end{multline}
\begin{multline} \label{2-2.30}
||\mathcal{M}(\hat W, \hat Q)||_1^{\alpha}\leq C_4 \sqrt{\pi}\alpha^{-1/2}\left\{ (||(\hat W, \hat Q)||_1^{\alpha})^2+||(\hat v_0, \hat{B}_0)||_{N}||(\hat W, \hat Q)||_1^{\alpha}\right\} \\
+\alpha ^{-1}||(\hat v_1, \hat B_1)||_{N},
\end{multline}
and
\begin{multline}\label{1-2.31}
||\mathcal{N}(\hat H^{[1]}, \hat S^{[1]})-\mathcal{N}(\hat H^{[2]}, \hat S^{[2]})||_1^{\alpha}\leq \hspace{2.6 in}\\
C_2 \sqrt{\pi}\alpha^{-1/2}\left\{ \left(||(\hat H^{[1]}, \hat S^{[1]})||_1^{\alpha}+||(\hat H^{[2]}, \hat S^{[2]})||_1^{\alpha}\right)\left(||(\hat H^{[1]}, \hat S^{[1]})-(\hat H^{[2]}, \hat S^{[2]})||_1^{\alpha}\right) \right. \\
+\left. ||(\hat u_0, \hat{\Theta}_0)||_{N}||(\hat H^{[1]}, \hat S^{[1]})-(\hat H^{[2]}, \hat S^{[2]})||_1^{\alpha} \right\} +\alpha^{-1}C_3||\hat S^{[1]}-\hat S^{[2]}||_1^{\alpha},\\
\end{multline}
\begin{multline} \label{2-2.31}
||\mathcal{M}(\hat W^{[1]}, \hat Q^{[1]})-\mathcal{M}(\hat W^{[2]}, \hat Q^{[2]})||_1^{\alpha}\leq \hspace{2.5 in}\\
C_4 \sqrt{\pi}\alpha^{-1/2}\left\{ \left(||(\hat W^{[1]}, \hat Q^{[1]})||_1^{\alpha}+||(\hat W^{[2]}, \hat Q^{[2]})||_1^{\alpha}\right)\left(||(\hat W^{[1]}, \hat Q^{[1]})-(\hat W^{[2]}, \hat Q^{[2]})||_1^{\alpha}\right) \right. \\
+\left. ||(\hat v_0, \hat{B}_0)||_{N}||(\hat W^{[1]}, \hat Q^{[1]})-(\hat W^{[2]}, \hat Q^{[2]})||_1^{\alpha} \right\}. 
\end{multline}
\noindent Similarly, for $\mathcal{A}_L^{\infty}$, we have
\begin{multline}\label{1-2.32}
||\mathcal{N}(\hat H, \hat S)||_L^{\infty}\leq C_2 \sqrt{L}\left\{ L(||(\hat H, \hat S)||_L^{\infty})^2+||(\hat u_0, \hat{\Theta}_0)||_{N}||(\hat H, \hat S)||_L^{\infty}\right\} \hspace{1.05 in}\\
+||(\hat u_1, \hat \Theta_1)||_{N}+LC_3||\hat S||_L^{\infty},
\end{multline}
\begin{multline} \label{2-2.32}
||\mathcal{M}(\hat W, \hat Q)||_L^{\infty}\leq C_4 \sqrt{L}\left\{ L(||(\hat W, \hat Q)||_L^{\infty})^2+||(\hat v_0, \hat{B}_0)||_{N}||(\hat W, \hat Q)||_L^{\infty}\right\} \hspace{.9 in}\\
+||(\hat v_1, \hat B_1)||_{N},
\end{multline}
and
\begin{multline} \label{1-2.33}
||\mathcal{N}(\hat H^{[1]}, \hat S^{[1]})-\mathcal{N}(\hat H^{[2]}, \hat S^{[2]})||_L^{\infty}\leq \hspace{3.15 in}\\
C_2 \sqrt{L}\left\{ L\left(||(\hat H^{[1]}, \hat S^{[1]})||_L^{\infty}+||(\hat H^{[2]}, \hat S^{[2]})||_L^{\infty}\right)\left(||(\hat H^{[1]}, \hat S^{[1]})-(\hat H^{[2]}, \hat S^{[2]})||_L^{\infty}\right) \right. \\
+\left. ||(\hat u_0, \hat{\Theta}_0)||_{N}||(\hat H^{[1]}, \hat S^{[1]})-(\hat H^{[2]}, \hat S^{[2]})||_L^{\infty} \right\} 
+ LC_3||\hat S^{[1]}-\hat S^{[2]}||_L^{\infty},
\end{multline}
\begin{multline} \label{2-2.33}
||\mathcal{M}(\hat W^{[1]}, \hat Q^{[1]})-\mathcal{M}(\hat W^{[2]}, \hat Q^{[2]})||_L^{\infty}\leq \hspace{3 in}\\
C_4 \sqrt{L}\left\{ L\left(||(\hat W^{[1]}, \hat Q^{[1]})||_L^{\infty}+||(\hat W^{[2]}, \hat Q^{[2]})||_L^{\infty}\right)\left(||(\hat W^{[1]}, \hat Q^{[1]})-(\hat W^{[2]}, \hat Q^{[2]})||_L^{\infty}\right) \right. \\
+\left. ||(\hat v_0, \hat{B}_0)||_{N}||(\hat W^{[1]}, \hat Q^{[1]})-(\hat W^{[2]}, \hat Q^{[2]})||_L^{\infty} \right\}. 
\end{multline}
\end{Lemma}

\begin{proof} For the space $\mathcal{A}_1^{\alpha}$ and any $L>0$, we note that
\begin{equation}\nonumber
\int_0^Le^{-\alpha p}||(\hat u_1,\hat \Theta_1)||_{N}dp\leq \alpha^{-1}||(\hat u_1,\hat \Theta_1)||_{N}
\end{equation}
\noindent and
\begin{equation}\nonumber
\int_0^L e^{-\alpha p}p^{-1/2}dp\leq \Gamma \left( \frac{1}{2} \right) \alpha^{-1/2}=\sqrt{\pi}\alpha^{-1/2}.
\end{equation}
\noindent We further notice that for $y(p')\geq 0$, we have
\begin{multline}
\int_0^L e^{-\alpha p}p^{-1/2}\left( \int_0^p y(p')dp'\right) dp=\int_0^L y(p')e^{-\alpha p'}\left( \int_{p'}^L e^{-\alpha (p-p')}p^{-1/2}dp\right) dp'\\ \label{2.34}
\leq \int_0^L y(p')e^{-\alpha p'}\left( \int_{0}^L e^{-\alpha s}s^{-1/2}ds\right) dp'\leq \int_0^L y(p')e^{-\alpha p'}\sqrt{\pi}\alpha^{-1/2}dp'.
\end{multline}
\noindent Similarly,
\begin{multline}\nonumber
\int_0^L e^{-\alpha p}\left( \int_0^p ||\hat S(\cdot,p')||_{N}dp'\right) dp=\int_0^L ||\hat S(\cdot,p')||_{N}e^{-\alpha p'}\left( \int_{p'}^L e^{-\alpha (p-p')}dp\right) dp'\\ \nonumber
=\int_0^L ||\hat S(\cdot,p')||_{N}e^{-\alpha p'}\left( \int_{0}^L e^{-\alpha s}ds\right) dp'\leq \alpha ^{-1}||\hat S||_1^{\alpha}.
\end{multline}
\noindent Then, using (\ref{2.34}) in ($\ref{1-2.23}$) and the idea in Lemma \ref{2.6.} that $\int_0^p e^{-\alpha p}[(||g||_{N}*||h||_{N})(p)]dp\leq ||g||_1^{\alpha}||h||_1^{\alpha}$, we have 
\begin{multline}
\int_0^Le^{-\alpha p}||\mathcal{N}(\hat H, \hat S)||_{N} dp\leq C_2 \sqrt{\pi}\alpha^{-1/2}\left\{ (||(\hat H, \hat S)||_1^{\alpha})^2\right. \hspace{1 in}\\
+\left. ||(\hat u_0, \hat{\Theta}_0)||_{N}||(\hat H, \hat S)||_1^{\alpha} \right\} +\alpha ^{-1}||(\hat u_1, \hat \Theta_1)||_{N}+\alpha^{-1}C_3||\hat S||_1^{\alpha}.
\end{multline}

\noindent This proves (\ref{1-2.30}). Further, from ($\ref{1-2.24}$), it also follows that
\begin{multline}
\int_0^L e^{-\alpha p}||\mathcal{N}(\hat H^{[1]}, \hat S^{[1]})(\cdot,p)-\mathcal{N}(\hat H^{[2]}, \hat S^{[2]})(\cdot,p)||_{N}dp\leq \hspace{1.5 in}\\
C_2 \sqrt{\pi}\alpha^{-1/2}\left\{ \left(||(\hat H^{[1]}, \hat S^{[1]})||_1^{\alpha}+||(\hat H^{[2]}, \hat S^{[2]})||_1^{\alpha}\right)\left\|(\hat H^{[1]}, \hat S^{[1]})-(\hat H^{[2]}, \hat S^{[2]})\right\|_1^{\alpha} \right. \\
+\left. ||(\hat u_0, \hat{\Theta}_0)||_{N}\left\|(\hat H^{[1]}, \hat S^{[1]})-(\hat H^{[2]}, \hat S^{[2]})\right\|_1^{\alpha} \right\} +\alpha^{-1}C_3||\hat S^{[1]}-\hat S^{[2]}||_1^{\alpha}.
\end{multline}

\noindent This proves (\ref{1-2.31}). The inequalities for $(\hat{W},\hat{Q})$ similarly follow from (\ref{2-2.23}) and (\ref{2-2.24}). 

Now, we consider $\mathcal{A}_L^{\infty}$. We note that for $p\in[0,L]$, we have
\begin{equation}\nonumber
\left| p^{-1/2}\int_0^p y(p')dp'\right| \leq \sup_{p\in [0,L]}|y(p)|\sqrt{L}.
\end{equation}
\noindent We recall from Lemma \ref{2.6.} that
\begin{equation}\nonumber
\left| \int_0^p y_1(s)y_2(p-s)ds\right| \leq L\left( \sup_{p\in [0,L]}|y_1(p)|\right) \left( \sup_{p\in [0,L]}|y_2(p)|\right). 
\end{equation}
\noindent Taking
\begin{align*}
&y(p)=||(\hat H, \hat S)(\cdot,p)||_{N}*||(\hat H, \hat S)(\cdot,p)||_{N}+||(\hat u_0, \hat \Theta_0)||_{N}||(\hat H, \hat S)(\cdot ,p)||_{N}\\
&\textnormal{and }y_1(p)=y_2(p)=||(\hat H, \hat S)(\cdot,p)||_{N},
\end{align*}

\noindent ($\ref{1-2.32}$) follows from ($\ref{1-2.23}$). To get the bound in ($\ref{1-2.33}$) we will choose,
\begin{multline}\nonumber
y(p)= \left(||(\hat H^{[1]}, \hat S^{[1]})||_{N}+||(\hat H^{[2]}, \hat S^{[2]})||_{N}\right)*\left\|(\hat H^{[1]}, \hat S^{[1]})-(\hat H^{[2]}, \hat S^{[2]})\right\|_{N} \\
+ ||(\hat u_0, \hat{\Theta}_0)||_{N}||(\hat H^{[1]}, \hat S^{[1]})-(\hat H^{[2]}, \hat S^{[2]})||_{N} 
\end{multline}
\begin{align}\nonumber
y_1(p)&=||(\hat H^{[1]}, \hat S^{[1]})||_{N}+||(\hat H^{[2]}, \hat S^{[2]})||_{N} \\ \nonumber
y_2(p)&=\left\|(\hat H^{[1]}, \hat S^{[1]})-(\hat H^{[2]}, \hat S^{[2]})\right\|_{N} 
\end{align}
\noindent now using ($\ref{1-2.24}$) the proof follows. The bounds on $(\hat{W},\hat{Q})$, (\ref{2-2.32}) and (\ref{2-2.33}), are proved in similar spirit.
\end{proof}

\begin{Lemma}\label{2.8.} Equation ($\ref{1-2.18}$) has a unique solution in $\mathcal{A}_1^{\omega}$ for any $L > 0$ in a ball of size $2\omega^{-1}||(\hat{u}_1,\hat \Theta_1)||_{N}$ for $\omega$ large enough to guarantee
\begin{equation}\label{1-2.39}
2C_2\sqrt{\pi}\omega^{-1/2}\left\{2\omega^{-1} ||(\hat u_1, \hat \Theta_1)||_{N}+||(\hat u_0,\hat \Theta_0)||_{N} +\frac{C_3}{C_2\sqrt{\pi}}\,\omega^{-1/2}\right\}<1
\end{equation}
where $(\hat{u}_1,\hat{\Theta}_1)$ is given in (\ref{u1}). Similarly, equation ($\ref{2-2.18}$) has a unique solution in $\mathcal{A}_1^{\alpha}$ for any $L > 0$ in a ball of size $2\alpha^{-1}||(\hat{v}_1,\hat B_1)||_{N}$ for $\alpha$ large enough to guarantee
\begin{equation}\label{2-2.39}
2C_4\sqrt{\pi}\alpha^{-1/2}\left\{2\alpha^{-1} ||(\hat v_1, \hat B_1)||_{N}+||(\hat v_0,\hat B_0)||_{N}\right\} <1
\end{equation}
where $(\hat{v}_1,\hat{B}_1)$ is given in (\ref{v1}). Furthermore, the solutions also belong to $\mathcal{A}_L^{\infty}$ for $L$ small enough to ensure either
\begin{equation}\label{1-2.40}
2C_2L^{1/2}\left\{ 2L||(\hat u_1, \hat \Theta_1)||_{N}+||(\hat u_0,\hat \Theta_0)||_{N}+\frac{C_3}{C_2}L^{1/2}\right\} <1
\end{equation}
or
\begin{equation}\label{2-2.40}
2C_4L^{1/2}\left\{ 2L||(\hat v_1, \hat B_1)||_{N}+||(\hat v_0,\hat B_0)||_{N}\right\} <1
\end{equation}
\noindent depending on the equation being considered. Moreover, $\lim_{p\rightarrow 0^+}(\hat H,\hat S)(k,p)=(\hat u_1, \hat \Theta_1)(k)$ and $\lim_{p\rightarrow 0^+}(\hat W, \hat Q)(k,p)=(\hat v_1, \hat B_1)(k)$. 
\end{Lemma}
\begin{proof} The estimates in Lemma \ref{2.7.} imply that $\mathcal{M}$ maps a ball of radius $2\alpha^{-1}||(\hat v_1, \hat B_1)||_{N}$ in $\mathcal{A}_1^{\alpha}$ into itself and is contractive when $\alpha$ is large enough to satisfy ($\ref{2-2.39}$). Similarly, $\mathcal{M}$ maps a ball of size $2||(\hat v_1, \hat B_1)||_{N}$ in $\mathcal{A}_L^{\infty}$ into itself and is contractive when $L$ is small enough to satisfy ($\ref{2-2.40}$). Therefore, there is a unique solution to the B$\acute{\textnormal{e}}$nard integral system of equations in the ball. Furthermore, $\mathcal{A}_L^{\infty}\subseteq \mathcal{A}_1^{\alpha}$, so the solutions are in fact one and the same. Similarly, $\mathcal{N}$ is contractive on a ball of radius $2\omega^{-1}||(\hat u_1, \hat \Theta_1)||_{N}$ in $\mathcal{A}_1^{\omega}$ for $\omega$ large enough to satisfy ($\ref{1-2.39}$) and a ball of size $2||(\hat u_1, \hat \Theta_1)||_{N}$ in $\mathcal{A}_L^{\infty}$ for $L$ small enough to satisfy ($\ref{1-2.40}$). So, the Boussinesq integral system has a unique solution in each of these spaces. Since $\mathcal{A}_L^{\infty}\subseteq \mathcal{A}_1^{\alpha}$, the solutions are in fact one and the same.

Moreover, applying (\ref{1-2.33}) (respectively, (\ref{2-2.33})) with $(\hat{H}^{[1]},\hat{S}^{[1]})=(\hat{H},\hat{S})$ (respectively, $(\hat{W}^{[1]},\hat{Q}^{[1]})=(\hat{W},\hat{Q})$) and $(\hat{H}^{[2]},\hat{S}^{[2]})=0=(\hat{W}^{[2]},\hat{Q}^{[2]})$, we obtain
\begin{multline}\nonumber
\left\|(\hat{H},\hat{S})(k,p)-\left(\hat{u}_1(k)\frac{2J_1(z)}{z},\hat{\Theta}_1(k)\frac{2J_1(\zeta)}{\zeta}\right)\right\|_L^{\infty}\leq\hspace{2 in}\\ \nonumber C_2L^{1/2}\left\{L(||(\hat{H},\hat{S})||_L^{\infty})^2+||(\hat{u}_0,\hat{\Theta}_0)||_{N}||(\hat{H},\hat{S})||_L^{\infty}\right\}+LC_3||\hat{S}||_L^{\infty}
\end{multline}
and
\begin{align*}
\left\|(\hat{W},\hat{Q})\right.(k,p)&\left.-\left(\hat{v}_1(k)\frac{2J_1(\tilde{z})}{\tilde{z}},\hat{B}_1(k)\frac{2J_1(\tilde{\zeta})}{\tilde{\zeta}}\right)\right\|_L^{\infty}\leq \\ &C_4L^{1/2}\left\{\right.L(||(\hat{W},\hat{Q})||_L^{\infty})^2+||(\hat{v}_0,\hat{B}_0)||_{N}||(\hat{W},\hat{Q})||_L^{\infty}\left.\right\}.
\end{align*}
\noindent Since $||(\hat{H},\hat{S})||_L^{\infty}$ and $||(\hat{W},\hat{Q})||_L^{\infty}$ are bounded for small $L$, letting $L\rightarrow 0$,
\begin{equation}\nonumber
\left\|(\hat{H},\hat{S})(k,p)-\left(\hat{u}_1(k)\frac{2J_1(z)}{z},\hat{\Theta}_1(k)\frac{2J_1(\zeta)}{\zeta}\right)\right\|_L^{\infty}\rightarrow 0
\end{equation}
and
\begin{equation}\nonumber
\left\|(\hat{W},\hat{Q})(k,p)-\left(\hat{v}_1(k)\frac{2J_1(\tilde{z})}{\tilde{z}},\hat{B}_1(k)\frac{2J_1(\tilde{\zeta})}{\tilde{\zeta}}\right)\right\|_L^{\infty}\rightarrow 0.
\end{equation}
As $\lim_{z\rightarrow 0}2J_1(z)/z=1$, for fixed $k$,  $\lim_{p\rightarrow 0}(\hat{H},\hat{S})(k,p)=(\hat{u}_1,\hat{\Theta}_1)(k)$. Similarly, for fixed $k$,  $\lim_{p\rightarrow  0}(\hat{W},\hat{Q})(k,p)=(\hat{v}_1,\hat{B}_1)(k)$.
\end{proof}

\section{Properties of the solutions}\label{sec5}

We have unique solutions to our two integral equations, (\ref{IE-p-1}) and (\ref{IE-p-2}). We show in the following Lemma \ref{2.9.} that these solutions Laplace transforms give solutions to (\ref{FB2}) and (\ref{MB2}), which are analytic in $t$ for $\Re \frac{1}{t}>\omega$ (resp. $\alpha$). Lemma \ref{equivalence} below shows that any solution of (\ref{FB2}) with $||(1+|\cdot|)^2(\hat{u},\hat{\Theta})(\cdot, t)||_{N}<\infty$ or respectively (\ref{MB2}) with $||(1+|\cdot|)^2(\hat{v},\hat{B})(\cdot, t)||_{N}<\infty$ is inverse Fourier transformable with $(u,\Theta)$ solving (\ref{B}) and $(v,B)$ solving (\ref{Ber}).  Lemma \ref{2.11.} below insures that $||(1+|\cdot|)^2(\hat{u},\hat{\Theta})(\cdot, t)||_{N}<\infty$ and $||(1+|\cdot|)^2(\hat{v},\hat{B})(\cdot, t)||_{N}<\infty$. Thus, combining these two results, we have $(u,\Theta)(x,t)=\mathcal{F}^{-1}(\hat{u},\hat{\Theta})(k,t)$ and $(v,B)(x,t)=\mathcal{F}^{-1}(\hat{v},\hat{B})(k,t)$ are classical solutions to (\ref{B}) and (\ref{Ber}) respectively. 
 
\begin{Lemma}\label{2.9.} For any solutions $(\hat{H}, \hat{S})$ and $(\hat{W},\hat{Q})$ of (\ref{IE-p-1}) and (\ref{IE-p-2}) such that $||(\hat{H},\hat{S})(\cdot,p)||_{N}\in L^1(e^{-\omega p}dp)$ and $||(\hat{W},\hat{Q})(\cdot,p)||_{N}\in L^1(e^{-\alpha p}dp)$ the Laplace transform
\begin{equation}
(\hat{u},\hat{\Theta})(k,t)=(\hat{u}_0,\hat{\Theta}_0)(k)+\int_0^{\infty}(\hat{H},\hat{S})(k,p)e^{-p/t}dp
\end{equation}
and
\begin{equation}
(\hat{v},\hat{B})(k,t)=(\hat{v}_0,\hat{B}_0)(k)+\int_0^{\infty}(\hat{W},\hat{Q})(k,p)e^{-p/t}dp
\end{equation}
solve (\ref{FB2}) for $\Re (1/t)>\omega$ and (\ref{MB2}) for $\Re (1/t) >\alpha$ respectively. Moreover, $(\hat{u},\hat{\Theta})(k,t)$ is analytic for $t\in(0,\omega^{-1})$ and $(\hat{v},\hat{B})(k,t)$ is analytic for $t\in(0,\alpha^{-1})$. 
\end{Lemma}
\begin{proof} We may write
\begin{equation}\nonumber
\mathcal{H}^{(\nu)}(p,p',k)=\int_{0}^1\left\{\frac{1}{2\pi i}\int_{c-i\infty}^{c+i\infty}\tau^{-1}exp[-\nu|k|^2\tau^{-1}(1-s)+(p-p's^{-1})\tau]d\tau\right\}ds
\end{equation}
since by contour deformation the integral with respect to $\tau$ can be pushed to $+\infty$ and is therefore zero for $s\in (0,p'/p)$. Let $\hat{G}_1=-ik_j\hat{G}_j^{[1]}+P_k(ae_2\hat{S})$ and $\hat{G}_l=-ik_j\hat{G}_j^{[l]}$ for $l=2, 3, 4$. Changing variable $p'/s\rightarrow p'$ and applying Fubini's theorem gives
\begin{align}\label{long time 4.35}
\int_0^p&\left(\mathcal{H}^{(\nu)}(p,p',k)\hat{G}_1(k,p'),\mathcal{H}^{(\mu)}(p,p',k)\hat{G}_2(k,p')\right)dp'\\ \nonumber
&=\int_0^1s\left\{\int_0^p\left(\hat{G}_1(k,p's)\mathcal{I}^{(\nu)}(p-p',s,k),\hat{G}_2(k,p's)\mathcal{I}^{(\mu)}(p-p',s,k)\right)dp'\right\}ds
\end{align}
and 
\begin{align}\label{long time 4.36}
\int_0^p&\left(\mathcal{H}^{(\nu)}(p,p',k)\hat{G}_3(k,p'),\mathcal{H}^{(\frac{1}{\mu\sigma})}(p,p',k)\hat{G}_4(k,p')\right)dp'\\ \nonumber
&=\int_0^1s\left\{\int_0^p\left(\hat{G}_3(k,p's)\mathcal{I}^{(\nu)}(p-p',s,k),\hat{G}_4(k,p's)\mathcal{I}^{(\frac{1}{\mu\sigma})}(p-p',s,k)\right)dp'\right\}ds,
\end{align}
where for $p>0$ 
\begin{equation}
\mathcal{I}^{(\nu)}(p,s,k)=\frac{1}{2\pi i}\int_{c-i\infty}^{c+i\infty} \tau^{-1}exp[-\nu|k|^2\tau^{-1}(1-s)+p\tau]d\tau.
\end{equation}
Taking the Laplace transform of (\ref{long time 4.35}) and (\ref{long time 4.36}) with respect to $p$ and again using Fubini's theorem yields
\begin{align}\nonumber
\int_0^{\infty}e^{-pt^{-1}}\left\{\int_0^1\int_0^p\left(\hat{G}_1(k,p's)\mathcal{I}^{(\nu)}(p-p',s,k),\hat{G}_2(k,p's)\mathcal{I}^{(\mu)}(p-p',s,k)\right)sdp'ds\right\}dp\\ \nonumber
=\int_0^1\left(\hat{g}_1(k,st)I^{(\nu)}(t,s,k),\hat{g}_2(k,st)I^{(\mu)}(t,s,k)\right)ds
\end{align}
and 
\begin{align}\nonumber
\int_0^{\infty}e^{-pt^{-1}}\left\{\int_0^1\int_0^p\left(\hat{G}_3(k,p's)\mathcal{I}^{(\nu)}(p-p',s,k),\hat{G}_4(k,p's)\mathcal{I}^{(\frac{1}{\mu\sigma})}(p-p',s,k)\right)sdp'ds\right\}dp\\ \nonumber
=\int_0^1\left(\hat{g}_3(k,st)I^{(\nu)}(t,s,k),\hat{g}_4(k,st)I^{(\frac{1}{\mu\sigma})}(t,s,k)\right)ds,
\end{align}
where $\hat{g}(k,t)=\mathcal{L}[\hat{G}(k,\cdot)](t^{-1})$ and $I(t,s,k)=\mathcal{L}[\mathcal{I}(\cdot,s,k)](t^{-1})$. By assumption $||(\hat{H},\hat{S})(\cdot,p)||_{N}\in L^1(e^{-\omega p}dp)$, $||(\hat{W},\hat{Q})(\cdot,p)||_{N}\in L^1(e^{-\alpha p}dp)$, $||(\hat{u}_0,\hat{\Theta}_0)||_{N}<\infty$, and $||(\hat{v}_0,\hat{B}_0)||_{N}<\infty$. From the definition of $\hat{G}^{[l]}_j$ given in (\ref{G_j}) and 
Lemma \ref{2.6.} it follows that $\hat{G}$ are Laplace transformable in p, for $t\in(0,\omega^{-1})$ or $t\in(0,\alpha^{-1})$ as appropriate. Thus,
\begin{align}\nonumber
\hat{g}_1&:=-i k_j P_k[\hat{h}_j\hat{*}\hat{h}+\hat{h}_j\hat{*}\hat{u}_0+
\hat{u}_{0,j}\hat{*}\hat{h}]+P_k[ae_2\hat{s}]\\ \nonumber
\hat{g}_2&:=-i k_j[\hat{h}_j\hat{*}\hat{s}+\hat{h}_j\hat{*}\hat{\Theta}_0+
\hat{u}_{0,j}\hat{*}\hat{s}]\\ \nonumber
\end{align}
while in similar spirit $(\hat{g}_3,\hat{g}_4)(k,t)$ is given by multiplying the right hand side of (\ref{2-2.7}) by $ik_j$. We also have
\begin{equation}
I^{(\nu)}(t,s,k)=te^{-\nu|k|^2t(1-s)}.
\end{equation}
Recalling the integral equations for $(\hat{H},\hat{S})$ and $(\hat{W},\hat{Q})$ given in (\ref{IE-p-1}) and (\ref{IE-p-2}), we have  
\begin{align}\nonumber
(\hat{h},\hat{s})(k,t)-&\left(\hat{u}_1(k)\left(\frac{1-e^{-\nu|k|^2t}}{\nu|k|^2}\right),\hat{\Theta}_1(k)\left(\frac{1-e^{-\mu|k|^2t}}{\mu|k|^2}\right)\right)\\ \nonumber
&=t\int_0^{1}\left(e^{-\nu|k|^2t(1-s)}\hat{g}_1(k,st),e^{-\nu|k|^2t(1-s)}\hat{g}_2(k,st)\right)ds\\ \nonumber
&=\int_0^{t}\left(e^{-\nu|k|^2(t-s)}\hat{g}_1(k,s),e^{-\nu|k|^2(t-s)}\hat{g}_2(k,s)\right)ds\\ \nonumber
\end{align}
and 
\begin{align}\nonumber
(\hat{w},\hat{q})(k,t)=\int_0^{t}&\left(e^{-\nu|k|^2(t-s)}\hat{g}_3(k,s),e^{-\nu|k|^2(t-s)}\hat{g}_4(k,s)\right)ds\\ \nonumber
&+\left(\hat{v}_1(k)\left(\frac{1-e^{-\nu|k|^2t}}{\nu|k|^2}\right),\hat{B}_1(k)\left(\frac{1-e^{-(\mu\sigma)^{-1}|k|^2t}}{(\mu\sigma)^{-1}|k|^2}\right)\right).
\end{align}
Therefore, we directly verify $(\hat{u},\hat{\Theta})(k,t)=(\hat{u}_0,\hat{\Theta}_0)(k)+(\hat{h},\hat{s})(k,t)$ satisfies (\ref{FB2}) and $(\hat{v},\hat{B})(k,t)=(\hat{v}_0,\hat{B}_0)(k)+(\hat{w},\hat{q})(k,t)$ satisfies (\ref{MB2}). Moreover, analyticity in $t$ follows from the representations
\begin{align*} 
(\hat{u},\hat{\Theta})(k,t)&=(\hat{u}_0,\hat{\Theta}_0)(k)+\int_0^{\infty}\left(\hat{H}, \hat{S}\right)(k,p)e^{-p/t}dp\\
(\hat{v},\hat{B})(k,t)&=(\hat{v}_0,\hat{B}_0)(k)+\int_0^{\infty}\left(\hat{W}, \hat{Q}\right)(k,p)e^{-p/t}dp.
\end{align*}
\end{proof}

\begin{Lemma}\label{2.11.} (Instantaneous smoothing) Assume $||(\hat u_0, \hat{\Theta}_0)||_{N}<\infty $, $||(\hat v_0, \hat{B}_0)||_{N}<\infty $, and $|| \hat f||_{N}<\infty$ with $N$ either $L^1\cap L^{\infty}(\mathbb{R}^d)$ or $(\gamma, \beta)$ with $\gamma >d$, $\beta \geq 0$. For the solution $(\hat v, \hat B)$ known to exist by Lemma \ref{2.8.} for $t\in(0,T]$ with $T<\alpha^{-1}$, we have $||(1+|\cdot|)^2(\hat{v},\hat{B})(\cdot, t)||_{N}<\infty$ for $t\in(0,T]$. Respectively, $||(1+|\cdot|)^2(\hat{u},\hat{\Theta})(\cdot, t)||_{N}<\infty$ for $t\in(0,T]$ with $T<\omega^{-1}$.
\end{Lemma}

\begin{proof} The two cases are similar, we present the B$\acute{\textnormal{e}}$nard case. Our goal is to boot strap up using derivatives of $(v, B)$. Consider the time interval $[\epsilon,T]$ for $\epsilon> 0$ and $T<\alpha^{-1}$. Define 
\begin{equation}\nonumber
\hat{V}_\epsilon(k)=\sup_{\epsilon\leq t\leq T}|(\hat v, \hat B)|(k,t).
\end{equation}
Since $|(\hat{v}, \hat{B})(k,t)|\leq |(\hat{v}_0, \hat{B}_0)(k)|+\int_0^{\infty}|(\hat{W}, \hat{Q})(k,p)|e^{-\alpha p}dp$,
\begin{equation}\nonumber
||\hat{V}_\epsilon(k)||_{N}\leq ||(\hat{v}_0, \hat{B}_0)(k)||_{N}+||(\hat{W}, \hat{Q})(k,p)||_{1}^{\alpha}<\infty.
\end{equation}
On $[\epsilon,T]$ for $\epsilon >0$, 
\begin{align}\nonumber 
\hat{v}(k,t)&=e^{-\nu|k|^2t}\hat v_0(k)- \int_0^t e^{-\nu|k|^2(t-\tau)}\left(i k_j P_k[\hat v_j\hat *\hat v-\frac{1}{\mu \rho}\hat B_j\hat *\hat B](k,\tau)-\hat f(k)\right)d\tau\\ \nonumber
\hat B(k,t)&=e^{\frac{-|k|^2t}{\mu \sigma}}\hat B_0(k)-i k_j \int_0^t e^{\frac{-|k|^2(t-\tau)}{\mu \sigma}}\left\{ P_k[\hat v_j\hat *\hat B+\hat B_j\hat *\hat v](k,\tau)\right\}d\tau.
\end{align}

\noindent Therefore,
\begin{align*}\nonumber 
|k||(\hat{v}, \hat{B})(k, t)|\leq \left|(\hat{v}_0, \hat{B}_0)(k)\right|\sqrt{\frac{1}{\epsilon\min(\nu ,\frac{1}{\mu \sigma})}}\sup_{z\geq 0}ze^{-z^2}+|\hat{f}|\int_0^t|k|e^{-\min(\nu ,\frac{1}{\mu \sigma})|k|^2(t-\tau)}d\tau\\
+2\hat{V}_{0}\hat{*}\hat{V}_{0}\int_0^t |k|^2e^{-\min(\nu ,\frac{1}{\mu \sigma})|k|^2(t-\tau)}d\tau.
\end{align*}
Noticing that 
\begin{equation}\nonumber
\int_0^t |k|^2e^{-\min(\nu ,\frac{1}{\mu \sigma})|k|^2(t-\tau)}d\tau\leq \frac{1}{\min(\nu ,\frac{1}{\mu \sigma})}
\end{equation}
and 
\begin{equation}\nonumber
\int_0^t|k|e^{-\min(\nu ,\frac{1}{\mu \sigma})|k|^2(t-\tau)}d\tau\leq\sup_{z\geq 0}\frac{1-e^{-z}}{\sqrt{z}}\frac{\sqrt{T}}{\min(\nu ,\frac{1}{\mu \sigma})},
\end{equation}
it follows that 
\begin{equation}\nonumber
\left\| |k|\hat{V}_{\epsilon/2}\right\|_{N}\leq\frac{C}{\epsilon^{1/2}}||(\hat{v}_0, \hat{B}_0)||_{N}+\frac{1}{\min(\nu ,\frac{1}{\mu \sigma})}\left(2C_0||\hat{V}_0||_{N}^2+C\sqrt{T}||\hat{f}||_{N}\right)<\infty.
\end{equation}
In the same spirit, for $t\in[\frac{\epsilon}{2},T]$, we have
\begin{align}\nonumber 
\hat{v}(k,t)&=e^{-\nu|k|^2t}\hat v(k,\epsilon/2)- \int_{\epsilon/2}^t e^{-\nu|k|^2(t-\tau)}\left(iP_k(\hat v_j\hat *[k_j\hat v]-\frac{1}{\mu \rho}\hat B_j\hat *[k_j \hat B])(k,\tau)-\hat f(k)\right)d\tau\\\nonumber
\hat B(k,t)&=e^{\frac{-|k|^2t}{\mu \sigma}}\hat B(k,\epsilon/2)-i \int_{\epsilon/2}^t e^{\frac{-|k|^2(t-\tau)}{\mu \sigma}}\left\{ P_k(\hat v_j\hat *[k_j\hat{B}]+\hat B_j\hat *[k_j \hat v])(k,\tau)\right\}d\tau,
\end{align}
where we used the divergence free conditions $k\cdot\hat{v}=0$ and $k\cdot\hat{B}=0$. Multiplying by $|k|^2$ and using our previous bounds, we have for $t\in[\epsilon,T]$
\begin{align}\nonumber 
|k|^2|(\hat{v}, \hat{B})(k, t)|\leq &\left|(\hat{v}, \hat{B})(k,\epsilon/2)\right|\frac{1}{(t-\epsilon/2)\min(\nu,\frac{1}{\mu\sigma})}\sup_{z\geq0}ze^{-z}\\ \nonumber
&+(2\hat{V}_{\epsilon/2}\hat{*}|k|\hat{V}_{\epsilon/2}+|\hat{f}|) \int_{\epsilon/2}^t |k|^2e^{-\min(\nu ,\frac{1}{\mu \sigma})|k|^2(t-\tau)}d\tau
\end{align}
Hence, 
\begin{equation}\nonumber
\left\| |k|^2\hat{V}_{\epsilon}\right\|_{N}\leq\frac{C}{\epsilon}||(\hat{v}_0, {B}_0)||_{N}+\frac{1}{ \min(\nu ,\frac{1}{\mu \sigma})}\left(2C_0\left\|\hat{V}_{\epsilon/2}\right\|_{N}\left\||k|\hat{V}_{\epsilon/2}\right\|_{N}+||\hat{f}||_{N}\right).
\end{equation}
All the terms on the right hand side are bounded, which gives $||(1+|k|)^2\hat{V}_{\epsilon}||_{N}<\infty$. Further, as $\epsilon>0$ is arbitrary, it follows that $||(1+|\cdot|)^2(\hat{v}, \hat{B})(\cdot,t)||_{N}<\infty$ for $t\in(0,T]$.
\end{proof}
\begin{Remark}\label{remsmooth}
We note that the smoothness argument in $x$ 
of the previous Lemma
can be easily extended further 
to show $\left \| (1+|k|)^4 {\hat V}_{\epsilon} 
\right \|_N$ is finite provided $\| (1+|k|^2) {\hat f} \|_N$,
is finite. Since
$\epsilon > 0$ is arbitrary this implies instantaneous smoothing
two orders more than the forcing.
\end{Remark}  

\begin{Lemma}\label{equivalence}
Given $(\hat{u},\hat{\Theta})$ a solution to (\ref{FB2}) such that $||(1+|\cdot|)^2(\hat{u},\hat{\Theta})(\cdot, t)||_N<\infty$ for $t\in (0,\omega^{-1})$, then $(u,\Theta)\in L^{\infty}[0,\omega^{-1},H^2(\mathbb{R}^d)]$ solves (\ref{B}). Respectively, given $(\hat{v},\hat{B})$ a solution to (\ref{MB2}) such that $||(1+|\cdot|)^2(\hat{v},\hat{B})(\cdot, t)||_N<\infty$ for $t\in(0,\alpha^{-1})$, then $(v,B)\in L^{\infty}[0,\alpha^{-1},H^2(\mathbb{R}^d)]$ solves (\ref{Ber}).
\end{Lemma}

\begin{proof}
The two cases are similar, we show the Boussinesq case. Suppose $(\hat{u},\hat{\Theta})$ a solution to (\ref{FB2}) such that $||(1+|\cdot|)^2(\hat{u},\hat{\Theta})(\cdot, t)||_N<\infty$ for $t\in (0,\omega^{-1})$. We notice that by our choice of norms, $(1+|\cdot|)^2(\hat{u},\hat{\Theta})(\cdot,t)\in L^2(\mathbb{R}^d)$ for any $t\in (0,\omega^{-1})$. Indeed for $N=(\gamma, \beta)$, we have
\begin{align*}
\left(\int(1+|k|)^4|(\hat{u},\hat{\Theta})(k,t)|^2dk\right)&^{1/2}\leq \\
&||(1+|\cdot|)^2|(\hat{u},\hat{\Theta})(\cdot,t)||_{\gamma, \beta}\left(\int\frac{e^{-2\beta|k|}}{(1+|k|)^{2\gamma}}dk\right)^{1/2}.
\end{align*}
As $\gamma>d$, $\int\frac{1}{(1+|k|)^{2\gamma}}e^{-2\beta|k|}dk<\infty$. For $N=L^1\cap L^{\infty}$ we have,
\begin{equation*}
\int(1+|k|)^4|(\hat{u},\hat{\Theta})(k,t)|^2dk\leq\int(1+|k|)^2|(\hat{u},\hat{\Theta})(k,t)|dk\sup_{k\in \mathbb{R}^d}(1+|k|)^2|(\hat{u},\hat{\Theta})(k,t)|.
\end{equation*}
So, $||(1+|\cdot|)^2(\hat{u},\hat{\Theta})(\cdot,t)||_{L^2(\mathbb{R}^d)}\leq ||(1+|\cdot|)^2(\hat{u},\hat{\Theta})(\cdot,t)||_{L^1\cap L^{\infty}(\mathbb{R}^d)}$. Thus, by well known properties of the Fourier transform $(u,\Theta)=\mathcal{F}^{-1}(\hat{u},\hat{\Theta})(x,t)\in L^{\infty}(0,\omega^{-1},H^2(\mathbb{R}^d))$. As $(\hat{u},\hat{\Theta})$ solves (\ref{FB2}), $(\hat{u},\hat{\Theta})$ is differentiable almost everywhere and 
\begin{align}
\hat{u}_t+\nu |k|^2\hat{u}&=-i k_j P_k[\hat{u}_j\hat{*}\hat{u}]+a P_k[e_2\hat{\Theta}]+\hat{f}\\ \nonumber
\hat{\Theta}_t+\mu |k|^2\hat{\Theta}&=-i k_j[\hat{u}_j\hat{*}\hat{\Theta}], \quad k\in \mathbb{R}^d \quad t\in \mathbb{R}^{+}.
\end{align}
Further, $(\hat{u}_t,\hat{\Theta}_t)(k,t)\in L^{\infty}(0,\omega^{-1},L^2(\mathbb{R}^d))$ since $(1+|k|)^2(\hat{u},\hat{\Theta})(k,t)\in L^{\infty}(0,\omega^{-1},L^2(\mathbb{R}^d))$. Hence, $(u,\Theta)(x,t)=\mathcal{F}^{-1}(\hat{u},\hat{\Theta})(x,t)$ solves 
\begin{align*}
u_t-\nu\Delta u&=-P[ u\cdot \nabla u - a e_2\Theta] +f(x)\\ \nonumber
\Theta_t-\mu \Delta \Theta&= -u\cdot \nabla \Theta.
\end{align*}
\end{proof}

\noindent {\bf Proof of Theorems \ref{existence}
{\bf and} \ref{existenceMHD}:} 
Suppose $||(1+|\cdot|)^2(\hat{u}_0, \hat{\Theta}_0)||_{N}<\infty$ and $||\hat{f}||_{N}<\infty$. Then from the definition of $(\hat u_1, \hat \Theta_1)$ in $(\ref{u1})$ we see $||(\hat{u}_1,\hat{\Theta}_1)||_{N}<\infty$, since
\begin{align*}
||(\hat u_1,\hat \Theta_1)||_{N}\leq \max(\nu, \mu) \left\||k|^2(\hat u_0, \hat{\Theta}_0)\right\|_{N}&+ C_0||\hat u_0||_{N}\left\||k|(\hat u_0,\hat \Theta_0)\right\|_{N}\\
&+ a||\hat \Theta_0||_{N}+||\hat f||_{N}.
\end{align*}
Therefore, when $\omega$ is large enough to ensures ($\ref{1-2.39}$), Lemma \ref{2.8.} gives $(\hat H, \hat S)(k, \cdot)$ is in $L^1(e^{-\omega p}dp)$. Applying Lemma \ref{2.9.}, we know for $t$ such that $\Re \frac{1}{t}>\omega$, $(\hat{H},\hat{S})(k,p)$ is Laplace transformable in $1/t$ with $(\hat{u},\hat{\Theta})(k,t)=(\hat{u}_0,\hat{\Theta}_0)(k)+(\hat{h},\hat{s})(k,t)$
satisfying Boussinesq equation in the Fourier space, (\ref{FB2}). Since $||(\hat H, \hat S)(\cdot, p)||_{N}<\infty$, we have $||(\hat u, \hat \Theta)(\cdot, t)||_{N}<\infty$ if $\Re \frac{1}{t}>\omega$, and i) is proved. Moreover, Lemma \ref{2.9.} shows that $(\hat{u},\hat{\Theta})$ is analytic for $\Re \frac{1}{t}>\omega$ and has the representation
\begin{equation}
(\hat{u}, \hat{\Theta})(k,t)=(\hat{u}_0, \hat{\Theta}_0)(x)+\int_0^{\infty}(\hat{H}, \hat{S})(k,p)e^{-p/t}dp
\end{equation}
proving ii). For iii), Lemma \ref{2.11.} shows that $||(1+|\cdot|)^2(\hat{u}, \hat{\Theta})(\cdot,t)||_{N} < \infty$ for $t\in[0,\omega^{-1})$ while Lemma \ref{equivalence} shows that $(u, \Theta)(x,t)\in L^{\infty}(0,T,H^2(\mathbb{R}^d))$ solves (\ref{B}). Moreover, $(u, \Theta)(x,t)$ is the unique solution to (\ref{B}) in $L^{\infty}(0,T,H^2(\mathbb{R}^d))$ as classical solutions are known to be unique, \cite{Temam}. Finally, suppose $(\hat H, \hat S)(k, \cdot)$ is in $L^1(e^{-\omega p}dp)$ for any $\omega >0$. By Lemma \ref{2.9.}, we know for any $t>0$, $(\hat{H},\hat{S})(k,p)$ is Laplace transformable with $(\hat{u},\hat{\Theta})(k,t)=(\hat{u}_0,\hat{\Theta}_0)(k)+(\hat{h},\hat{s})(k,t)$
satisfying Boussinesq equation in the Fourier space, (\ref{FB2}). Further, appealing to instantaneous smoothing Lemma \ref{2.11.} the solution is smooth. Thus, if $(\hat H, \hat S)(k, \cdot)$ is in $L^1(e^{-\omega p}dp)$ for any $\omega >0$, then a smooth global solution exists and iv) is proved. This shows the Boussinesq existence theorem. 
The proof of Theorem \ref{existenceMHD} is very similar.

\section{Borel-Summability}\label{sec6}

\indent We now show Borel-summability of the solutions guaranteed by Theorem \ref{existence} and Theorem \ref{existenceMHD} for $\beta >0$. This requires us to show that the solutions $(\hat{H}, \hat{S})(k,p)$ and $(\hat{W}, \hat{Q})(k,p)$ to the Boussinesq and MHD equations, respectively, are analytic in $p$ for $p\in \left\{0\right\}\cup\mathbb{R}^+$. First, we will seek a solution which is a power series

\begin{eqnarray}\label{1-3.49}
(\hat{H}, \hat{S})(k,p)-(\hat{u}_1,\hat{\Theta}_1)(k)=\sum_{l=1}^{\infty}(\hat{H}^{[l]}, \hat{S}^{[l]})(k)p^l\\ \label{2-3.49}
(\hat{W}, \hat{Q})(k,p)-(\hat{v}_1,\hat{B}_1)(k)=\sum_{l=1}^{\infty}(\hat{W}^{[l]}, \hat{Q}^{[l]})(k)p^l.
\end{eqnarray} 

\begin{Remark}
We will use induction to bound the successive terms of the power series. Many of these bounds have constants depending on the dimension in $k$ as before. For brevity of notation the dependence on dimension is suppressed after introducing the constants.
\end{Remark}

For the purpose of finding power series solutions, (\ref{1-2.18}) and (\ref{2-2.18}) are not good representations of the equations. By construction, $\frac{\pi}{z} \mathcal{G}(z,z')$ satisfies $[p\partial_{pp}+2\partial_p+\nu|k|^2]y=0$ with $\frac{\pi}{z} \mathcal{G}(z,z')\rightarrow 0$ and $\partial_p\left(\frac{\pi}{z} \mathcal{G}(z,z')\right)\rightarrow \frac{1}{p}$ as $p'$ approaches $p$ from below. Hence, we have the equivalent equations
\begin{align}\label{1-2.9}
[p\partial_{pp}+2\partial_p+\nu|k|^2]\hat{H}&=ik_j\hat{G}_j^{[1]}+aP_k[\hat{e}_2\hat{S}]\\ \nonumber
[p\partial_{pp}+2\partial_p+\nu|k|^2]\hat{S}&=ik_j\hat{G}_j^{[2]}
\end{align}
and
\begin{align}\label{2-2.9}
[p\partial_{pp}+2\partial_p+\nu|k|^2]\hat{W}&=ik_j\hat{G}_j^{[3]}\\ \nonumber
[p\partial_{pp}+2\partial_p+\nu|k|^2]\hat{Q}&=ik_j\hat{G}_j^{[4]}.
\end{align}
\noindent We substitute ($\ref{1-3.49}$) into ($\ref{1-2.9}$) and ($\ref{2-3.49}$) into ($\ref{2-2.9}$) and identify powers of $p^l$ to get a relationship for the coefficients. We will use that $1*p^l=p^{l+1}/(l+1)$. We will also use the fact that
\begin{equation}\nonumber
p^l*p^n=\frac{l!n!}{(l+n+1)!}p^{l+n+1}.
\end{equation}
For $l=0$, we have
\begin{align}\label{1-3.50}
2\hat{H}^{[1]}&=-i k_j P_k[\hat{u}_{1,j}\hat{*}\hat{u}_0+\hat{u}_{0,j}\hat{*}\hat{u}_1]-\nu |k|^2\hat{u}_1+P_k[a e_2\hat{\Theta}_1]\\ \nonumber
2\hat{S}^{[1]}&=-i k_j[\hat{u}_{1,j}\hat{*}\hat{\Theta}_0+\hat{u}_{0,j}\hat{*}\hat{\Theta}_1]-\mu |k|^2\hat{\Theta}_1
\end{align}
and
\begin{align}\label{2-3.50}
2\hat{W}^{[1]}&=-i k_j P_k\left[\hat{v}_{1,j}\hat{*}\hat{v}_0+\hat{v}_{0,j}\hat{*}\hat{v}_1-\left(\frac{\hat{B}_{0,j}}{\mu \rho}\hat{*}\hat{B}_1+\frac{\hat{B}_{1,j}}{\mu \rho}\hat{*}\hat{B}_0\right)\right]-\nu |k|^2\hat{v}_1\\ \nonumber
2\hat{Q}^{[1]}&=-i k_j P_k\left[\hat{v}_{1,j}\hat{*}\hat{B}_0+\hat{v}_{0,j}\hat{*}\hat{B}_1-(\hat{B}_{1,j}\hat{*}\hat{v}_0+\hat{B}_{0,j}\hat{*}\hat{v}_1)\right]-\frac{1}{\mu \sigma} |k|^2\hat{B}_1.
\end{align}
For $l=1$, we have
\begin{align}\label{1-3.51}
6\hat{H}^{[2]}+\nu |k|^2\hat{H}^{[1]}&=-i k_j P_k[\hat{H}^{[1]}_j\hat{*}\hat{u}_0+\hat{u}_{0,j}\hat{*}\hat{H}^{[1]}+\hat{u}_{1,j}\hat{*}\hat{u}_1]+P_k[a e_2\hat{S}^{[1]}] \\ \nonumber
6\hat{S}^{[2]}+\mu |k|^2\hat{S}^{[1]}&=-i k_j[\hat{S}^{[1]}_j\hat{*}\hat{\Theta}_0+\hat{u}_{0,j}\hat{*}\hat{S}^{[1]}+\hat{u}_{1,j}\hat{*}\hat{\Theta}_1]
\end{align}
and
\begin{align}\label{2-3.51}
6\hat{W}^{[2]}+\nu |k|^2\hat{W}^{[1]}&=-i k_j P_k[\hat{W}^{[1]}_j\hat{*}\hat{v}_0+\hat{v}_{0,j}\hat{*}\hat{W}^{[1]}+\hat{v}_{1,j}\hat{*}\hat{v}_1] \\ \nonumber
 &+\frac{i k_j}{\mu \rho} P_k[\hat{Q}^{[1]}_j\hat{*}\hat{B}_0+\hat{B}_{0,j}\hat{*}\hat{Q}^{[1]}+\hat{B}_{1,j}\hat{*}\hat{B}_1] \\ \nonumber
6\hat{Q}^{[2]}+\frac{1}{\mu \sigma} |k|^2\hat{Q}^{[1]}&=-i k_j P_k[\hat{W}^{[1]}_j\hat{*}\hat{B}_0+\hat{v}_{0,j}\hat{*}\hat{Q}^{[1]}+\hat{v}_{1,j}\hat{*}\hat{B}_1]\\ \nonumber
&+i k_j P_k[\hat{Q}^{[1]}_j\hat{*}\hat{v}_0+\hat{B}_{0,j}\hat{*}\hat{W}^{[1]}+\hat{B}_{1,j}\hat{*}\hat{v}_1].
\end{align}
\noindent More generally, for $l\geq 2$ in the Boussinesq case, we have
\begin{align}\label{1-3.53}
(l+1)(l+2)\hat{H}^{[l+1]}&=-\nu|k|^2\hat{H}^{[l]}-i k_j P_k\left[\sum_{l_1=1}^{l-2}\frac{l_1!(l-l_1-1)!}{l!}\hat{H}_j^{[l_1]}\hat{*}\hat{H}^{[l-l_1-1]}\right]\\ \nonumber 
-i k_j P_k[\hat{u}_{0,j}&\hat{*}\hat{H}^{[l]}+ \hat{H}_j^{[l]}\hat{*}\hat{u}_0+\frac{1}{l}\hat{u}_{1,j}\hat{*}\hat{H}^{[l-1]} +\frac{1}{l}\hat{H}_j^{[l-1]}\hat{*}\hat{u}_1] +P_k[a e_2\hat{S}^{[l]}]\\ \label{1-3.53b}
(l+1)(l+2)\hat{S}^{[l+1]}&=-\mu|k|^2\hat{S}^{[l]}-i k_j\left[\sum_{l_1=1}^{l-2}\frac{l_1!(l-l_1-1)!}{l!}\hat{H}_j^{[l_1]}\hat{*}\hat{S}^{[l-l_1-1]}\right]\\ \nonumber
-i k_j\left[\right.\hat{u}_{0,j}&\hat{*}\hat{S}^{[l]}+\hat{H}_j^{[l]}\hat{*}\hat{\Theta}_0+ \frac{1}{l}\hat{u}_{1,j}\hat{*}\hat{S}^{[l-1]}+ \frac{1}{l}\hat{H}_j^{[l-1]}\hat{*}\hat{\Theta}_1\left. \right].
\end{align}
In the MHD case, we have
\begin{align}\label{2-3.53}
(l+1)(l+2)\hat{W}^{[l+1]}&=-\nu|k|^2\hat{W}^{[l]}-i k_j P_k\left[\sum_{l_1=1}^{l-2}\right.\frac{l_1!(l-l_1-1)!}{l!}\hat{W}_j^{[l_1]}\hat{*}\hat{W}^{[l-l_1-1]}\\ \nonumber
+\hat{v}_{0,j}\hat{*}\hat{W}^{[l]}+\hat{W}_j^{[l]}&\hat{*}\hat{v}_0+\frac{1}{l}\hat{v}_{1,j}\hat{*}\hat{W}^{[l-1]}+ \frac{1}{l}\hat{W}_j^{[l-1]}\hat{*}\hat{v}_1\left.\right]
+\frac{i k_j}{\mu \rho} P_k\left[\hat{B}_{0,j}\right.\hat{*}\hat{Q}^{[l]}+\hat{Q}_j^{[l]}\hat{*}\hat{B}_0\\ \nonumber
+\sum_{l_1=1}^{l-2}\frac{l_1!(l-l_1-1)!}{l!}&\hat{Q}_j^{[l_1]}\hat{*}\hat{Q}^{[l-l_1-1]}+\frac{1}{l}\hat{B}_{1,j}\hat{*}\hat{Q}^{[l-1]}+\frac{1}{l}\hat{Q}_j^{[l-1]}\hat{*}\hat{B}_1\left.\right]
\\ \label{2-3.53b}
(l+1)(l+2)\hat{Q}^{[l+1]}&=-\frac{1}{\mu \sigma}|k|^2\hat{Q}^{[l]}-i k_j P_k\left[\sum_{l_1=1}^{l-2}\right.\frac{l_1!(l-l_1-1)!}{l!}\hat{W}_j^{[l_1]}\hat{*}\hat{Q}^{[l-l_1-1]}\\ \nonumber
+\hat{v}_{0,j}\hat{*}\hat{Q}^{[l]}+ \hat{W}_j^{[l]}&\hat{*}\hat{B}_0+\frac{1}{l}\hat{v}_{1,j}\hat{*}\hat{Q}^{[l-1]}+ \frac{1}{l}\hat{W}_j^{[l-1]}\hat{*}\hat{B}_1\left.\right]+ i k_j P_k\left[\right.\hat{B}_{0,j}\hat{*}\hat{W}^{[l]}+\hat{Q}_j^{[l]}\hat{*}\hat{v}_0\\ \nonumber
+\sum_{l_1=1}^{l-2}\frac{l_1!(l-l_1-1)!}{l!}&\hat{Q}_j^{[l_1]}\hat{*}\hat{W}^{[l-l_1-1]}+\frac{1}{l}\hat{B}_{1,j}\hat{*}\hat{W}^{[l-1]}+\frac{1}{l}\hat{Q}_j^{[l-1]}\hat{*}\hat{v}_1\left.\right].
\end{align}
\begin{Definition} It is useful to define a n-th order polynomial, call it $\mathcal{Q}_n$, 
\begin{equation}\nonumber
\mathcal{Q}_n(y)=\sum_{j=0}^{n}2^{n-j}\frac{y^j}{j!}.
\end{equation}
\end{Definition}
\begin{Definition}\label{maximums}
It is also useful to define constants
\begin{align*}
M_1&=\max(\nu,\mu)\\
M_2&=\max\left(\nu,\frac{1}{\mu\sigma}\right)\\
M_3&=\max\left(1,\frac{1}{\mu\rho}\right).
\end{align*}
\end{Definition}

\begin{Lemma}\label{3.2.} If $||(\hat{u}_0, \hat{\Theta}_0)||_{\gamma+2,\beta}<\infty$ and $||(\hat{v}_0, \hat{B}_0)||_{\gamma+2,\beta}<\infty$ for $\gamma >d$ and $\beta>0$, then there are constants $A_0$, $\tilde A_0$, $D_0$, $\tilde D_0>0$ not depending on $l$ or $k$ such that
\begin{eqnarray}\label{1-3.54}
|(\hat{H}^{[l]}, \hat{S}^{[l]})|\leq e^{-\beta |k|}A_0D_0^l(1+|k|)^{-\gamma}\frac{\mathcal{Q}_{2l}(|\beta k|)}{(2l+1)^2}\\ \label{2-3.54}
|(\hat{W}^{[l]}, \hat{Q}^{[l]})|\leq e^{-\beta |k|}\tilde A_0\tilde D_0^l(1+|k|)^{-\gamma}\frac{\mathcal{Q}_{2l}(|\beta k|)}{(2l+1)^2}.
\end{eqnarray}
Furthermore, the solutions guaranteed to exist in Lemma (\ref{2.8.}) have convergent power series representations in $p$, and for $|p|<(4D_0)^{-1}$ 
\begin{equation*}
(\hat{H}, \hat{S})(k,p)=(\hat{u}_1, \hat{\Theta}_1)(k)+\sum_{l=1}^{\infty}(\hat{H}^{[l]}, \hat{S}^{[l]})(k)p^l \end{equation*}
and for $|p|<(4\tilde{D}_0)^{-1}$
\begin{equation*}
(\hat{W}, \hat{Q})(k,p)=(\hat{v}_1, \hat{B}_1)(k)+\sum_{l=1}^{\infty}(\hat{W}^{[l]}, \hat{Q}^{[l]})(k)p^l.
\end{equation*}
\end{Lemma}

To prove this lemma we will establish bounds for $(\hat{H}^{[l]},\hat{S}^{[l]})$ and $(\hat{W}^{[l]}, \hat{Q}^{[l]})$ using induction. 

\begin{Lemma}\label{3.3.} For the base case, we have
\begin{eqnarray}\label{1-3.55}
|(\hat{H}^{[1]}, \hat{S}^{[1]})(k,p)|\leq \frac{e^{-\beta |k|}\mathcal{Q}_{2}(\beta |k|)A_0D_0}{(1+|k|)^{\gamma}9}\\ \label{2-3.55}
|(\hat{W}^{[1]}, \hat{Q}^{[1]})(k,p)|\leq \frac{e^{-\beta |k|}\mathcal{Q}_{2}(\beta |k|)\tilde A_0\tilde D_0}{(1+|k|)^{\gamma}9}
\end{eqnarray}
\noindent for
\begin{align}\nonumber
A_0D_0&\geq \frac{9}{\beta^2}||(\hat{u}_1, \hat{\Theta}_1)||_{\gamma,\beta}\left(C_0\beta||(\hat{u}_0, \hat{\Theta}_0)||_{\gamma,\beta}+ M_1+a\beta^2\right)\\ \nonumber
\tilde A_0\tilde D_0&\geq \frac{9}{\beta^2}||(\hat{v}_1, \hat{B}_1)||_{\gamma,\beta}M_2 M_3\left(1+C_0\beta||(\hat{v}_0, \hat{B}_0)||_{\gamma,\beta}\right)
\end{align}
\end{Lemma}

\begin{proof} From $(\ref{1-3.50})$, $(\ref{2-3.50})$, and Lemma \ref{2.3.}, we get
\begin{align}\label{1-3.52}
|(\hat{H}^{[1]}, \hat{S}^{[1]})(k,p)|\leq &\frac{e^{-\beta|k|}}{2(1+|k|)^{\gamma}}\left( |k|^2||(\hat{u}_1, \hat{\Theta}_1)||_{\gamma, \beta}M_1\right. \\ \nonumber
 &+\left. 2C_0|k|\left\|(\hat{u}_0, \hat{\Theta}_0)\right\|_{\gamma, \beta}||(\hat{u}_1, \hat{\Theta}_1)||_{\gamma, \beta}+ a||\hat{\Theta}_1||_{\gamma, \beta}\right)
\end{align}
\begin{equation}\label{2-3.52}
|(\hat{W}^{[1]}, \hat{Q}^{[1]})(k,p)|\leq \frac{e^{-\beta|k|}M_2M_3}{2(1+|k|)^{\gamma}}||(\hat{v}_1, \hat{B}_1)||_{\gamma, \beta}\left( |k|^2+4C_0|k|\left\|(\hat{v}_0, \hat{B}_0)\right\|_{\gamma, \beta}\right)
\end{equation}
The result now follows from (\ref{1-3.52}) and (\ref{2-3.52}) after noting that $\mathcal{Q}_2(\beta |k|)=4+2\beta |k|+1/2(\beta |k|)^2$.
\end{proof}

\noindent For the general terms we will need a series of lemmas, which depend heavily on the lemmas developed in the Fourier inequalities section, bounding the terms that appear on the right side of (\ref{1-3.53}) and (\ref{2-3.53}). 

\begin{Lemma}\label{3.4.} Assume that $(\hat{H}^{[l]}, \hat{S}^{[l]})$ satisfies (\ref{1-3.54}) and $(\hat{W}^{[l]}, \hat{Q}^{[l]})$ satisfies (\ref{2-3.54}) for $l\geq 1$. Then we have,
\begin{align}\nonumber
\frac{|k|^2|(\hat{H}^{[l]}, \hat{S}^{[l]})|}{(l+1)(l+2)}&\leq\frac{6A_0D_0^l e^{-\beta|k|}\mathcal{Q}_{2l+2}(\beta|k|)}{\beta^2(1+|k|)^{\gamma}(2l+3)^2}\\[8pt]\nonumber
\frac{|k|^2|(\hat{W}^{[l]}, \hat{Q}^{[l]})|}{(l+1)(l+2)}&\leq\frac{6\tilde A_0\tilde D_0^l e^{-\beta|k|}\mathcal{Q}_{2l+2}(\beta|k|)}{\beta^2(1+|k|)^{\gamma}(2l+3)^2}.
\end{align}
\end{Lemma}
\begin{proof} The proof follows from (\ref{1-3.54}) or (\ref{2-3.54}) directly by noting that for $y\geq 0$
\begin{equation}\nonumber
\frac{y^2\mathcal{Q}_{2l}(y)}{(2l+2)(2l+1)}\leq \mathcal{Q}_{2l+2}(y) \textnormal{ and }\frac{(2l+2)(2l+3)^2}{(l+1)(l+2)(2l+1)}\leq 6.
\end{equation}
\end{proof}
\begin{Lemma}\label{3.5.} Suppose $(\hat{H}^{[l]}, \hat{S}^{[l]})$ satisfies (\ref{1-3.54}) and $(\hat{W}^{[l]}, \hat{Q}^{[l]})$ satisfies (\ref{2-3.54}) for $l\geq 1$. Then both
\begin{equation}\nonumber
\frac{1}{(l+1)(l+2)}\left|k_j\left(P_k(\hat{u}_{0,j}\hat{*}\hat{H}^{[l]}), \hat{u}_{0,j}\hat{*}\hat{S}^{[l]}\right)\right| \textnormal{ and }
\frac{1}{(l+1)(l+2)}\left|k_j\left(P_k(\hat{H}^{[l]}_j\hat{*}\hat{u}_{0}), \hat{H}^{[l]}_j\hat{*}\hat{\Theta}_0\right)\right|
\end{equation}
are bounded by 
\begin{equation}\nonumber
 2^{\gamma}||(\hat{u}_0, \hat{\Theta}_0)||_{\gamma,\beta}\frac{9C_7\pi A_0D_0^l e^{-\beta|k|}}{2\beta^d(2l+3)^2(1+|k|)^\gamma}\mathcal{Q}_{2l+2}(|\beta k|).
\end{equation}
Similarly,
\begin{align}\nonumber
\frac{1}{(l+1)(l+2)}\left|k_j\left(P_k(\hat{v}_{0,j}\hat{*}\hat{W}^{[l]}),  P_k(\hat{v}_{0,j}\hat{*}\hat{Q}^{[l]})\right)\right|&,\,
\frac{1}{(l+1)(l+2)}\left|k_j\left(P_k(\hat{W}^{[l]}_j\hat{*}\hat{v}_{0}), P_k(\hat{W}^{[l]}_j\hat{*}\hat{B}_0)\right)\right|,\\ \nonumber
\frac{1}{(l+1)(l+2)}\left|k_j\left(P_k(\hat{B}_{0,j}\hat{*}\hat{Q}^{[l]}), P_k(\hat{B}_{0,j}\hat{*}\hat{W}^{[l]})\right)\right|&,\textnormal{ and } \frac{1}{(l+1)(l+2)}\left|k_j\left(P_k(\hat{Q}^{[l]}_j\hat{*}\hat{B}_{0}), P_k(\hat{Q}^{[l]}_j\hat{*}\hat{v}_0)\right)\right|
\end{align}
are bounded by 
\begin{equation}\nonumber
2^{\gamma}||(\hat{v}_0,\hat{B}_0)||_{\gamma,\beta}\frac{9C_7\pi \tilde A_0\tilde D_0^l e^{-\beta|k|}}{2\beta^d(2l+3)^2(1+|k|)^\gamma}\mathcal{Q}_{2l+2}(|\beta k|).
\end{equation}
\end{Lemma}
\begin{proof} We use the estimate (\ref{1-3.54}) on $(\hat{H}^{[l]}, \hat{S}^{[l]})$ and Lemma \ref{6.7.} in $\mathbb{R}^d$ with $n=0$ to get
\begin{align}\nonumber
|k_j\hat{u}_{0,j}\hat{*}(\hat{H}^{[l]},\hat{S}^{[l]})|&\leq ||\hat{u}_0||_{\gamma, \beta}\frac{A_0D_0^l}{(2l+1)^2}\left(|k|\int_{k'\in\mathbb{R}^d}\frac{e^{-\beta(|k'|+|k-k'|)}}{(1+|k'|)^{\gamma}(1+|k-k'|)^{\gamma}}\mathcal{Q}_{2l}(\beta|k'|)dk'\right)\\ \nonumber
&\leq \frac{||\hat{u}_0||_{\gamma, \beta}A_0D_0^l}{(2l+1)^2}\sum_{m=0}^{2l}\frac{2^{2l-m}}{m!}|k|\int_{k'\in\mathbb{R}^d}\frac{e^{-\beta(|k'|+|k-k'|)}}{(1+|k'|)^{\gamma}(1+|k-k'|)^{\gamma}}|\beta k'|^m dk'\\ \nonumber 
&\leq \frac{C_7\pi ||\hat{u}_0||_{\gamma, \beta}A_0D_0^l2^{\gamma}e^{-\beta|k|}}{(2l+1)^2\beta^d(1+|k|)^{\gamma}}\sum_{m=0}^{2l}2^{2l-m}(m+2)\mathcal{Q}_{m+2}(\beta|k|)\\ \nonumber
&\leq \frac{2^{\gamma}C_7\pi ||\hat{u}_0||_{\gamma, \beta}A_0D_0^l e^{-\beta|k|}}{(2l+1)\beta^d(1+|k|)^{\gamma}}(l+2)\mathcal{Q}_{2l+2}(\beta|k|).
\end{align}
\noindent The first part of the lemma now follows noting $\frac{2(2l+3)^2}{(2l+1)(l+1)}\leq 9$ for $l\geq 1$. For the other four terms, we use the estimate (\ref{2-3.54}) on $(\hat{W}^{[l]}, \hat{Q}^{[l]})$ and Lemma \ref{6.7.} in $\mathbb{R}^d$ with $n=0$. Hence, the proof is the same as that given above with $\tilde A_0$ in place of $A_0$ and $\tilde D_0$ in place of $D_0$. 
\end{proof}

\begin{Lemma}\label{3.6.} Suppose $(\hat{H}^{[l-1]},\hat{S}^{[l-1]})$ satisfies (\ref{1-3.54}) and  $(\hat{W}^{[l-1]},\hat{Q}^{[l-1]})$ satisfies (\ref{2-3.54}) for $l\geq 2$. Then both
\begin{equation}\nonumber
\frac{1}{l(l+1)}\left|k_j\left(P_k(\hat{u}_{1,j}\hat{*}\hat{H}^{[l-1]}), \hat{u}_{1,j}\hat{*}\hat{S}^{[l-1]}\right)\right|\textnormal{ and } \frac{1}{l(l+1)}\left|k_j\left(P_k(\hat{H}^{[l-1]}_j\hat{*}\hat{u}_{1}),\hat{H}^{[l-1]}_j\hat{*}\hat{\Theta}_1\right)\right|
\end{equation}
are bounded by
\begin{equation}\nonumber
2^{\gamma}||(\hat{u}_1, \hat{\Theta}_1)||_{\gamma,\beta}\frac{9C_7\pi A_0 D_0^{l-1} e^{-\beta|k|}\mathcal{Q}_{2l}(|\beta k|)}{2\beta^d(2l+1)^2(1+|k|)^\gamma}.
\end{equation}
Similarly,
\begin{align}\nonumber
\frac{1}{l(l+1)}\left|k_j\left(P_k(\hat{v}_{1,j}\hat{*}\hat{W}^{[l-1]}), P_k(\hat{v}_{1,j}\hat{*}\hat{Q}^{[l-1]})\right)\right|,&\, \frac{1}{l(l+1)}\left|k_j\left(P_k(\hat{W}^{[l-1]}_j\hat{*}\hat{v}_{1}), P_k(\hat{W}^{[l-1]}_j\hat{*}\hat{B}_1)\right)\right|,\\ \nonumber
\frac{1}{l(l+1)}\left|k_j\left(P_k(\hat{B}_{1,j}\hat{*}\hat{Q}^{[l-1]}), P_k(\hat{B}_{1,j}\hat{*}\hat{W}^{[l-1]})\right)\right|,&\textnormal{ and } \frac{1}{l(l+1)}\left|k_j\left(P_k(\hat{Q}^{[l-1]}_j\hat{*}\hat{B}_{1}), P_k(\hat{Q}^{[l-1]}_j\hat{*}\hat{v}_1)\right)\right|
\end{align}
are bounded by
\begin{equation}\nonumber
2^{\gamma}||(\hat{v}_1, \hat{B}_1)||_{\gamma,\beta}\frac{9C_7\pi \tilde A_0\tilde D_0^{l-1} e^{-\beta|k|}\mathcal{Q}_{2l}(|\beta k|)}{2\beta^d(2l+1)^2(1+|k|)^\gamma}.
\end{equation}
\end{Lemma}
\noindent The proof is the same as that for Lemma \ref{3.5.} with $l$ replaces by $l-1$ and $(\hat{u}_0, \hat{\Theta}_0)$ replaced by $(\hat{u}_1, \hat{\Theta}_1)$ or $(\hat{v}_0, \hat{B}_0)$ replaced by $(\hat{v}_1, \hat{B}_1)$.

\begin{Lemma}\label{3.7.} Let $l\geq 3$. Suppose $(\hat{H}^{[l_1]}, \hat{S}^{[l_1]})$ and $(\hat{H}^{[l-1-l_1]}, \hat{S}^{[l-1-l_1]})$ satisfy $(\ref{1-3.54})$ and $(\hat{W}^{[l_1]}, \hat{Q}^{[l_1]})$ and $(\hat{W}^{[l-1-l_1]}, \hat{Q}^{[l-1-l_1]})$ satisfy $(\ref{2-3.54})$ for $l_1=1,\dots, l-2$. Then
\begin{equation}\nonumber
\left|\frac{k_j}{(l+1)(l+2)}\left[\sum_{l_1=1}^{l-2}\frac{l_1!(l-1-l_1)!}{l!}\left(P_k(\hat{H}_j^{[l_1]}\hat{*}\hat{H}^{[l-1-l_1]}),\hat{H}_j^{[l_1]}\hat{*}\hat{S}^{[l-1-l_1]}\right)\right]\right| 
\end{equation}
is bounded by
\begin{equation}\nonumber
2^{\gamma+3} C_7 A_0^2D_0^{l-1}(1+|k|)^{-\gamma}e^{-\beta|k|}\frac{\mathcal{Q}_{2l}(\beta|k|)}{\beta^d (2l+3)^2}.
\end{equation} Similarly, both
\begin{equation}\nonumber
\left|\frac{k_j}{(l+1)(l+2)}\left[\sum_{l_1=1}^{l-2}\frac{l_1!(l-1-l_1)!}{l!}\left(P_k(\hat{W}_j^{[l_1]}\hat{*}\hat{W}^{[l-1-l_1]}), P_k(\hat{W}_j^{[l_1]}\hat{*}\hat{Q}^{[l-1-l_1]})\right)\right]\right| \end{equation}
and
\begin{equation}\nonumber
\left|\frac{k_j}{(l+1)(l+2)}\left[\sum_{l_1=1}^{l-2}\frac{l_1!(l-1-l_1)!}{l!}\left(P_k(\hat{Q}_j^{[l_1]}\hat{*}\hat{Q}^{[l-1-l_1]}), P_k(\hat{Q}_j^{[l_1]}\hat{*}\hat{W}^{[l-1-l_1]})\right)\right]\right| 
\end{equation}
are bounded by
\begin{equation}\nonumber
2^{\gamma+3} C_7\tilde A_0^2\tilde D_0^{l-1}(1+|k|)^{-\gamma}e^{-\beta|k|}\frac{\mathcal{Q}_{2l}(\beta|k|)}{\beta^d (2l+3)^2}.
\end{equation}
\end{Lemma}

\begin{proof} The proof is similar to that in \cite{smalltime} with $\hat{W}^{[l_2]}$ replaced by $(\hat{W}^{[l_2]}, \hat{Q}^{[l_2]})$.
If $l\geq 3$ then $l_2=l-l_1-1\geq 0$ for $l_1=1,\dots, l-2$ and we apply Lemma \ref{6.9.} in $\mathbb{R}^d$ giving,
\begin{multline}\nonumber
\frac{l_1!l_2!}{l!}|k_j \hat{W}_j^{[l_1]}\hat{*}(\hat{W}^{[l_2]}, \hat{Q}^{[l_2]})|\leq
\frac{l_1!l_2!}{l!(2l_1+1)^2(2l_2+1)^2}\tilde A_0^2\tilde D_0^{l-1}|k|\cdot \hspace{1in}\\ \nonumber
\int_{k'\in\mathbb{R}^d}e^{-\beta(|k'|+|k-k'|)}(1+|k'|)^{-\gamma}(1+|k-k'|)^{-\gamma}\mathcal{Q}_{2l_1}(\beta|k'|)\mathcal{Q}_{2l_2}(\beta|k-k'|)dk'\\ \nonumber
\leq\frac{C_7l_1!l_2!\tilde A_0^2\tilde D_0^{l-1}}{l!(2l_1+1)^2(2l_2+1)^2}\frac{\pi2^{\gamma}e^{-\beta|k|}(2l-1)(2l)(2l+1)}{3\beta^d(1+|k|)^{\gamma}}\mathcal{Q}_{2l}(\beta|k|).
\end{multline}
Thus,
\begin{align}\nonumber
\sum_{l_1=1}^{l-2}&\frac{l_1!l_2!}{l!(l+1)(l+2)}\left|k_j \hat{W}_j^{[l_1]}\hat{*}(\hat{W}^{[l_2]},\hat{Q}^{[l_2]})\right|\leq\\ \nonumber
& \frac{C_7\tilde A_0^2\tilde D_0^{l-1}\pi 2^{\gamma+1	}e^{-\beta|k|}(2l-1)(2l+1)}{3\beta^d(l+1)(l+2)(1+|k|)^{\gamma}}\mathcal{Q}_{2l}(\beta|k|)\sum_{l_1=1}^{l-2}\frac{l_1!l_2!}{(l-1)!(2l_1+1)^2(2l_2+1)^2}.
\end{align}
\noindent After noting that $\frac{l_1!l_2!}{(l-1)!}\leq 1$, $\frac{(2l-1)(2l+1)}{(l+1)(l+2)}\leq 4$, and
\begin{equation}\nonumber
\sum_{l_1=1}^{l-2}\frac{1}{((2l_1+1)^2(2l_2+1)^2}\leq \frac{C}{(2l+3)^2},
\end{equation}
\noindent  where $C=1.0755\cdots \leq 3$, the second inequality is proved. The others are done in the same manner.
\end{proof}

\begin{Lemma}\label{3.8.} For $l=2$ we have,
\begin{align}\nonumber
|(\hat{H}^{[2]},\hat{S}^{[2]})|\leq& \frac{e^{-\beta|k|}\mathcal{Q}_4(\beta |k|)}{5^2(1+|k|)^{\gamma }}\left(\frac{6A_0D_0M_1}{\beta^2}+\frac{2^{\gamma}9C_7\pi A_0D_0||(\hat{u}_0, \hat{\Theta}_0)||_{\gamma,\beta}}{\beta^d}\right.\\ \nonumber
&\left. \quad+A_0D_0a+\frac{C_0}{\beta}||(\hat{u}_1, \hat{\Theta}_1)||^2_{\gamma, \beta}\right)\\ \nonumber
|(\hat{W}^{[2]}, \hat{Q}^{[2]})|\leq& \frac{e^{-\beta|k|}Q_4(\beta |k|)}{5^2(1+|k|)^{\gamma }}\left(\frac{6\tilde A_0\tilde D_0M_2}{\beta^2}+\frac{2^{\gamma+1}9C_7\pi M_3\tilde A_0\tilde D_0||(\hat{v}_0, \hat{B}_0)||_{\gamma,\beta}}{\beta^d}\right. \\ \nonumber
&\left. \quad +M_3\frac{2C_0}{\beta}||(\hat{v}_1, \hat{B}_1)||^2_{\gamma, \beta}\right).
\end{align}
\noindent Thus, $(\hat{H}^{[2]}, \hat{S}^{[2]})$ satisfies ($\ref{1-3.54}$) and $(\hat{W}^{[2]}, \hat{Q}^{[2]})$ satisfies ($\ref{2-3.54}$) for 
\begin{equation}\label{1-3.57}
D_0^2\geq\frac{6D_0M_1}{\beta^2}+D_0a+\frac{2^{\gamma}9C_7\pi D_0}{\beta^d}||(\hat{u}_0, \hat{\Theta}_0)||_{\gamma,\beta}+ \frac{C_0}{A_0\beta}||(\hat{u}_1, \hat{\Theta}_1)||^2_{\gamma, \beta}
\end{equation}
and
\begin{equation}\label{2-3.57}
\tilde D_0^2\geq \frac{6\tilde D_0M_2}{\beta^2}+\frac{2^{\gamma+1}9C_7\pi M_3 \tilde D_0||(\hat{v}_0,\hat{B}_0)||_{\gamma,\beta}}{\beta^d}+\frac{2C_0M_3||(\hat{v}_1, \hat{B}_1)||^2_{\gamma, \beta}}{\tilde A_0\beta}.
\end{equation}
\end{Lemma}
\begin{proof} We start from $(\ref{1-3.51})$ or $(\ref{2-3.51})$. For the first term we use Lemma \ref{3.4.}. For the second term, appearing in (\ref{1-3.57}) the Boussinesq case only, we use our induction assumption and $\frac{Q_2(\beta |k|)}{54}\leq\frac{Q_4(\beta|k|)}{25}$. For the next term, we use Lemma \ref{3.5.}. For the last terms, apply Corollary \ref{2.2.} and use $\frac{|k|}{6}\leq\frac{Q_4(\beta|k|)}{25 \beta}$.
\end{proof}

\noindent{\bf Proof of Lemma \ref{3.2.}}
The base case is proved picking $D_0$ and $\tilde D_0$ large enough so $(\ref{1-3.57})$, $(\ref{2-3.57})$, $(\ref{1-3.55})$, and $(\ref{2-3.55})$ hold. For general $l\geq 2$ suppose $(\hat{H}^{[m]}, \hat{S}^{[m]})$ satisfies (\ref{1-3.54}) and $(\hat{W}^{[m]}, \hat{Q}^{[m]})$ satisfies (\ref{2-3.54}) for $m=1, \dots, l$. We estimate terms on the right of $(\ref{1-3.53})$, $(\ref{1-3.53b})$, $(\ref{2-3.53})$, and $(\ref{2-3.53b})$, using Lemma \ref{3.4.}, \ref{3.5.}, \ref{3.6.}, and \ref{3.7.} and the fact that $Q_{2l}(y)\leq 1/4Q_{2l+2}(y)$, to get 
\begin{align}\nonumber
|(\hat{H}^{[l+1]},\hat{S}^{[l+1]})|&\leq \frac{A_0D_0^{l-1}Q_{2l+2}(\beta |k|)}{(2l+3)^2(1+|k|)^{\gamma}}\left\{\frac{6D_0M_1}{\beta^2}+\frac{a D_0}{2}+\frac{2^{\gamma}9C_7\pi D_0}{\beta^d}||(\hat{u}_0, \hat{\Theta}_0)||_{\gamma, \beta} \right.  \\ \nonumber
 &\left. \qquad +\frac{2^{\gamma}9C_7\pi (2l+3)^2}{4(l+2)(2l+1)^2\beta^d}||(\hat{u}_1, \hat{\Theta}_1)||_{\gamma, \beta}+\frac{2^{\gamma +3}C_7A_0}{4\beta^d}\right\}\\ \nonumber
&\leq \frac{A_0D_0^{l+1}e^{-\beta |k|}}{(1+|k|)^{\gamma}(2l+3)^2}\mathcal{Q}_{2l+2}(\beta |k|)
\end{align}
and
\begin{align}\nonumber
|(\hat{W}^{[l+1]}, \hat{Q}^{[l+1]})|&\leq \frac{\tilde A_0\tilde D_0^{l-1}\mathcal{Q}_{2l+2}(\beta |k|)}{(2l+3)^2(1+|k|)^{\gamma}}\left\{\frac{6\tilde D_0M_2}{\beta^2} +M_3\left[ \frac{2^{\gamma+1}9C_7\pi\tilde D_0}{\beta^d}||(\hat{v}_0, \hat{B}_0)||_{\gamma, \beta}\right. \right. \\ \nonumber
 &\left. \left.\qquad+\frac{2^{\gamma+1}9C_7\pi (2l+3)^2}{4(l+2)(2l+1)^2\beta^d}||(\hat{v}_1, \hat{B}_1)||_{\gamma, \beta}+\frac{2^{\gamma+4}C_7\tilde A_0}{4\beta^d}\right] \right\}\\ \nonumber
&\leq \frac{\tilde A_0\tilde D_0^{l+1}e^{-\beta |k|}}{(1+|k|)^{\gamma}(2l+3)^2}Q_{2l+2}(\beta |k|),
\end{align}
\noindent where $D_0$ has been chosen large enough so
\begin{align}\nonumber
\left\{  \frac{6D_0M_1}{\beta^2}\right.+\frac{a D_0}{2} +&\frac{2^{\gamma}9C_7\pi D_0}{\beta^d}||(\hat{u}_0, \hat{\Theta}_0)||_{\gamma, \beta}+\frac{2^{\gamma}9C_7\pi D_0}{4\beta^d}||(\hat{u}_1, \hat{\Theta}_1)||_{\gamma, \beta}\\ \nonumber
&\left. +\frac{2^{\gamma +1}C_7A_0}{\beta^d}\right\} \leq D_0^2
\end{align}
and $\tilde D_0$ large enough so
\begin{align}\nonumber
\left\{ \frac{6\tilde D_0M_2}{\beta^2}\right. +&M_3\left[ \frac{2^{\gamma+1}9C_7\pi\tilde D_0}{\beta^d}||(\hat{v}_0, \hat{B}_0)||_{\gamma, \beta} +\frac{2^{\gamma+1}9C_7\pi }{4\beta^d}||(\hat{v}_1, \hat{B}_1)||_{\gamma, \beta} \right. \\ \nonumber
&\left. \left.+\frac{2^{\gamma+2}C_7\tilde A_0}{\beta^d}\right] \right\}\leq \tilde D_0^2.
\end{align}
\noindent We also used $\frac{(2l+3)^2}{(2l+1)^2(l+2)}\leq 1$ in the above. Thus, by induction, we have $(\ref{1-3.54})$ and $(\ref{2-3.54})$ satisfied for any $l\geq 1$. So, $\sum_{l=1}^{\infty}(\hat{H}^{[l]},\hat{S}^{[l]})(k)p^l$ is convergent for $|p|\leq \frac{1}{4 D_0}$ and $\sum_{l=1}^{\infty}(\hat{W}^{[l]},\hat{Q}^{[l]})(k)p^l$ is convergent for $|p|\leq \frac{1}{4\tilde D_0}$ since $\mathcal{Q}_{2l}(\beta |k|)\leq 4^l e^{\beta|k|/2}$. By construction of the iteration, $(\hat{H}, \hat{S})-(\hat{u}_1,\hat{\Theta}_1)= \sum_{l=1}^{\infty}(\hat{H}^{[l]}, \hat{S}^{[l]})(k)p^l$ is a solution to $(\ref{1-2.9})$ which is zero at $p=0$. Similarly, $(\hat{W}, \hat{Q})-(\hat{v}_1,\hat{B})= \sum_{l=1}^{\infty}(\hat{W}^{[l]}, \hat{Q}^{[l]})(k)p^l$ is a solution to $(\ref{2-2.9})$ which is zero at $p=0$. However, we know there are unique solutions to $(\ref{1-2.9})$ and $(\ref{2-2.9})$ which are zero and $p=0$ in the space $\mathcal{A}_L^{\infty}$, which includes analytic functions at the origin for $L$ sufficiently small. Thus, for $(\hat{H}, \hat{S})$ and $(\hat{W}, \hat{Q})$ the solutions guaranteed by Lemma \ref{2.8.} we have,
\begin{align}\nonumber
(\hat{H}, \hat{S})(k,p)&=(\hat{u}_1, \hat{\Theta}_1)(k)+\sum_{l=1}^{\infty}(\hat{H}^{[l]}, \hat{S}^{[l]})(k)p^l\\ \nonumber
(\hat{W}, \hat{Q})(k,p)&=(\hat{v}_1, \hat{B}_1)(k)+\sum_{l=1}^{\infty}(\hat{W}^{[l]}, \hat{Q}^{[l]})(k)p^l.
\end{align}

\begin{center}
 \large Estimates on $\partial_p^l(\hat{H}, \hat{S})(k,p)$ and $\partial_p^l(\hat{W}, \hat{Q})(k,p)$
\end{center}

We now want to develop estimates on $\partial_p^l(\hat{H}, \hat{S})(k,p)$ and $\partial_p^l(\hat{W}, \hat{Q})(k,p)$ in order to show that the series about any $p=p_0\in \mathbb{R}^+$ is convergent. We will proceed in the same spirit as above. That is to use induction to bound the successive derivatives. Our goal is to show that we can analytically extend our solutions along $\mathbb{R}^{+}$ with a radius of convergence independent of center $p_0$ along $\mathbb{R}^{+}$. Combining this with the fact that the solutions are exponentially bounded will give Borel Summability.
\begin{Definition}\label{Def4.1}For $l\geq 1$ we define,
\begin{align}\nonumber
(\hat{H}^{[l]}, \hat{S}^{[l]})(k,p)&=\frac{1}{l!}\partial_p^l(\hat{H}, \hat{S})(k,p) \\ \nonumber (\hat{H}^{[0]}, \hat{S}^{[0]})(k,p)&=(\hat{H}, \hat{S})(k,p)-(\hat{u}_1,\hat{\Theta}_1)\\ \nonumber
(\hat{W}^{[l]}, \hat{Q}^{[l]})(k,p)&=\frac{1}{l!}\partial_p^l(\hat{W}, \hat{Q})(k,p) \\ \nonumber (\hat{W}^{[0]}, \hat{Q}^{[0]})(k,p)&=(\hat{W}, \hat{Q})(k,p)-(\hat{v}_1,\hat{B}_1).
\end{align}
\end{Definition}

\begin{Lemma}\label{4.2.} If $||(\hat{u}_0, \hat{\Theta}_0)||_{\gamma+2,\beta}<\infty$ and $||(\hat{v}_0, \hat{B}_0)||_{\gamma+2,\beta}<\infty$ for and $\beta>0$, then there are constants $A$, $D$, $\tilde A$, $\tilde D>0$ not depending on $l, k$ or $p$ such that
\begin{eqnarray}\label{1-4.60}
|(\hat{H}^{[l]}, \hat{S}^{[l]})(k,p)|\leq \frac{e^{\omega' p}e^{-\beta |k|}AD^l}{(1+p^2)(1+|k|)^{\gamma}}\frac{\mathcal{Q}_{2l}(|\beta k|)}{(2l+1)^2}\\ \label{2-4.60}
|(\hat{W}^{[l]}, \hat{Q}^{[l]})(k,p)|\leq \frac{e^{\alpha' p}e^{-\beta |k|}\tilde A\tilde D^l}{(1+p^2)(1+|k|)^{\gamma}}\frac{\mathcal{Q}_{2l}(|\beta k|)}{(2l+1)^2}
\end{eqnarray}
\noindent where $\omega'=\omega+1$ and $\alpha'=\alpha+1$ for $\omega$ and $\alpha$ chosen as in Lemma \ref{2.8.}. We will prove the lemma by induction, and as before we will develop several lemmas to establish the bound. 
\end{Lemma}

For $l=0$, we use Lemma \ref{2.9.} which says that for $\omega$ and $\alpha$ sufficiently large
\begin{align}\nonumber
|(\hat H, \hat S)(k,p)|&\leq \frac{2e^{-\beta |k|+\omega p}||(\hat u_1, \hat \Theta_1)||_{\gamma, \beta}}{(1+|k|)^{\gamma}}\\ \nonumber
|(\hat W, \hat Q)(k,p)|&\leq \frac{2e^{-\beta |k|+\alpha p}||(\hat v_1, \hat B_1)||_{\gamma, \beta}}{(1+|k|)^{\gamma}}.
\end{align}
\noindent We chose $\omega'=\omega+1$ and $\alpha'=\alpha +1$ and recall Definition \ref{Def4.1} to get
\begin{align}\label{4.59}
|(\hat H^{[0]}, \hat{S}^{[0]})(k,p)|&\leq \frac{3e^{-\beta |k|+\omega' p}||(\hat u_1, \hat \Theta_1)||_{\gamma, \beta}}{(1+p^2)(1+|k|)^{\gamma}},\\ \nonumber
|(\hat W^{[0]}, \hat{Q}^{[0]})(k,p)|&\leq \frac{3e^{-\beta |k|+\alpha' p}||(\hat v_1, \hat B_1)||_{\gamma, \beta}}{(1+p^2)(1+|k|)^{\gamma}},
\end{align}
\noindent and the base cases of $(\ref{1-4.60})$ and $(\ref{2-4.60})$ are proved for $A=3||(\hat{u}_1, \hat{\Theta}_1)||_{\gamma, \beta}$ and $\tilde A=3||(\hat{v}_1, \hat{B}_1)||_{\gamma, \beta}$. 

For the general case ($l\geq1$) we take $\partial _p^l$ in $(\ref{1-2.9})$ or $(\ref{2-2.9})$ and divide by $l!$, to obtain
\begin{multline}\label{1-4.62}
p\hat{H}^{[l]}_{pp}+(l+2)\hat{H}^{[l]}_p+\nu |k|^2\hat{H}^{[l]}=\left(-i k_j P_k[\hat{u}_{0,j}\hat{*}\hat{u}_1+\hat{u}_{1,j}\hat{*}\hat{u}_0]-\nu|k|^2\hat{u}_1\right)\delta_{l,0}\hspace{.5 in}\\ 
-i k_j P_k\left[ \int_0^p\hat{H}^{[l]}_j(\cdot,p-s)\hat{*}\hat{H}^{[0]}(\cdot,s)ds +\sum_{l_1=1}^{l-1}\frac{l_1!(l-l_1-1)!}{l!}\hat{H}^{[l_1]}_j(\cdot,0)\hat{*}\hat{H}^{[l-l_1-1]}(\cdot,p)\right]\\ 
-i k_j P_k[\frac{1}{l}(\hat{u}_{1,j}\hat{*}\hat{H}^{[l-1]}+\hat{H}_j^{[l-1]}\hat{*}\hat{u}_1)+\hat{H}^{[l]}_j\hat{*}\hat{u}_0+
\hat{u}_{0,j}\hat{*}\hat{H}^{[l]} + \delta_{l,1}\hat{u}_{1,j}\hat{*}\hat{u}_1]+P_k(a e_2\hat S^{[l]})
\end{multline}
\begin{multline}\label{2-4.62}
p\hat{S}^{[l]}_{pp}+(l+2)\hat S^{[l]}_p+\mu |k|^2\hat{S}^{[l]}=\left(-i k_j[\hat{u}_{0,j}\hat{*}\hat{\Theta}_1+\hat{u}_{1,j}\hat{*}\hat{\Theta}_0]-\mu|k|^2\hat{\Theta}_1\right)\delta_{l,0}\hspace{.5in}\\ 
-i k_j\left[ \int_0^p\hat{H}^{[l]}_j(\cdot,p-s)\hat{*}\hat{S}^{[0]}(\cdot,s)ds +\sum_{l_1=1}^{l-1}\frac{l_1!(l-l_1-1)!}{l!}\hat{H}^{[l_1]}_j(\cdot,0)\hat{*}\hat{S}^{[l-l_1-1]}(\cdot,p)\right]\\ -i k_j[\frac{1}{l}(u_{1,j}\hat{*}\hat{S}^{[l-1]}+\hat{H}_{j}^{[l-1]}\hat{*}\hat{\Theta}_1)+\hat{H}_j^{[l]}\hat{*}\hat{\Theta}_0+\hat{u}_{0,j}\hat{*}\hat{S}^{[l]}+ \delta_{l=1}\hat{u}_{1,j}\hat{*}\hat{\Theta}_1]
\end{multline}
and 
\begin{multline}\label{3-4.62}
p\hat{W}^{[l]}_{pp}+(l+2)\hat{W}^{[l]}_p+\nu |k|^2\hat{W}^{[l]}=\hspace{3 in}\\ 
\left(-i k_j P_k[\hat{v}_{0,j}\hat{*}\hat{v}_1+\hat{v}_{1,j}\hat{*}\hat{v}_0]+\frac{i k_j}{\mu \rho} P_k[\hat{B}_{0,j}\hat{*}\hat{B}_1+\hat{B}_{1,j}\hat{*}\hat{B}_0] -\nu|k|^2\hat{v}_1\right)\delta_{l,0}\hspace{.5 in}\\ 
-i k_j P_k\left[ \int_0^p\hat{W}^{[l]}_j(\cdot,p-s)\hat{*}\hat{W}^{[0]}(\cdot,s)ds +\sum_{l_1=1}^{l-1}\frac{l_1!(l-l_1-1)!}{l!}\hat{W}^{[l_1]}_j(\cdot,0)\hat{*}\hat{W}^{[l-l_1-1]}(\cdot,p)\right]\\
+\frac{i k_j}{\mu \rho} P_k\left[ \int_0^p\hat{Q}^{[l]}_j(\cdot,p-s)\hat{*}\hat{Q}^{[0]}(\cdot,s)ds +\sum_{l_1=1}^{l-1}\frac{l_1!(l-l_1-1)!}{l!}\hat{Q}^{[l_1]}_j(\cdot,0)\hat{*}\hat{Q}^{[l-l_1-1]}(\cdot,p)\right]\\
-i k_j P_k[\frac{1}{l}(\hat{v}_{1,j}\hat{*}\hat{W}^{[l-1]}+\hat{W}_j^{[l-1]}\hat{*}\hat{v}_1)+\hat{W}^{[l]}_j\hat{*}\hat{v}_0+
\hat{v}_{0,j}\hat{*}\hat{W}^{[l]} + \delta_{l,1}\hat{v}_{1,j}\hat{*}\hat{v}_1]\\ 
+\frac{i k_j}{\mu \rho} P_k[\frac{1}{l}(\hat{B}_{1,j}\hat{*}\hat{Q}^{[l-1]}+\hat{Q}_j^{[l-1]}\hat{*}\hat{B}_1)+\hat{Q}^{[l]}_j\hat{*}\hat{B}_0+
\hat{B}_{0,j}\hat{*}\hat{Q}^{[l]} + \delta_{l,1}\hat{B}_{1,j}\hat{*}\hat{B}_1]
\end{multline}
\begin{multline}\label{4-4.62}
p\hat{Q}^{[l]}_{pp}+(l+2)\hat Q^{[l]}_p+\frac{1}{\mu \sigma} |k|^2\hat{Q}^{[l]}=\hspace{3 in}\\ 
\left(-i k_j P_k[\hat{v}_{0,j}\hat{*}\hat{B}_1+\hat{v}_{1,j}\hat{*}\hat{B}_0]+i k_j P_k[\hat{B}_{0,j}\hat{*}\hat{v}_1+\hat{B}_{1,j}\hat{*}\hat{v}_0]-\frac{1}{\mu \sigma}|k|^2\hat{B}_1\right)\delta_{l,0}\hspace{.5in}\\ 
-i k_j P_k\left[ \int_0^p\hat{W}^{[l]}_j(\cdot,p-s)\hat{*}\hat{Q}^{[0]}(\cdot,s)ds +\sum_{l_1=1}^{l-1}\frac{l_1!(l-l_1-1)!}{l!}\hat{W}^{[l_1]}_j(\cdot,0)\hat{*}\hat{Q}^{[l-l_1-1]}(\cdot,p)\right]\\ 
+i k_j P_k\left[ \int_0^p\hat{Q}^{[l]}_j(\cdot,p-s)\hat{*}\hat{W}^{[0]}(\cdot,s)ds +\sum_{l_1=1}^{l-1}\frac{l_1!(l-l_1-1)!}{l!}\hat{Q}^{[l_1]}_j(\cdot,0)\hat{*}\hat{W}^{[l-l_1-1]}(\cdot,p)\right]\\ -i k_j P_k[\frac{1}{l}(v_{1,j}\hat{*}\hat{Q}^{[l-1]}+\hat{W}_{j}^{[l-1]}\hat{*}\hat{B}_1)+\hat{W}_j^{[l]}\hat{*}\hat{B}_0+\hat{v}_{0,j}\hat{*}\hat{Q}^{[l]}+ \delta_{l=1}\hat{v}_{1,j}\hat{*}\hat{B}_1]\\  +i k_j P_k[\frac{1}{l}(B_{1,j}\hat{*}\hat{W}^{[l-1]}+\hat{Q}_{j}^{[l-1]}\hat{*}\hat{v}_1)+\hat{Q}_j^{[l]}\hat{*}\hat{v}_0+\hat{B}_{0,j}\hat{*}\hat{W}^{[l]}+ \delta_{l=1}\hat{B}_{1,j}\hat{*}\hat{v}_1]
\end{multline}
\noindent Identify the right hand side of these four equations by $R_m^{[l]}$ for $m=1,\dots, 4$ respectively.

\begin{Lemma}\label{4.4.} For any $l\geq 0$ and for some absolute constant $C_6$, if $(\hat{H}^{[l]}, \hat{S}^{[l]})$ satisfies $(\ref{1-4.60})$, $(\hat{W}^{[l]}, \hat{Q}^{[l]})$ satisfies $(\ref{2-4.60})$, and both are bounded at $p=0$ then
\begin{multline}\nonumber
|(\hat{H}^{[l+1]}, \hat{S}^{[l+1]})(k,p)|\leq \frac{C_6}{(l+1)^{5/3}}\sup_{p'\in [0,p]}|(\hat{R}^{[l]}_1, \hat{R}_2^{[l]})|+\frac{M_1|k|^2|(\hat{H}^{[l]}, \hat{S}^{[l]})(k,0)|}{(l+1)(l+2)}\\ \nonumber
|(\hat{W}^{[l+1]}, \hat{Q}^{[l+1]})(k,p)|\leq \frac{C_6}{(l+1)^{5/3}}\sup_{p'\in [0,p]}|(\hat{R}^{[l]}_3, \hat{R}_4^{[l]})|+\frac{M_2|k|^2|(\hat{W}^{[l]}, \hat{Q}^{[l]})(k,0)|}{(l+1)(l+2)}.
\end{multline}
\end{Lemma}
\begin{proof} The proof is in \cite{smalltime} under Lemma 4.4. The lemma is dependent only on the operator $\mathcal{D}$ which is the same in our case. The idea of the proof is as follows. We invert the operator on the left of (\ref{1-4.62}) with the requirement that $\hat{H}$ is bounded at $p=0$, obtaining
\begin{equation}\nonumber
\hat{H}^{[l]}(k,p)=\int_0^p\mathcal{L}(2|k|\sqrt{\nu p},2|k|\sqrt{\nu p'})\hat{R}_1^{(l)}(k,p')dp'+2^{l+1}(l+1)!\frac{J_{l+1}(z)}{z^{l+1}}\hat{H}^{[l]}(k,0),
\end{equation}
where
\begin{equation}\nonumber
\mathcal{L}(z,z')=\pi z^{-(l+1)}\left[-J_{l+1}(z)z'^{(l+1)}Y_{l+1}(z')+z'^{(l+1)}J_{l+1}(z')Y_{l+1}(z)\right].
\end{equation}
Then, we take a derivative with respect to $p$ yielding
\begin{equation}\nonumber
(l+1)\hat{H}^{[l+1]}(k,p)=\frac{|k|\sqrt{\nu}}{\sqrt{p}}\int_0^p\mathcal{L}_z(2|k|\sqrt{\nu p},2|k|\sqrt{\nu p'})\hat{R}_1^{(l)}(k,p')dp'-2^{l+2}(l+1)!|k|^2\frac{J_{l+2}(z)}{z^{l+2}}\hat{H}^{[l]}(k,0).
\end{equation}
Using properties of Bessel functions, it is know that 
\begin{equation}\nonumber
2^{l+2}(l+1)!\left|\frac{J_{l+2}(z)}{z^{l+2}}\right|\leq\frac{1}{l+2} 
\end{equation}
and that
\begin{equation}\nonumber
\int_0^z\frac{z'}{z}|\mathcal{L}_z(z,z')|dz'\leq\frac{C}{(l+1)^{2/3}},
\end{equation}
where the constant is independent of $l$. Thus, after a change of variables,
\begin{equation*}
(l+1)|\hat{H}^{[l+1]}(k,p)|\leq \sup_{p'\in[0,p]}|\hat{R}_1^{(l)}|\frac{C}{(l+1)^{2/3}}+\frac{|k|^2|\hat{H}(k,0)|}{l+2},
\end{equation*}
and the claim follows.
\end{proof}

\begin{Lemma}\label{4.6.} Sup
\begin{equation}\nonumber
\left|k_j\left(P_k(\hat{u}_{0,j}\hat{*}\hat{H}^{[l]}), \hat{u}_{0,j}\hat{*}\hat{S}^{[l]}\right)\right|\textnormal{ and }
\left|k_j\left(P_k(\hat{H}^{[l]}_j\hat{*}\hat{u}_{0}), \hat{H}^{[l]}_j\hat{*}\hat{\Theta}_0\right)\right|
\end{equation} 
are bounded by
\begin{equation}\nonumber
C_1 ||(\hat{u}_0, \hat{\Theta}_0)||_{\gamma,\beta}\frac{(l+1)^{2/3} AD^l e^{-\beta|k|+\omega' p}}{(2l+1)(1+p^2)(1+|k|)^\gamma}\mathcal{Q}_{2l+2}(|\beta k|).
\end{equation}
Similarly,
\begin{align*}
\left|k_j\left(P_k(\hat{v}_{0,j}\hat{*}\hat{W}^{[l]}), P_k(\hat{v}_{0,j}\hat{*}\hat{Q}^{[l]})\right)\right|&,\,
\left|k_j\left(P_k(\hat{W}^{[l]}_j\hat{*}\hat{v}_{0}), P_k(\hat{W}^{[l]}_j\hat{*}\hat{B}_0)\right)\right|,\\ \nonumber
\left|k_j\left(P_k(\hat{B}_{0,j}\hat{*}\hat{S}^{[l]}), P_k(\hat{B}_{0,j}\hat{*}\hat{W}^{[l]})\right)\right|&,\textnormal{ and }
\left|k_j\left(P_k(\hat{Q}^{[l]}_j\hat{*}\hat{B}_{0}), P_k(\hat{Q}^{[l]}_j\hat{*}\hat{v}_0)\right)\right|
\end{align*}
are bounded by
\begin{equation}\nonumber
C_1 ||(\hat{v}_0, \hat{B}_0)||_{\gamma,\beta}\frac{(l+1)^{2/3} \tilde A\tilde D^l e^{-\beta|k|+\alpha' p}}{(2l+1)(1+p^2)(1+|k|)^\gamma}\mathcal{Q}_{2l+2}(|\beta k|).
\end{equation}
We also have
\begin{equation}\nonumber
|P_k(a e_2\hat{S}^{[l]})|\leq a\frac{e^{\omega' p}e^{-\beta |k|}AD^l}{(1+p^2)(1+|k|)^{\gamma}}\frac{\mathcal{Q}_{2l}(|\beta k|)}{(2l+1)^2}.
\end{equation}
\noindent In the above, $C_1=C_1(d)$ is defined in Lemma \ref{6.10.}.
\end{Lemma}

\begin{proof} For the first inequality, we use (\ref{1-4.60}) and then apply Lemma \ref{6.10.} to get
\begin{multline}\nonumber
(1+p^2)e^{-\omega' p}|k_j\hat{u}_{0,j}\hat{*}(\hat{H}^{[l]}, \hat{S}^{[l]})|\leq||\hat{u}_0||_{\gamma,\beta}\frac{A D^l}{(2l+1)^2}|k|\int_{k'\in\mathbb{R}^d}\frac{e^{-\beta(|k'|+|k-k'|)}}{(1+|k'|)^{\gamma}(1+|k-k'|)^{\gamma}}\mathcal{Q}_{2l}(\beta|k'|)dk'\\ \nonumber
\leq C_1(l+1)^{2/3}||\hat{u}_0||_{\gamma,\beta}\frac{A D^l e^{-\beta |k|}}{(2l+1)(1+|k|)^{\gamma}}\mathcal{Q}_{2l+2}(\beta|k|).
\end{multline}
\noindent The other inequalities are similar except for the last which is simply the statement of the assumed bound.
\end{proof}

\begin{Lemma}\label{4.7.} Suppose $(\hat{H}^{[l]}, \hat{S}^{[l]})$ satisfies $(\ref{1-4.60})$ and $(\hat{W}^{[l]}, \hat{Q}^{[l]})$ satisfies $(\ref{2-4.60})$ for $l\geq 1$. Then both
\begin{equation}\nonumber
\left|\frac{k_j}{l}\left(P_k(\hat{u}_{1,j}\hat{*}\hat{H}^{[l-1]}), \hat{u}_{1,j}\hat{*}\hat{S}^{[l-1]}\right)\right|\textnormal{ and }
\left|\frac{k_j}{l}\left(P_k(\hat{H}^{[l-1]}_j\hat{*}\hat{u}_{1}), \hat{H}^{[l-1]}_j\hat{*}\hat{\Theta}_1\right)\right|
\end{equation}
are bounded by 
\begin{equation}\nonumber
C_1 ||(\hat{u}_1, \hat{\Theta}_1)||_{\gamma,\beta}\frac{l^{2/3} AD^{l-1}e^{-\beta|k|+\omega' p}}{l(2l-1)(1+p^2)(1+|k|)^\gamma}\mathcal{Q}_{2l}(|\beta k|).
\end{equation}
Similarly,
\begin{align*}
\left|\frac{k_j}{l}\left(P_k(\hat{v}_{1,j}\hat{*}\hat{W}^{[l-1]}), P_k(\hat{v}_{1,j}\hat{*}\hat{Q}^{[l-1]})\right)\right|&, \,
\left|\frac{k_j}{l}\left(P_k(\hat{W}^{[l-1]}_j\hat{*}\hat{v}_{1}), P_k(\hat{W}^{[l-1]}_j\hat{*}\hat{B}_1)\right)\right|,\\ \nonumber
\left|\frac{k_j}{l}\left(P_k(\hat{B}_{1,j}\hat{*}\hat{Q}^{[l-1]}), P_k(\hat{B}_{1,j}\hat{*}\hat{W}^{[l-1]})\right)\right|&,\textnormal{ and }
\left|\frac{k_j}{l}\left(P_k(\hat{Q}^{[l-1]}_j\hat{*}\hat{B}_{1}), P_k(\hat{Q}^{[l-1]}_j\hat{*}\hat{v}_1)\right)\right|
\end{align*}
are bounded by
\begin{equation}\nonumber
C_1 ||(\hat{v}_1, \hat{B}_1)||_{\gamma,\beta}\frac{l^{2/3} \tilde A\tilde D^{l-1}e^{-\beta|k|+\alpha' p}}{l(2l-1)(1+p^2)(1+|k|)^\gamma}\mathcal{Q}_{2l}(|\beta k|).
\end{equation}
\end{Lemma}
\noindent The proof is the same as Lemma \ref{4.6.} with $l-1$ replacing $l$, $(\hat{u}_1,\hat{\Theta}_1)$ replacing $(\hat{u}_0, \hat{\Theta}_0)$, and $(\hat{v}_1, \hat{B}_1)$ replacing $(\hat{v}_0, \hat{B}_0)$.

\begin{Lemma}\label{4.8.} Suppose $(\hat{H}^{[l]}, \hat{S}^{[l]})$ satisfies $(\ref{1-4.60})$ and $(\hat{W}^{[l]}, \hat{Q}^{[l]})$ satisfies $(\ref{2-4.60})$ for $l\geq1$. Then 
\begin{equation}\nonumber
\left|\frac{k_j}{l}\left(P_k(\hat{H}_j^{[l-1]}(\cdot,0)\hat{*}\hat{H}^{[0]}(\cdot, p)),\hat{H}_j^{[l-1]}(\cdot,0)\hat{*}\hat{S}^{[0]}(\cdot, p)\right)\right|
\end{equation}
is bounded by
\begin{equation}\nonumber
C_1\frac{(l+1)^{2/3}\tilde A^2 \tilde D^{l-1}e^{-\beta |k|+\alpha' p}}{l(2l-1)(1+|k|)^{\gamma}(1+p^2)}\mathcal{Q}_{2l}(\beta |k|).
\end{equation}
We also have 
\begin{equation}\nonumber
\left|\frac{k_j}{l}\left(P_k(\hat{W}_j^{[l-1]}(\cdot,0)\hat{*}\hat{W}^{[0]}(\cdot, p)), P_k(\hat{W}_j^{[l-1]}(\cdot,0)\hat{*}\hat{Q}^{[0]}(\cdot, p))\right)\right|
\end{equation}
and
\begin{equation}\nonumber
\left|\frac{k_j}{l}\left[\left(P_k(\hat{Q}_j^{[l-1]}(\cdot,0)\hat{*}\hat{Q}^{[0]}(\cdot, p)), P_k(\hat{Q}_j^{[l-1]}(\cdot,0)\hat{*}\hat{W}^{[0]}(\cdot, p))\right)\right]\right|
\end{equation}
bounded by
\begin{equation}\nonumber
C_1\frac{(l+1)^{2/3}\tilde A^2 \tilde D^{l-1}e^{-\beta |k|+\alpha' p}}{l(2l-1)(1+|k|)^{\gamma}(1+p^2)}\mathcal{Q}_{2l}(\beta |k|).
\end{equation}
\end{Lemma}
 
\begin{proof} We give the proof of one of the magnetic B$\acute{\textnormal{e}}$nard cases the others are similar. Using $(\ref{2-4.60})$ with $p=0$ and $(\ref{4.59})$ with $\tilde A=3||(\hat{v}_1, \hat{B_1})||_{\gamma, \beta}$ along with Lemma \ref{6.10.}, we get
\begin{align}\nonumber 
(1+p^2)e^{-\alpha' p}&\left|\right.\frac{k_j}{l}[\hat{W}_j^{[l-1]}(\cdot,0)\hat{*}(\hat{W}^{[0]}, \hat{Q}^{[0]})(\cdot, p)]\left.\right|\\ \nonumber 
&\leq \frac{\tilde A^2\tilde D^{l-1}}{l(2l-1)^2}|k|\int_{k'\in \mathbb{R}^d}\frac{e^{-\beta(|k'|+|k-k'|)}}{(1+|k'|)^{\gamma}(1+|k-k'|)^{\gamma}}\mathcal{Q}_{2l-2}(\beta |k'|)dk'\\ \nonumber
&\leq C_1\frac{l^{2/3}\tilde A^2 \tilde D^{l-1}e^{-\beta |k|}}{l(2l-1)(1+|k|)^{\gamma}}\mathcal{Q}_{2l}(\beta |k|)
\end{align}
\noindent From this the lemma follows after noting $(1+l)^{2/3}\geq l^{2/3}$ and using Lemma \ref{2.3.}.
\end{proof}

\begin{Lemma}\label{4.9.} Suppose $(\hat{H}^{[l_1]}, \hat{S}^{[l_1]})$ and $(\hat{H}^{[l-l_1-1]}, \hat{S}^{[l-l_1-1]})$ satisfies $(\ref{1-4.60})$ and $(\hat{W}^{[l_1]}, \hat{Q}^{[l_1]})$ and $(\hat{W}^{[l-l_1-1]}, \hat{Q}^{[l-l_1-1]})$ satisfies $(\ref{2-4.60})$ for $l_1=1, \dots , l-2$ where $l\geq 2$. Then for $C_8=82$ and  $C_7=C_7(d)$ given in Lemma \ref{6.9.}, we have
\begin{equation}\nonumber
\left|k_j\sum_{l_1=1}^{l-2}\frac{l_1!(l-l_1-1)!}{l!}\left(P_k(\hat{H}^{[l_1]}_j(\cdot,0)\hat{*}\hat{H}^{[l-l_1-1]}(\cdot,p)), \hat{H}^{[l_1]}_j(\cdot,0)\hat{*}\hat{S}^{[l-l_1-1]}(\cdot,p)\right)\right|
\end{equation}
is bounded by 
\begin{equation*}
C_8C_72^{\gamma}\pi A^2D^{l-1}\frac{e^{-\beta|k|+\omega' p}}{3\beta^d(1+p^2)(1+|k|)^{\gamma}}\frac{l\mathcal{Q}_{2l}(\beta |k|)}{(2l+3)^2}.
\end{equation*}
Both
\begin{equation} \nonumber
\left|k_j\sum_{l_1=1}^{l-2}\frac{l_1!(l-l_1-1)!}{l!}\left(P_k(\hat{W}^{[l_1]}_j(\cdot,0)\hat{*}\hat{W}^{[l-l_1-1]}(\cdot,p)), P_k(\hat{W}^{[l_1]}_j(\cdot,0)\hat{*}\hat{Q}^{[l-l_1-1]}(\cdot,p))\right)\right|
\end{equation}
and 
\begin{equation}\nonumber
\left|k_j\sum_{l_1=1}^{l-2}\frac{l_1!(l-l_1-1)!}{l!}\left(P_k(\hat{Q}^{[l_1]}_j(\cdot,0)\hat{*}\hat{Q}^{[l-l_1-1]}(\cdot,p)), P_k(\hat{Q}^{[l_1]}_j(\cdot,0)\hat{*}\hat{W}^{[l-l_1-1]}(\cdot,p))\right)\right|
\end{equation}
are bounded by
\begin{equation}\nonumber
C_8C_72^{\gamma}\pi\tilde  A^2\tilde D^{l-1}\frac{e^{-\beta|k|+\alpha' p}}{3\beta^d(1+p^2)(1+|k|)^{\gamma}}\frac{l\mathcal{Q}_{2l}(\beta |k|)}{(2l+3)^2}.
\end{equation}
\end{Lemma}
\noindent The proof is the same as in \cite{smalltime} the only difference is a change in the constants arising when Lemma \ref{6.9.} in $\mathbb{R}^2$ or $\mathbb{R}^3$ is applied.

\begin{Lemma}\label{4.10.} Suppose $(\hat{H}^{[l]}, \hat{S}^{[l]})$ satisfies $(\ref{1-4.60})$ and $(\hat{W}^{[l]}, \hat{Q}^{[l]})$ satisfies $(\ref{2-4.60})$ for $l\geq0$. Then
\begin{multline}\nonumber
\left|k_j\int_0^p\left(P_k(\hat{H}^{[l]}_j(\cdot,p-s)\hat{*}\hat{H}^{[0]}(\cdot,s)), \hat{H}^{[l]}_j(\cdot,p-s)\hat{*}\hat{S}^{[0]}(\cdot,s)\right)ds\right| \hspace{.5 in}\\ \nonumber
\leq C_1M_0A^2 D^{l}\frac{(l+1)^{2/3}e^{-\beta |k|+\omega' p}}{(2l+1)(1+|k|)^{\gamma}(1+p^2)}\mathcal{Q}_{2l+2}(\beta |k|).
\end{multline}
Similarly,
\begin{align*}
\left|k_j\int_0^p\left(P_k(\hat{W}^{[l]}_j(\cdot,p-s)\hat{*}\hat{W}^{[0]}(\cdot,s)), P_k(\hat{W}^{[l]}_j(\cdot,p-s)\hat{*}\hat{Q}^{[0]}(\cdot,s))\right) ds\right| \textnormal{ and}\\ \nonumber
\left|k_j\int_0^p\left(P_k(\hat{Q}^{[l]}_j(\cdot,p-s)\hat{*}\hat{Q}^{[0]}(\cdot,s)), P_k(\hat{Q}^{[l]}_j(\cdot,p-s)\hat{*}\hat{W}^{[0]}(\cdot,s))\right) ds\right| 
\end{align*}
are bounded by
\begin{equation}\nonumber
C_1M_0\tilde A^2 \tilde D^{l}\frac{(l+1)^{2/3}e^{-\beta |k|+\alpha' p}}{(2l+1)(1+|k|)^{\gamma}(1+p^2)}\mathcal{Q}_{2l+2}(\beta |k|).
\end{equation}
\noindent In the above, $M_0$, defined in Lemma \ref{2.6.}, is such that 
\begin{equation}\nonumber
\int_0^p\frac{1}{(1+(p-s)^2)(1+s^2)}ds\leq\frac{M_0}{1+p^2}.
\end{equation}
\end{Lemma}

\begin{proof} Using (\ref{2-4.60}) for the first inequality and Lemma \ref{6.10.} and Lemma \ref{2.6.} for the second, we have
\begin{align}\nonumber
&\left|k_j\int_0^p\left(P_k(\hat{W}^{[l]}_j(\cdot,p-s)\hat{*}\hat{Q}^{[0]}(\cdot,s)), P_k(\hat{W}^{[l]}_j(\cdot,p-s)\hat{*}\hat{W}^{[0]}(\cdot,s))\right)ds\right|\leq \\ \nonumber
&|k|\frac{\tilde A^2\tilde D^l}{(2l+1)^2}\int_0^p\int_{k'\in\mathbb{R}^d} \frac{e^{-\beta|k'|+|k-k'|}e^{\alpha'(p-s)+\alpha's}}{(1+(p-s)^2)(1+s^2)(1+|k'|)^{\gamma}(1+|k-k'|)^{\gamma}}\mathcal{Q}_{2l}(\beta|k'|)ds dk'\\ \nonumber
&\qquad\qquad\leq C_1M_0\tilde A^2 \tilde D^{l}\frac{(l+1)^{2/3}e^{-\beta |k|+\alpha p}}{(2l+1)(1+|k|)^{\gamma}(1+p^2)}\mathcal{Q}_{2l+2}(\beta |k|)
\end{align}
The rest are computed in the same way.
\end{proof}

\begin{Lemma}\label{4.11.} We have 
\begin{align}\nonumber k_j\left(P_k(\hat{u}_{0,j}\hat{*}\hat{u}_1),\hat{u}_{0,j}\hat{*}\hat{\Theta}_1\right)&+k_j\left(P_k(\hat{u}_{1,j}\hat{*}\hat{u}_0), \hat{u}_{1,j}\hat{*}\hat{\Theta}_0\right)\\ \nonumber
& \leq\frac{2C_0|k|e^{-\beta |k|}}{(1+|k|)^{\gamma}}||(\hat{u}_0,\hat{\Theta}_0)||_{\gamma, \beta}||(\hat{u}_1,\hat{\Theta}_1)||_{\gamma, \beta}\\ \nonumber
 \left|k_j\left(P_k(\hat{u}_{1,j}\hat{*}\hat{u}_1),\hat{u}_{1,j}\hat{*}\hat{\Theta}_1\right)\right|&\leq \frac{|k|e^{-\beta |k|}C_0}{(1+|k|)^{\gamma}}||\hat{u}_1, \hat{\Theta}_1||_{\gamma, \beta}^2.
\end{align}
Similarly, we have 
\begin{align}\nonumber
k_j\left(P_k(\hat{v}_{0,j}\hat{*}\hat{v}_1), P_k(\hat{v}_{0,j}\hat{*}\hat{B}_1)\right)&+k_j\left(P_k(\hat{v}_{1,j}\hat{*}\hat{v}_0), P_k(\hat{v}_{1,j}\hat{*}\hat{B}_0)\right)\textnormal{ and }\\ \nonumber
k_j\left(P_k(\hat{B}_{0,j}\hat{*}\hat{B}_1), P_k(\hat{B}_{0,j}\hat{*}\hat{v}_1)\right)&+k_j\left(P_k(\hat{B}_{1,j}\hat{*}\hat{B}_0), P_k(\hat{B}_{1,j}\hat{*}\hat{v}_0)\right)
\end{align}
bounded by
\begin{equation*}
\frac{2C_0|k|e^{-\beta |k|}}{(1+|k|)^{\gamma}}||(\hat{v}_0, \hat{B}_0)||_{\gamma, \beta}||(\hat{v}_1, \hat{B}_1)||_{\gamma, \beta}.
\end{equation*}
Finally, we have
\begin{equation*}
\left|k_j\left(P_k(\hat{v}_{1,j}\hat{*}\hat{v}_1), P_k(\hat{v}_{1,j}\hat{*}\hat{B}_1)\right)\right|
\textnormal{ and }\left|k_j\left(P_k(\hat{B}_{1,j}\hat{*}\hat{B}_1), P_k(\hat{B}_{1,j}\hat{*}\hat{v}_1)\right)\right|
\end{equation*} 
bounded by 
\begin{equation*}
\frac{A e^{-\beta |k|+\alpha'p}}{(1+|k|)^{\gamma}(1+p^2)}\frac{4C_0\mathcal{Q}_4(\beta |k|)}{25A\beta}||(\hat{v}_1, \hat{B}_1)||^2_{\gamma, \beta}.
\end{equation*}
\end{Lemma}

\begin{proof} The first two claims follow directly from Corollary \ref{2.2.} and Lemma \ref{2.3.}. The last uses the additional fact that 
\begin{equation}\nonumber
\frac{4\mathcal{Q}_4(\beta |k|)C_0}{25\beta}\geq \frac{32\beta |k|C_0}{25\beta}\geq C_0|k|.
\end{equation}
Thus,
\begin{equation*}
\frac{|k|e^{-\beta |k|}C_0}{(1+|k|)^{\gamma}}||(\hat{v}_1, \hat{B}_1)||_{\gamma, \beta}^2\leq \frac{A e^{-\beta |k|+\alpha'p}}{(1+|k|)^{\gamma}(1+p^2)}\frac{4C_0\mathcal{Q}_4(\beta |k|)}{25A\beta}||(\hat{v}_1, \hat{B}_1)||^2_{\gamma, \beta}
\end{equation*}
and the last claim follows.
\end{proof}

\begin{Lemma}\label{4.12.} For the case $l=1$, we have 
\begin{align*}
|(\hat{H}^{[1]}, \hat{S}^{[1]})(k,p)|&\leq \frac{e^{\omega' p}e^{-\beta |k|}AD}{(1+p^2)(1+|k|)^{\gamma}}\mathcal{Q}_{2}(|\beta k|),\\ 
|(\hat{W}^{[1]}, \hat{Q}^{[1]})(k,p)|&\leq \frac{e^{\alpha' p}e^{-\beta |k|}\tilde A \tilde D}{(1+p^2)(1+|k|)^{\gamma}}\mathcal{Q}_{2}(|\beta k|),
\end{align*}
\noindent where
\begin{align} \nonumber
AD\geq &C_6\left(\right. \frac{C_0}{\beta}|| (\hat{u}_0, \hat{\Theta}_0)||_{\gamma, \beta}||(\hat{v}_1, \hat{\Theta}_1)||_{\gamma, \beta}+M_1\frac{2}{\beta^2}||(\hat{u}_1, \hat{\Theta}_1)||_{\gamma, \beta}\\ \nonumber
 &\left.\qquad+C_1M_0A^2+2C_1A||(\hat{u}_0, \hat{\Theta}_0)||_{\gamma, \beta}+\frac{a A}{4}\right),\\ \nonumber
\tilde A \tilde D\geq &C_6\left(\right. \frac{2C_0M_3}{\beta}|| (\hat{v}_0, \hat{B}_0)||_{\gamma, \beta}||(\hat{v}_1, \hat{B}_1)||_{\gamma, \beta}+\frac{2M_2}{\beta^2}||(\hat{v}_1, \hat{B}_1)||_{\gamma, \beta}\\ \nonumber
&\left. \qquad +2C_1M_3M_0A^2+4C_1M_3A||(\hat{v}_0, \hat{B}_0)||_{\gamma, \beta}\right).
\end{align}
\end{Lemma}
\begin{proof} Lemma \ref{4.4.} with $l=0$ tells us that 
\begin{align*}
|(\hat{H}^{[1]}, \hat{S}^{[1]})(k,p)|&\leq C_6 \sup_{p'\in [0,p]}|(\hat{R}^{[0]}_1, \hat{R}^{[0]}_2)(k,p')|\\ 
|(\hat{W}^{[1]}, \hat{Q}^{[1]})(k,p)|&\leq C_6 \sup_{p'\in [0,p]}|(\hat{R}^{[0]}_3, \hat{R}^{[0]}_4)(k,p')|
\end{align*}
\noindent since $(\hat{H}^{[0]}, \hat{S}^{[0]})(k,0)=0$ and $(\hat{W}^{[0]}, \hat{Q}^{[0]})(k,0)=0$. In both cases, we use Lemma \ref{4.6.}, Lemma \ref{4.10.}, and Lemma \ref{4.11.} to bound the terms appearing in $R_m$s. The terms are kept in the same order as they appear in $R_m$s as much as possible to help with organization.
\begin{align*}
&|(\hat{R}^{[0]}_1, \hat{R}^{[0]}_2)(k,p)|\leq \frac{2C_0|k|e^{-\beta |k|}}{(1+|k|)^{\gamma}}||(\hat{u}_0, \hat{\Theta}_0)||_{\gamma, \beta}||(\hat{u}_1, \hat{\Theta}_1)||_{\gamma, \beta}\\ 
&\qquad +M_1\frac{|k|^2e^{-\beta |k|}}{(1+|k|)^{\gamma}}||(\hat{v}_1, \hat{\Theta}_1)||_{\gamma,\beta}+C_1M_0A^2 \frac{e^{-\beta |k|
+\omega' p}}{(1+|k|)^{\gamma}(1+p^2)}\mathcal{Q}_{2}(\beta |k|)\\ 
&\qquad +2C_1||(\hat{u}_0, \hat{\Theta}_0)||_{\gamma,\beta}\frac{A e^{-\beta|k|+\omega' p}}{(1+p^2)(1+|k|)^\gamma}\mathcal{Q}_{2}(|\beta k|)+a\frac{e^{\omega' p}e^{-\beta |k|}A}{(1+p^2)(1+|k|)^{\gamma}}
\end{align*}
and
\begin{align*}
&|(\hat{R}^{[0]}_3, \hat{R}^{[0]}_4)(k,p)|\leq \frac{4C_0M_3|k|e^{-\beta |k|}}{(1+|k|)^{\gamma}}||(\hat{v}_0, \hat{B}_0)||_{\gamma, \beta}||(\hat{v}_1, \hat{B}_1)||_{\gamma, \beta}\\ 
&\qquad +M_2\frac{|k|^2e^{-\beta |k|}}{(1+|k|)^{\gamma}}||(\hat{v}_1, \hat{B}_1)||_{\gamma,\beta}+2C_1M_3M_0\tilde A^2 \frac{e^{-\beta |k|+\alpha' p}}{(1+|k|)^{\gamma}(1+p^2)}\mathcal{Q}_{2}(\beta |k|)\\ 
&\qquad+4C_1M_3||(\hat{v}_0, \hat{B}_0)||_{\gamma,\beta}\frac{\tilde A e^{-\beta|k|+\alpha' p}}{(1+p^2)(1+|k|)^\gamma}\mathcal{Q}_{2}(|\beta k|).
\end{align*}
\noindent The lemma now follows since $4|k|\leq \frac{2\mathcal{Q}_2}{\beta}$ and $|k|^2\leq\frac{2\mathcal{Q}_2}{\beta^2}$.
\end{proof}

{\bf Proof of Lemma \ref{4.2.}} Lemma \ref{4.12.} and $(\ref{4.59})$ prove the base case. Suppose, for the purpose of induction, that for $l\geq1$ $(\ref{1-4.60})$ and $(\ref{2-4.60})$ hold. Then by Lemma \ref{4.4.} we need only prove a bound for $|(\hat{R}_1^{[l]}, \hat{R}_2^{[l]})|$ and $|(\hat{R}_3^{[l]}, \hat{R}_4^{[l]})|$ whose terms we bounded in the previous lemmas. 
\begin{align}\nonumber
|(\hat{R}_1^{[l]}, \hat{R}_2^{[l]})|\leq &\frac{AD^{l-1}e^{-\beta |k|+\omega' p}}{(2l+3)^2(1+p^2)(1+|k|)^{\gamma}}\mathcal{Q}_{2l+2}(\beta |k|)\left\{\frac{C_1M_0AD(l+1)^{2/3}(2l+3)^2}{(2l+1)}\right. \\ \nonumber
&+\frac{C_1A(l+1)^{2/3}(2l+3)^2}{4l(2l-1)}+\frac{C_8C_72^{\gamma}\pi Al}{12\beta^d}+\frac{C_1l^{2/3}||(\hat{u}_1, \hat{\Theta}_1)||_{\gamma,\beta}(2l+3)^2}{2l(2l-1)}\\ \nonumber
&\left. +2C_1D||(\hat{u}_0, \hat{\Theta}_0)||_{\gamma,\beta}\frac{(l+1)^{2/3}(2l+3)^2}{2l+1} +25\delta_{l,1}\frac{C_0}{A\beta}||(\hat{u}_1, \hat{\Theta}_1)||^2_{\gamma,\beta}+\frac{a D(2l+3)^2}{4(2l+1)^2}\right\}
\end{align} and 
\begin{align*}
|(\hat{R}_3^{[l]}, \hat{R}_4^{[l]})|\leq& \frac{\tilde A\tilde D^{l-1}e^{-\beta |k|+\alpha' p}}{(2l+3)^2(1+p^2)(1+|k|)^{\gamma}}\mathcal{Q}_{2l+2}(\beta |k|)\left\{M_3\left(\frac{2C_1M_0\tilde A\tilde D(l+1)^{2/3}(2l+3)^2}{(2l+1)}\right. \right.\\ 
&+\frac{2C_1\tilde A(l+1)^{2/3}(2l+3)^2}{4l(2l-1)}+\frac{2C_8C_72^{\gamma}\pi Al}{12\beta^d}+\frac{C_1 l^{2/3}||(\hat{v}_1, \hat{B}_1)||_{\gamma,\beta}(2l+3)^2}{l(2l-1)}\\ 
&\left. \left. +4C_1D||(\hat{v}_0, \hat{B}_0)||_{\gamma,\beta}\frac{(l+1)^{2/3}(2l+3)^2}{2l+1}+25\delta_{l,1}\frac{C_0}{\tilde A\beta}||(\hat{v}_1, \hat{B}_1)||^2_{\gamma,\beta}\right) \right\}.
\end{align*}
\noindent We also note that as $(\hat{H}^{[l]}, \hat{S}^{[l]})$ satisfies $(\ref{1-4.60})$ and $(\hat{W}^{[l]}, \hat{Q}^{[l]})$ satisfies $(\ref{2-4.60})$,
\begin{align}\nonumber
\frac{|k|^2|(\hat{H}^{[l]}, \hat{S}^{[l]})(k,0)|}{(l+1)(l+2)}&\leq \frac{|k|^2e^{-\beta |k|}AD^l \mathcal{Q}_{2l}(\beta|k|)}{(l+1)(l+2)(1+|k|)^{\gamma}(2l+1)^2}\\ \nonumber
&\leq\frac{AD^{l}e^{-\beta |k|+\alpha' p}}{(2l+3)^2(1+p^2)(1+|k|)^{\gamma}}\mathcal{Q}_{2l+2}(\beta |k|)\frac{6}{\beta^2}
\end{align}
and
\begin{equation*}
\frac{|k|^2|(\hat{W}^{[l]}, \hat{Q}^{[l]})(k,0)|}{(l+1)(l+2)}\leq\frac{\tilde A\tilde D^{l}e^{-\beta |k|+\alpha' p}}{(2l+3)^2(1+p^2)(1+|k|)^{\gamma}}\mathcal{Q}_{2l+2}(\beta |k|)\frac{6}{\beta^2}.
\end{equation*}
\noindent Here, we used the following two facts
\begin{equation*}
\frac{y^2\mathcal{Q}_{2l}(y)}{(2l+2)(2l+1)}\leq \mathcal{Q}_{2l+2}(y) \textnormal{ and }\frac{(2l+2)(2l+3)^2}{(l+1)(l+2)(2l+1)}\leq 6.
\end{equation*}
\noindent Thus, for $D$ and $\tilde D$ chosen, independently of $l,k,$ and $p$, large enough so
\begin{align*}
D^2\geq& C_6\left\{\frac{C_1M_0AD(2l+3)^2}{(l+1)(2l+1)}+\frac{C_1A(2l+3)^2}{4(l+1)l(2l-1)}+\frac{C_8C_72^{\gamma}\pi Al}{12\beta^d(l+1)^{5/3}}\right. \\ 
& \qquad+\frac{C_1||(\hat{u}_1, \hat{\Theta}_1)||_{\gamma,\beta}(2l+3)^2}{2(l+1)^{5/3}l^{1/3}(2l-1)} +2C_1D||(\hat{u}_0, \hat{\Theta}_0)||_{\gamma,\beta}\frac{((2l+3)^2}{(l+1)(2l+1)}\\ 
&\qquad +25\delta_{l,1}\frac{C_0}{A2^{5/3}\beta}||(\hat{u}_1, \hat{\Theta}_1)||^2_{\gamma,\beta}+\frac{a D(2l+3)^2}{4(l+1)^{5/3}(2l+1)}\left.\right\}+M_1\frac{6D}{\beta^2}\\ 
\tilde D^2\geq &C_6M_3\left(\right.\frac{2C_1M_0\tilde A\tilde D(2l+3)^2}{(l+1)(2l+1)}+\frac{C_1\tilde A(2l+3)^2}{2(l+1)l(2l-1)}+\frac{C_8C_72^{\gamma}\pi \tilde Al}{12\beta^d(l+1)^{5/3}} \\ 
&\qquad+\frac{C_1 ||(\hat{v}_1, \hat{B}_1)||_{\gamma,\beta}(2l+3)^2}{(l+1)^{5/3}l^{1/3}(2l-1)}
 +4C_1\tilde D||(\hat{v}_0, \hat{B}_0)||_{\gamma,\beta}\frac{(2l+3)^2}{(l+1)(2l+1)}\\ 
&\qquad +25\delta_{l,1}\frac{C_0}{\tilde A\beta 2^{5/3}}||(\hat{v}_1, \hat{B}_1)||^2_{\gamma,\beta}\left.\right) \left.\right\}+M_2\frac{6\tilde D}{\beta^2},
\end{align*}
\noindent $(\ref{1-4.60})$ and $(\ref{2-4.60})$ hold and the lemma is proved. 

As $\mathcal{Q}_{2l}(\beta|k|)\leq 4^l e^{|\beta k|/2}$,
\begin{align}\label{series}
(\hat{H}, \hat{S})(k,p;p_0)&=\sum_{l=0}^{\infty}(\hat{H}^{[l]}, \hat{S}^{[l]})(k,p_0)(p-p_0)^l\\ \nonumber
(\hat{W}, \hat{Q})(k,p;p_0)&=\sum_{l=0}^{\infty}(\hat{W}^{[l]}, \hat{Q}^{[l]})(k,p_0)(p-p_0)^l
\end{align}
\noindent are convergent for $|p-p_0|\leq\frac{1}{4D}$ ( or respectively $|p-p_0|\leq\frac{1}{4D}$) where $D$ is independent of $p_0$. Moreover, the following lemma proved in \cite{smalltime} says that these series are indeed local representations of the solution $(\hat{H}, \hat{S})(k,p)$ or respectively $(\hat{W}, \hat{Q})(k,p)$.

\begin{Lemma}\label{4.13.} The unique solution to ($\ref{1-2.9}$) satisfying $(\hat{H}, \hat{S})(k,0)=0$ guaranteed in Lemma \ref{2.8.} has a local representation given by $(\hat{H}, \hat{S})(k,p;p_0)$ for $p_0\in \mathbb{R}^+$. So, the solution is analytic on $\mathbb{R}^{+}\cup\{ 0\}$. Similarly,
the unique solution to ($\ref{2-2.9}$) satisfying $(\hat{W}, \hat{Q})(k,0)=0$ again guaranteed in Lemma \ref{2.8.} has a local representation given by $(\hat{W}, \hat{Q})(k,p;p_0)$ for $p_0\in \mathbb{R}^+$ and is therefor analytic on $\mathbb{R}^{+}\cup\{ 0\}$.
\end{Lemma}

\begin{proof} The proof is in \cite{smalltime}.
\end{proof}

{\bf Proof of Theorem \ref{Borel summability} i)} We prove the Boussinesq case. The MHD case is the same with the obvious changes. Using Lemma \ref{4.2.} and the fact that $||g||_{L^{\infty}}\leq ||\hat{g}||_{L^1}$ we know that
\begin{align*}
|(H^{[l]},S^{[l]})(x,p_0)|&\leq\frac{8\pi A(4B)^le^{\omega p_0}}{\beta(2l+1)^2(1+p_0^2)}\\
|D(H^{[l]},S^{[l]})(x,p_0)|&\leq\frac{8\pi A(4B)^le^{\omega p_0}}{\beta(2l+1)^2(1+p_0^2)}\\
|D^2(H^{[l]},S^{[l]})(x,p_0)|&\leq\frac{16\pi A(4B)^le^{\omega p_0}}{\beta^2(2l+1)^2(1+p_0^2)}
\end{align*}
and the series (\ref{series}) converges for $|p-p_0|<\frac{1}{4B}$. By Lemma \ref{4.13.} the series is the local representation of the solution guaranteed to exist by Lemma \ref{2.8.} which is zero at $p=0$. Combining this with the facts that the solution is analytic in a neighborhood of zero and exponentially bounded for large p, recall ($\hat{H},\hat{S}\in \mathcal{A}^{\omega}$), implies Borel Summability in $1/t$. Watson's Lemma then implies as $t\rightarrow 0^+$
\begin{equation*} 
(u, \Theta)(x,t)\sim (u_0, \Theta_0)(x)+
\sum_{m=1}^{\infty}(u_m, \Theta_m) (x)t^m
\end{equation*}
\noindent where $|(u_m, \Theta_m)(x)|\leq m! A_0 D_0^m$ with constants
$A_0$ and $D_0$ generally dependent on the initial condition and forcing through Lemma \ref{3.2.}.

\section{Extension of existence time}\label{sec6}

\indent We have shown by Theorem \ref{existence} and Theorem \ref{existenceMHD} that there is a unique solution to (\ref{1-2.18}) and (\ref{2-2.18}) within the class of locally integrable functions, which are exponentially bounded in p, uniformly in $x$. Further, the solutions $(\hat{H}, \hat{S})(k,p)$ and $(\hat{W}, \hat{Q})(k,p)$ generate, in each case, a smooth solution to the Boussinesq and magnetic B$\acute{\textnormal{e}}$nard equation for $t\in [0,\omega^{-1})$ where $\omega$ is the exponential growth rate of the integral equation (\ref{1-2.18}) or respectively, for $t\in [0,\alpha^{-1})$ where $\alpha$ is the exponential growth rate of the integral equation (\ref{2-2.18}). By Theorem \ref{Borel summability} i), we know that the solution is Borel Summable. The question of global existence in either problem can then be reduced to a question of exponential growth for the integral equation solution. If $(\hat{H}, \hat{S})(k,p)$ or $(\hat{W},\hat{Q})(k,p)$ grow subexponentially, then global existence will follow. The exponential growth rate $\omega$ or $\alpha$ previously found is suboptimal and ignores possible cancellations in the integrals. If we improve the estimates, we get a longer interval of existence. One example of improvement is given in the second part of Theorem \ref{Borel summability}, in the special case when the initial condition and forcing have a finite number of Fourier modes, then the radius of convergence in the Borel plane is independent of the size of initial data and forcing. We then prove Theorem \ref{improved existence} which says that based on detailed knowledge of the solution to the integral equation in $[0,p_0)$ given either by the power series at $p=0$ or by numerical calculation, if the solution is small for $p$ towards the right of this interval then $\omega$ or $\alpha$ can be shown to be small. 

\subsection{Improved Radius of Convergence}\label{subsec1}

When the initial data and forcing are analytic Borel summability given in Theorem \ref{Borel summability} implies that 
\begin{align}\label{4.20}
(\hat{H}, \hat{S})(k,p)&=\sum_{m=1}^{\infty}(\hat{u}^{[m]}, \hat{\Theta}^{[m]})(k)\frac{p^{m-1}}{(m-1)!}=\sum_{m=0}^{\infty}(\hat{u}^{[m+1]}, \hat{\Theta}^{[m+1]})(k)\frac{p^{m}}{m!}\\ \nonumber
(\hat{W}, \hat{Q})(k,p)&=\sum_{m=1}^{\infty}(\hat{v}^{[m]}, \hat{B}^{[m]})(k)\frac{p^{m-1}}{(m-1)!}=\sum_{m=0}^{\infty}(\hat{v}^{[m+1]}, \hat{B}^{[m+1]})(k)\frac{p^{m}}{m!}
\end{align}
\noindent has a finite radius of convergence depending on the size of the initial data and forcing. However, in the special case when the initial data and forcing have only a finite number of Fourier modes the radius of convergence is in fact independent of the size of the initial data or $f$. The argument allows forcing to be time dependent. 

\noindent{\bf Proof of Theorem \ref{Borel summability} ii)} We show the Boussinesq case the other begin similar. For small time 
\begin{align*}
(u, \Theta)(x,t)&=(u^{[0]}, \Theta^{[0]})(x)+\sum_{m=1}^{\infty}(u^{[m]}, \Theta^{[m]})(x)t^m\\ 
\hat{f}(k,t)&=\hat{f}^{[0]}+\sum_{m=1}^{\infty}\hat{f}^{[m]}(k)t^m,
\end{align*}
\noindent where by $(\ref{FB2})$ for $m\geq 0$
\begin{align}\label{4.21}
\hat{u}^{[m+1]}&=\frac{1}{m+1}\left[\hat{f}^{[m]}-\nu|k|^2\hat{u}^{[m]}-i k_j P_k\left(\sum_{l=0}^m\hat{u}_j^{[l]}\hat{*}\hat{u}^{[m-l]}\right)+a P_k (e_2\hat{\Theta}^{[m]})\right]\\ \nonumber
\hat{\Theta}^{[m+1]}&=\frac{1}{m+1}\left[-\mu|k|^2\hat{\Theta}^{[m]}-i k_j\left(\sum_{l=0}^m\hat{u}_j^{[l]}\hat{*}\hat{\Theta}^{[m-l]}\right)\right].
\end{align}
Suppose the initial data and forcing have a finite number of Fourier modes. Let $K_1=\max(\sup_{k\in supp(\hat{u}^{[0]}, \hat{\Theta}^{[0]})}|k|,\, \sup_{k\in supp(\hat{f})}|k|)$. Then by induction on $k$ we have $\sup_{k\in supp(\hat{u}^{[m]}, \hat{\Theta}^{[m]})}|k|\leq (m+1)K_1$. Taking the $||\cdot||_{\gamma,\beta}$ norm of both sides of (\ref{4.21}) with respect to $k$ and writing
\begin{equation}\nonumber
a_m=||(\hat{u}^{[m]}, \hat{\Theta}^{[m]})||_{\gamma,\beta},\qquad b_m=||\hat{f}^{[m]}||_{\gamma,\beta},
\end{equation}
we obtain
\begin{align}\nonumber
a_{m+1}\leq&\frac{1}{m+1}\left[b_m+\max(\nu,\mu)\left\||k|^2|(\hat{u}^{[m]},\hat{\Theta}^{[m]})|\right\|_{\gamma,\beta}\right.\\ \nonumber
&\left. \qquad \qquad +\sum_{l=0}^m\left\| |k||\hat{u}^{[l]}|\hat{*}|(\hat{u}^{[m-l]},\hat{\Theta}^{[m-l]})|\right\|_{\gamma,\beta}+a a_m\right]\\ \nonumber
\leq&\frac{1}{m+1}\left[b_m+\max(\nu,\mu)K_1^2(m+1)^2a_m+K_1C_0(m+2)\sum_{l=0}^m a_la_{m-l}+a a_m\right]\\ \nonumber
\leq& \frac{b_m}{m+1}+\frac{a a_m}{m+1}+K_1^2\max(\nu,\mu)(m+1)a_m+2K_1C_0\sum_{l=0}^m a_la_{m-l}.
\end{align}
\noindent Now, consider the formal power series
\begin{equation}\nonumber
y_0(t):=\sum_{m=1}^{\infty}\tilde{a}_m t^m,
\end{equation}
where
\begin{align}\label {4.23}
\tilde{a}_0&=a_0\\ \nonumber
\tilde{a}_{m+1}&=\frac{b_m}{m+1}+\frac{a\tilde{a}_m}{m+1}+K_1^2\max(\nu,\mu)(m+1)\tilde{a}_m+2K_1C_0\sum_{l=0}^m \tilde{a}_l\tilde{a}_{m-l}.
\end{align}
\noindent Clearly, $a_m\leq\tilde{a}_m$, so $y_0(t)$ majorizes $||(\hat{u}, \hat{\Theta})(\cdot,t)||_{\gamma,\beta}$. If we multiply both sides of (\ref{4.23}) by $t^m$ and sum over $m$, then
\begin{align*}
\sum_{m=0}^{\infty}\tilde{a}_{m+1}t^m=\sum_{m=0}^{\infty}\frac{b_m+a \tilde{a}_m}{m+1}t^m +K_1^2&\max(\nu,\mu)\sum_{m=0}^{\infty}(m+1)\tilde{a}_mt^m\\
&+2K_1C_0\sum_{m=0}^{\infty}\sum_{l=0}^m \tilde{a}_l\tilde{a}_{m-l}t^m.
\end{align*}
\noindent In other words, $y_0(t)$ is a formal power series solution to 
\begin{equation}\nonumber
\frac{1}{t}(y-\tilde{a}_0)=w+\frac{a}{t}\int_0^t y(\tau)d\tau+K_1^2\max(\nu,\mu)(ty)'+2K_1C_0 y^2, 
\end{equation}
\noindent where $w(t)=\sum_{m=0}^{\infty}\frac{b_m}{m+1}t^m $. With the change of variables $s=1/t$, we have
\begin{align*}
-K_1^2\max(\nu,\mu)y'+2K_1C_0s^{-1}y^2+(K_1^2&\max(\nu,\mu)s^{-1}-1)y\\ 
&+(s^{-1}w+\tilde{a}_0)+a s\int_0^{1/s}y(\tau)d\tau=0.
\end{align*}
\noindent A singularity of $B(y(s))$ in the Borel plane exhibits itself as an exponential small correction to $y_0$. So, we let $y=y_0+\delta$ and construct the equation for $\delta$:
\begin{align*}
-K_1^2\max(\nu,\mu)\delta'+2K_1C_0s^{-1}(\delta^2+2y_0\delta)+(K_1^2&\max(\nu,\mu)s^{-1}-1)\delta\\
&+a s\int_0^{1/s}\delta(\tau)d\tau=0.
\end{align*}
\noindent If we assume $\delta$ is exponentially small, then to leading order the equation is
\begin{equation}\nonumber
-K_1^2\max(\nu,\mu)\delta'+\left[(4K_1C_0s^{-1}\tilde{a}_0+(K_1^2\max(\nu,\mu))s^{-1}-1\right]\delta=0,
\end{equation}
\noindent which yields
\begin{equation}\nonumber
\delta\sim e^{-K_1^{-2}\max(\nu,\mu)^{-1}s}s^{4\tilde{a}_0C_0 K_1^{-1}\max(\nu,\mu)^{-1}+1}.
\end{equation}
So, the radius of convergence of $B(y)$ is at least $K_1^{-2}\max(\nu,\mu)^{-1}$ which is independent of the size of initial data as claimed. As $y$ majorizes our solution $(\hat{u}, \hat{\Theta})(k,t)$ the radius of convergence of $(\ref{4.20})$ is independent of the size of initial data or forcing as well.

\subsection{Improved growth estimates based on knowledge of the solution to \ref{1-2.18} in $[0,p_0]$.}

Let $(\hat{H},\hat{S})(k,p)$ be the solution to \ref{1-2.18} provided by Theorem \ref{existence}. Define
\begin{equation}\label{l3.9}
(\hat{H}, \hat{S})^{(a)}(k,p)=\begin{cases}(\hat{H}, \hat{S})(k,p) \mbox{ for } p\in (0,p_0]\subset \mathbb{R}^{+}\\
0 \mbox{ otherwise}\end{cases}
\end{equation}
\noindent and
\begin{multline}\nonumber
\hat{H}^{(s)}(k,p)=\frac{i k_j\pi}{2|k|\sqrt{\nu p}}\int_0^{\min(p,2p_0)} \mathcal{G}(z,z')\hat{G}_j^{1,(a)}(k,p')dp'+2\hat u_1(k)\frac{J_1(2|k|\sqrt{\nu p})}{2|k|\sqrt{\nu p}}\\[8pt]\nonumber
\quad+\frac{a\pi}{2|k|\sqrt{\nu p}}\int_0^{\min(p,p_0)} \mathcal{G}(z,z')P_k[e_2\hat{S}^{(a)}(k,p')]dp'\\[8pt]
\hat{S}^{(s)}(k,p)=\frac{i k_j\pi}{2|k|\sqrt{\mu p}}\int_0^{\min(p,2p_0)} \mathcal{G}(\zeta ,\zeta ')\hat{G}_j^{2,(a)}(k,p')dp'+2\hat \Theta_1(k)\frac{J_1(2|k|\sqrt{\mu p})}{2|k|\sqrt{\mu p}},
\end{multline}
\noindent where 
\begin{align}\nonumber
\hat{G}_j^{[1],(a)}(k,p)&=-P_k[\hat{u}_{0,j}\hat{*}\hat{H}^{(a)}+\hat{H}_j^{(a)}\hat{*}\hat{u}_0+\hat{H}_j^{(a)}\, ^{\ast}_{\ast}\hat{H}^{(a)}]\\ \nonumber
\hat{G}_j^{[2],(a)}(k,p)&=-[\hat{u}_{0,j}\hat{*}\hat{S}^{(a)}+\hat{H}_j^{(a)}\hat{*}\hat{\Theta}_0+\hat{H}_j^{(a)}\, ^{\ast}_{\ast}\hat{S}^{(a)}]
\end{align}
\noindent are known functions depending on $(\hat{H}, \hat{S})^{(a)}(k,p)$. Using these 
definitions, we introduce the following functionals dependent on the initial condition, forcing, and $(\hat{H},\hat{S})^{(a)}$.
Further, 
for any chosen $\omega_0 \ge 0$, define
\begin{equation}\label{l3.10}
b=\omega_0 \int_{p_0}^{\infty}e^{-\omega_0 p}||(\hat{H}, \hat{S})^{(s)}(\cdot,p)||_{\gamma,\beta}dp
\end{equation}
\begin{equation}\label{l3.11}
\epsilon_1=\mathcal{B}_1+\mathcal{B}_4+\int_0^{p_0}e^{-\omega_0 p}\mathcal{B}_2(p)dp,
\end{equation}
\noindent where
\begin{multline}\nonumber
\mathcal{B}_0(k)=C_0\sup_{p_0\leq p'\leq p}|\mathcal{G}(z,z')/z|,\hspace{.5 in} \mathcal{B}_1=2\sup_{k\in\mathbb{R}^d}|k|\mathcal{B}_0(k)||(\hat{u}_0, \hat{\Theta}_0)||_{N},\\ \nonumber
\mathcal{B}_2=2\sup_{k\in\mathbb{R}^d}|k|\mathcal{B}_0(k)||(\hat{H}, \hat{S})^{(a)}(\cdot,p)||_{N},\,\, \mathcal{B}_3=\sup_{k\in\mathbb{R}^d}|k|\mathcal{B}_0(k), \,\,\mathcal{B}_4=a\sup_{k\in\mathbb{R}^d}\mathcal{B}_0(k).
\end{multline}
\noindent Now, let $(\hat{H}, \hat{S})^{(b)}=(\hat{H}, \hat{S})-(\hat{H}, \hat{S})^{(a)}$. It is convenient to write the integral equation for $(\hat{H}, \hat{S})^{(b)}$ for $p>p_0$,
\begin{align}\label{l8.63}
\hat{H}^{(b)}(k,p)=\frac{\pi}{2|k|\sqrt{\nu p}}\int_{p_0}^p \mathcal{G}(z,z')\left(i k_j\right.\hat{G}_j^{[1],(b)}(k,p')&+P_k[e_2\hat{s}^{(b)}(k,p'\left.)]\right)dp'\\ \nonumber
&+\hat{H}^{(s)}(k,p)\\ \nonumber
\hat{S}^{(b)}(k,p)=\frac{i k_j\pi}{2|k|\sqrt{\mu p}}\int_{p_0}^p \mathcal{G}(\zeta ,\zeta ')\hat{G}_j^{[2],(b)}(k,p')dp'&+\hat{S}^{(s)}(k,p),
\end{align}
\noindent where 
\begin{align}\nonumber
\hat{G}_j^{[1],(b)}(k,p)=-P_k[\hat{u}_{0,j}\hat{*}\hat{H}^{(b)}+\hat{H}_j^{(b)}\hat{*}\hat{u}_0+\hat{H}_j^{(a)}\, ^{\ast}_{\ast}\hat{H}^{(b)}+\hat{H}_j^{(b)}\, ^{\ast}_{\ast}\hat{H}^{(a)}+\hat{H}_j^{(b)}\, ^{\ast}_{\ast}\hat{H}^{(b)}]\\ \nonumber
\hat{G}_j^{[2],(b)}(k,p)=-[\hat{u}_{0,j}\hat{*}\hat{S}^{(b)}+\hat{H}_j^{(b)}\hat{*}\hat{\Theta}_0+\hat{H}_j^{(a)}\, ^{\ast}_{\ast}\hat{S}^{(b)}+\hat{H}_j^{(b)}\, ^{\ast}_{\ast}\hat{S}^{(a)}+\hat{H}_j^{(b)}\, ^{\ast}_{\ast}\hat{S}^{(b)}].
\end{align}
\noindent We also define 
\begin{equation}
\hat{R}^{(b)}(k,p)=i k_j(\hat{G}_j^{[1]}, \hat{G}_j^{[2]})^{(b)}(k,p)+a P_k[e_2\hat{S}^{(b)}(k,p)].
\end{equation}

\noindent{\bf Proof of Theorem \ref{improved existence}}
We note that
\begin{multline}\nonumber
|R^{(b)}(k,p)|\leq \left(|k|\left[|\hat{u}_0|\hat{*}|(\hat{H}, \hat{S})^{(b)}|+|\hat{H}^{(b)}|\hat{*}|(\hat{u}_0, \hat{\Theta}_0)|+2|(\hat{H}, \hat{S})^{(a)}|\, ^{\ast}_{\ast}|(\hat{H}, \hat{S})^{(b)}|\right.\right.\\ \nonumber
\left.\left.+|\hat{H}^{(b)}|\, ^{\ast}_{\ast}|(\hat{H}, \hat{S})^{(b)}|\right]+a|\hat{H}^{(b)}|\right)(k,p),
\end{multline}
\noindent where $|\cdot|$ is the usual euclidean norm. Let $\psi(p)=||(\hat{H}, \hat{S})^{(b)}(\cdot,p)||_{\gamma, \beta}$. Then
\begin{multline}\nonumber
\left\|\left(\frac{\mathcal{G}(z,z')}{z}(i k_j(\hat{G}_j^{[1]})^{(b)}(k,p)+a P_k[e_2\hat{S}^{(b)}(k,p)]), \frac{\mathcal{G}(\zeta,\zeta')}{\zeta}i k_j(\hat{G}_j^{[2]})^{(b)}(k,p)\right)\right\|_{\gamma,\beta}\leq \mathcal{B}_0(k)\cdot \hspace{.5 in}\\ \nonumber
\left(|k|\left[||\hat{u}_0||_{\gamma,
\beta}\psi(p)+\psi(p)||(\hat{u}_0, \hat{\Theta}_0)||_{\gamma,\beta}+ 2||(\hat{H}, \hat{S})^{(a)}||_{\gamma,\beta}*\psi(p)+\psi(p)*\psi(p)\right]+a\psi(p)\right)(k,p)\\ \nonumber
= \left(\mathcal{B}_1\psi+\mathcal{B}_2*\psi+\mathcal{B}_3\psi*\psi+\mathcal{B}_4\psi\right)(p).
\end{multline}
\noindent Taking the $(\gamma,\beta)$ norm in $k$ on both sides of (\ref{l8.63}) and multiplying by $e^{-\omega p}$ for $\omega\geq\omega_0\geq0$ and integrating from $p_0$ to $M$ gives
\begin{multline}\nonumber
L_{p_0,M}:=\int_{p_0}^Me^{-\omega p}\psi(p)dp\leq \int_{p_0}^Me^{-\omega p}\int_{p_0}^p\left(\mathcal{B}_1\psi+ \mathcal{B}_2*\psi+\mathcal{B}_3\psi*\psi+\mathcal{B}_4\psi\right)(p')dp'dp\\ \nonumber +\int_{p_0}^Me^{-\omega p}\psi^{(s)}(p)dp \leq \int_{p_0}^M\int_{p_0}^{p'}e^{-\omega (p-p')}e^{-\omega p'}\left(\mathcal{B}_1\psi+ \mathcal{B}_2*\psi+\mathcal{B}_3\psi*\psi+\mathcal{B}_4\psi\right)(p')dp dp'\\ \nonumber 
+\int_{p_0}^Me^{-\omega p}\psi^{(s)}(p)dp\leq \frac{1}{\omega}\int_{p_0}^Me^{-\omega p'}\left(\mathcal{B}_1\psi+ \mathcal{B}_2*\psi+\mathcal{B}_3\psi*\psi+\mathcal{B}_4\psi\right)(p')dp'\\ \nonumber
+\int_{p_0}^Me^{-\omega p}\psi^{(s)}(p)dp,
\end{multline}
\noindent where $\psi^{(s)}=||(\hat{H}, \hat{S})^{(s)}(\cdot,p)||_{\gamma,\beta}$. Recalling that $\psi=0$ on $[0,p_0]$, we note that for any $u$
\begin{align*}
\int_{p_0}^Me^{-\omega p}(\psi *u)(p)dp&=\int_{p_0}^M\int_{p_0}^p e^{-\omega p}\psi(s)u(p-s)ds dp\\
&=\int_{p_0}^M\psi(s)e^{-\omega s}\int_{0}^{M-s}e^{-\omega p}u(p)dp ds.
\end{align*}
\noindent Using this, we obtain
\begin{multline}\nonumber
L_{p_0,M}\leq\frac{1}{\omega}\left\{(\mathcal{B}_1+\int_0^{M-p_0}e^{-\omega p}\mathcal{B}_2(p)dp)L_{p_0,M}+\mathcal{B}_3L_{p_0,M}^2+\mathcal{B}_4L_{p_0,M}\right\}+b\omega^{-1}\\ \nonumber
\leq \omega ^{-1}\left\{\epsilon_1L_{p_0,M}+\mathcal{B}_3L_{p_0,M}^2\right\}+b\omega^{-1}.
\end{multline}
\noindent For 
\begin{equation}\nonumber
\epsilon_1<\omega\quad \mbox{and}\quad(\epsilon_1-\omega)^2>4\mathcal{B}_3b,
\end{equation}
\noindent we get an estimate for $L_{p_0,M}$ that is independent of $M$. Namely,
\begin{equation*}
L_{p_0,M}\leq\frac{1}{2\mathcal{B}_3}\left[\omega-\epsilon_1-\sqrt{(\epsilon_1-\omega)^2-4\mathcal{B}_3b}\right].
\end{equation*}

So, $||(\hat{H}, \hat{S})(\cdot,p)||_{\gamma,\beta}\in L^1(e^{-\omega p}dp)$, and the solution to the Boussinesq exists for $t\in(0,\omega^{-1})$ for $\omega$ sufficiently large so that 
\begin{equation}\nonumber
\omega\geq \omega_0 \quad\mbox{and}\quad \omega>\epsilon_1+2\sqrt{\mathcal{B}_3b}.
\end{equation}
Equivalently, we could choose our original $\omega_0$ large enough so that $ \omega_0>\epsilon_1+2\sqrt{\mathcal{B}_3b}$. This completes the proof of Theorem \ref{improved existence}.

\section{Appendix}\label{sec8}

\begin{Lemma}\label{kernel} The kernel $\mathcal{G}(z,z')$ given by
\begin{equation}\nonumber
\mathcal{G}(z,z')=z'(-J_1(z)Y_1(z')+Y_1(z)J_1(z')),\textnormal{where } z=2|k|\sqrt{\nu p} \textnormal{ and } z'=2|k|\sqrt{\nu p'}
\end{equation}
satisfies $\frac{\pi}{z}\mathcal{G}(z,z')=\mathcal{H}^{(\nu)}(p,p',k)$ with
\begin{align}\nonumber
\mathcal{H}^{(\nu)}(p,p',k)=&\int_{p'/p}^1\left\{\frac{1}{2\pi i}\int_{c-i\infty}^{c+i\infty}\tau^{-1}exp[-\nu|k|^2\tau^{-1}(1-s)+(p-p's^{-1})\tau]d\tau\right\}ds\\ \nonumber
=&\frac{p'}{p}\int_1^{p/p'}F(\eta)ds,
\end{align}
 where
\begin{equation}\nonumber
\eta=\nu |k|^2p\left(1-\frac{sp'}{p}\right)\left(1-\frac{1}{s}\right),\quad F(\eta)=\frac{1}{2\pi i}\int_{C}\zeta^{-1}e^{\zeta-\eta\zeta^{-1}}d\zeta,
\end{equation}
and $C$ is the contour starting and $\infty e^{-\pi i}$ turning around the origin in counterclockwise direction and ending at $\infty e^{\pi i}$.
\end{Lemma}
\textit{Proof.} We will show that $\mathcal{H}^{(\nu)}(p,p',k)$ solves $(p\partial_{pp}+2\partial_{p}+\nu|k|^2)\mathcal{H}^{(\nu)}=0$ for $0<p'<p$ with the condition that $\mathcal{H}^{(\nu)}(p,p',k)\rightarrow 0$ and $\mathcal{H}^{(\nu)}_p(p,p',k)\rightarrow \frac{1}{p}$ as $p'$ approaches p from below.

In the appendix of \cite{longtime}, it is shown that $F$ is entire, $F(0)=1$, and $F$ satisfies $\eta F''(\eta)+F'(\eta)+F(\eta)=0$. We will use these facts as given. As $F$ is continuous and the interval of integration shrinks to length zero, $\mathcal{H}^{(\nu)}(p,p',k)\rightarrow 0$ as $p'$ tends to $p$ from below. For $p> p'$, $\mathcal{H}^{(\nu)}$ is twice differentiable in $p$ as $F$ is twice continuously differentiable. Moreover, we have
\begin{align}\nonumber
\mathcal{H}^{(\nu)}_p(p,p',k)&=-\frac{1}{p}\mathcal{H}^{(\nu)}(p,p',k)+\frac{1}{p}F(0)+\frac{p'}{p}\int_1^{p/p'}F'(\eta)\frac{d\eta}{dp}ds,\\ \nonumber
(p\mathcal{H}^{(\nu)}_p)_p&=-\mathcal{H}^{(\nu)}_p+F'(0)\nu|k|^2(1-\frac{p'}{p})+p'\int_1^{p/p'}F''(\eta)\left(\frac{d\eta}{dp}\right)^2ds,
\end{align}
where the second equality uses that $\frac{d\eta}{dp}=\nu|k|^2\left(1-\frac{1}{s}\right)$ is $p$ independent. Thus, as $F(0)=1$, we have $\mathcal{H}^{(\nu)}_p(p,p',k)\rightarrow \frac{1}{p}$ as $p'$ tends to $p$ from below.
We notice that 
\begin{multline}\nonumber
\frac{d\eta}{dp}=\nu|k|^2\left(1-\frac{1}{s}\right)=\frac{\eta}{p-sp'}, \, -\frac{d\eta}{ds}\frac{s}{p}=\nu|k|^2\frac{(p's^2-p)}{ps}, \textnormal{ and}\hspace{1.5 in}\\ \nonumber
\left(\frac{d\eta}{dp}\right)^2=\frac{\eta\nu|k|^2}{p}\left(1+\frac{p's^2-p}{s(p-sp')}\right)=\frac{\eta\nu|k|^2}{p}-\frac{\eta s}{p(p-sp')}\frac{d\eta}{ds}=\frac{\eta\nu|k|^2}{p}-\frac{\nu|k|^2(s-1)}{p}\frac{d\eta}{ds}.
\end{multline}
So, integrating by parts and using $\eta F''(\eta)+F'(\eta)+F(\eta)=0$, we have
\begin{align}\nonumber
(p\mathcal{H}^{(\nu)}_p)_p+\mathcal{H}^{(\nu)}_p=&F'(0)\nu|k|^2(1-\frac{p'}{p})+p'\int_1^{p/p'}F''(\eta)\left(\frac{\eta\nu|k|^2}{p}\right)ds\\ \nonumber
&\qquad \qquad-p'\int_1^{p/p'}\frac{d}{ds}(F'(\eta))\frac{\nu|k|^2(s-1)}{p}ds\\ \nonumber
=&\frac{\nu|k|^2p'}{p}\int_1^{p/p'}\eta F''(\eta)ds+\frac{p'\nu|k|^2}{p}\int_1^{p/p'}F'(\eta)ds=-\nu|k|^2\mathcal{H}^{(\nu)}.
\end{align}
In other words, $p\mathcal{H}^{(\nu)}_{pp}+2\mathcal{H}^{(\nu)}_p+\nu|k|^2\mathcal{H}^{(\nu)}=0$, and the Lemma is proved.

\begin{Lemma}\label{U_0} We also have the representation in terms of Bessel functions
\begin{equation}\nonumber
\mathcal{L}^{-1}\left(\frac{1-e^{-\nu|k|^2\tau^{-1}}}{\nu|k|^2}\right)(p)=\frac{2J_1(z)}{z}.
\end{equation}
\end{Lemma}
\begin{proof}
Notice that by contour deformation the integral of $\frac{1}{\nu|k|^2}$ is zero. Factoring out $|k|\sqrt{\nu p}$ in the exponent and using the change of variables $\frac{\tau\sqrt{p}}{|k|\sqrt{\nu}}\rightarrow w$, we have
\begin{align*}
\mathcal{L}^{-1}\left(\frac{1-e^{-\nu|k|^2\tau^{-1}}}{\nu|k|^2}\right)(p)&=\frac{-1}{2\pi i}\int_{c-i\infty}^{c+i\infty}\frac{e^{-\nu|k|^2\tau^{-1}+p\tau}}{\nu|k|^2}d\tau\\
&=\frac{-1}{2\pi i}\int_{c-i\infty}^{c+i\infty}\frac{e^{|k|\sqrt{\nu p}(w-w^{-1})}}{|k|\sqrt{\nu p}}dw
=2\frac{J_1(z)}{z}.
\end{align*}
\end{proof}

\begin{flushleft}
\bf{Fourier Inequalities in two dimensions}
\end{flushleft}

In the appendix of \cite{smalltime}, Fourier inequalities are developed in $\mathbb{R}^3$. We present the counterparts to those inequalities in $\mathbb{R}^2$ here. Where a Lemma is referenced from this section, we use either the $\mathbb{R}^2$ version or $\mathbb{R}^3$ version
as appropriate for our two problems. The basic idea is that in 2-d Lemma \ref{6.6.} below differs by a constant from 3-d case. All other lemmas are basically the same for $\mathbb{R}^2$ or $\mathbb{R}^3$ once the change in Lemma \ref{6.6.} is taken into account.
\begin{Definition}\label{d6.1.} Define the polynomial
\begin{equation}\nonumber
P_n(z)=\sum_{j=0}^n\frac{n!}{j!}z^j.
\end{equation}
\end{Definition}
\begin{Remark}\label{r6.2.} Integration by parts gives 
\begin{equation}\nonumber
\int_0^z e^{-\tau}\tau ^n d\tau=-e^{-z}P_n(z)+n!.
\end{equation}
\end{Remark}
\begin{Lemma}\label{6.3.} For all $y\geq 0$ and nonnegative integers $m$, $n$ we have
\begin{equation}\nonumber
y^{m+1}\int_0^1\rho ^m P_n(y(1-\rho))d\rho=m!n!\sum_{j=0}^n\frac{y^{m+j+1}}{(m+j+1)!}.
\end{equation}
\end{Lemma}
\begin{proof} Integration by parts gives
\begin{equation}\nonumber
\int_0^1(1-\rho)^j\rho^m d\rho=\frac{m!j!}{(m+j+1)!}.
\end{equation}
\noindent The result now follows by a direct calculation using the definition of $P_n$ given by Definition \ref{d6.1.}.
\end{proof}

\begin{Lemma}\label{6.4.} For all $y\geq 0$ and integers $n\geq m \geq 0$, we have 
\begin{equation}\nonumber
y^{m+1}\int_1^{\infty}e^{-2y(\rho -1)}\rho ^m P_n(y(\rho-1))d\rho\leq 2^{-m} (m+n)!\sum_{j=0}^m\frac{y^{j}}{j!}.
\end{equation}
\end{Lemma}

\begin{proof} This again follows from direct calculation and is the same as in \cite{smalltime}. 
\end{proof}
\begin{Lemma}\label{6.5.} For all $y\geq 0$ and integers $n\geq m \geq 0$, we have 
\begin{equation*}
y^{m+1}\int_0^{\infty}e^{-y(\rho -1)[1+sgn(\rho -1)]}\rho ^m P_n(y|1-\rho|)d\rho\leq m!n!\mathcal{Q}_{m+n+1}(y).
\end{equation*}
\end{Lemma}
\noindent \textit{Proof.} This is a combination of the previous two lemmas after splitting the integral at $1$.

\begin{Proposition}\label{2-dbound} Let n be an integer no less than $0$ and $r\geq 0$ and $\rho \geq0$ fixed. Then
\begin{equation}\nonumber
\int_0^{2\pi}e^{-|\rho-re^{i \theta}|}|\rho-re^{i \theta}|^n d\theta\leq 6\pi e^{-|\rho-r|}P_n(|r-\rho|).
\end{equation}
\end{Proposition}
\begin{proof} Let $f(\theta)=e^{-|\rho-re^{i \theta}|}|\rho-re^{i \theta}|^n$. Then notice that $f'(\theta)=
e^{-|\rho-re^{i \theta}|}|\rho-re^{i \theta}|^{n-2}\rho r \sin(\theta)(-|\rho-re^{i\theta}|+n)$. We want to maximize $f(\theta)$, so we split into two cases. 

\noindent Case 1. Suppose $r\leq\rho-n$ or $r\geq\rho+n$. As $|\rho-re^{i\theta}|\geq n$, $f(\theta)$ reaches its maximum at $\theta=0$. Thus, $|f(\theta)|\leq e^{-|\rho-r| }|\rho-r|^n\leq e^{-|\rho-r|}P_n(|\rho-r|)$. If $n=0$ this is the only case to consider. For $n\geq1$ we have a second case.

\noindent Case 2. Suppose $\rho-n<r<\rho+n$. Now, $f(\theta)$ is maximized for $\theta$ such that $|\rho-re^{i\theta}|=n$. Hence, $|f(\theta)|\leq e^{-n}n^n$. Now, we use the fact that for $r\in(\rho-n,\rho+n)$
\begin{equation}\nonumber 
e^{|\rho-r|}\leq\sum_{j=0}^n\frac{|\rho-r|^j}{j!}+\frac{e^n}{(n+1)!}=\frac{P_n(|\rho-r|)}{n!}+\frac{e^n}{(n+1)!}.
\end{equation}
\noindent So, 
\begin{equation}\nonumber
|f(\theta)|\leq e^{-n}n^n\leq e^{-|\rho-r|}\left(\frac{P_n(|\rho-r|)e^{-n}n^n}{n!}+\frac{n^n}{(n+1)!}\right)\leq 3e^{-|\rho-r|}P_n(|\rho-r|),
\end{equation}
where the last inequality uses $e^{-n}n^n\leq n!$ and $\frac{n^n}{(n+1)!}\leq \frac{e^n}{n+1}\leq2 P_n(|\rho-r|)$. Putting these two cases together bounds the integrand by $3e^{-|\rho-r|}P_n(|\rho-r|)$ and the proposition follows.
\end{proof}

\begin{Lemma}\label{6.6.} If $m$ and $n$ are integers no less than $-1$, then
\begin{equation}\nonumber
|q|\int_{q'\in \mathbb{R}^d} e^{|q|-|q'|-|q-q'|}|q'|^m|q-q'|^n dq'\leq C_7(d) \pi (m+1)!(n+1)!\mathcal{Q}_{m+n+3}(|q|),
\end{equation}
\noindent where $C_7(2)=18$ and $C_7(3)=2$.
\end{Lemma}

\begin{proof} We note that we may assume without loss of generality that $m\leq n$ since a change of variables $q'\rightarrow q-q'$ switches the roles of $m$ and $n$. Write $q=\rho e^{i \phi}$, $q'=re^{i\varphi}$ and $\theta=\varphi-\phi$. Let I be the integral on the left hand side. Then switching to polar coordinates gives
\begin{equation}\nonumber
I=\rho\int_0^{\infty}\int_0^{2\pi}e^{\rho-r-|\rho-re^{i\theta}|}r^m|\rho-re^{i\theta}|^n r dr d\theta.
\end{equation}
For $n\geq 0$, using Proposition \ref{2-dbound} above gives,
\begin{equation}\nonumber
I\leq 6\pi\rho\int_0^{\infty}e^{\rho-r}r^{m+1}|\rho-re^{i\theta}|^n e^{|\rho-r|}P_n(|\rho-r|)dr.
\end{equation}
\noindent Now, we let $\tilde{\rho}=\frac{r}{\rho}$. Then $d\tilde{\rho}=\frac{dr}{\rho}$ and $-|\rho-r|=-\rho(\tilde{\rho}-1)sgn(\tilde{\rho}-1)$, so
\begin{equation}\nonumber
I\leq 6\pi\rho^{m+3}\int_0^{\infty}e^{-\rho(\tilde{\rho}-1) (1+sgn(\tilde{\rho}-1))}\tilde{\rho}^{m+1}P_n(\rho|\tilde{\rho}-1|)d\tilde{\rho}.
\end{equation}
\noindent Applying Lemma \ref{6.5.} with $m=m+1$ and $n=n$ gives
\begin{equation}\nonumber
I\leq 6\pi\rho(m+1)!n!\mathcal{Q}_{m+n+2}(\rho)\leq 18\pi(m+1)!(n+1)!\mathcal{Q}_{m+n+3}(\rho),
\end{equation}
\noindent where the last inequality follows as $m\leq n$, so 
\begin{align*}
\rho\sum_{j=0}^{m+n+2}\frac{2^{m+n+2-j}\rho^{j}}{j!}&\leq\sum_{j=1}^{m+n+3}\frac{2^{m+n+3-j}\rho^j}{(j-1)!}\\
&\leq \mathcal{Q}_{m+n+3}(\rho)(m+n+3)\leq3(n+1)\mathcal{Q}_{m+n+3}(\rho).
\end{align*} 
For $n=m=-1$, we use a slightly different approach. Assuming $q$ is not zero, we split the integral over two regions, a ball of radius $3|q|/2$ centered at zero and its compliment. For the compliment region we have $|q-q'|\geq|q|/2$, so
\begin{align*}
|q|\int_{|q'|\geq 3|q|/2}e^{|q|-|q'|-|q-q'|}&\frac{1}{|q'||q-q'|}dq'\\
&\leq 2e^{|q|/2}\int_0^{2\pi}\int_{3|q|/2}^{\infty}e^{-r} drd\theta=4\pi e^{-|q|}\leq4\pi.
\end{align*}
For the interior region we have 
\begin{equation*}
|q|\int_{|q'|\leq 3|q|/2}e^{|q|-|q'|-|q-q'|}\frac{1}{|q'||q-q'|}dq'\leq |q|\int_{|q'|\leq 3|q|/2}\frac{1}{|q'||q-q'|}dq'.
\end{equation*}
We now note that $\int_{|q'|\leq 3|q|/2}\frac{1}{|q'||q-q'|}dq'$ is bounded. Without trying to be precise we can bound the integral by $13\pi$ by spitting the region into two disks of radius $|q|/2$ centered at $0$ and $q$ and the compliment, call the compliment $D$. We have
\begin{equation*}
\int_{|q'|\leq|q|/2}\frac{1}{|q'||q-q'|}dq'\leq\frac{2}{|q|}\int_{|q'|\leq|q|/2}\frac{1}{|q'|}dq'\leq 2\pi.
\end{equation*}
Similarly,
\begin{equation*}
\int_{|q'-q|\leq|q|/2}\frac{1}{|q'||q-q'|}dq'\leq 2\pi.
\end{equation*}
Finally,
\begin{equation*}
\int_{D}\frac{1}{|q'||q-q'|}dq'\leq\frac{4}{|q|^2}\int_{D}dq'\leq \frac{4}{|q|^2}\int_{|q'|\leq 3|q|/2}dq'\leq9\pi.
\end{equation*}
Thus,
\begin{equation*}
|q|\int_e^{|q|-|q'|-|q-q'|}\frac{1}{|q'||q-q'|}dq'\leq 13\pi|q|+4\pi\leq 18(|q|+2)=18Q_1(|q|)
\end{equation*}
for all nonzero $q$. Hence, the lemma is proved with $C_7(2)=18$.
\end{proof}

\begin{Lemma}\label{6.7.} For any $\gamma \geq 1$ and nonnegative integers $m$ and $n$, we have
\begin{align}\nonumber
|k|\int_{k'\in \mathbb{R}^d} &\frac{e^{-\beta(|k'|+|k-k'|)}}{(1+|k'|)^{\gamma}(1+|k-k'|)^{\gamma}}(\beta|k'|)^m(\beta|k-k'|)^n dk'\\ \nonumber
&\quad \leq \frac{C_7 \pi 2^{\gamma}e^{-\beta|k|}m!n!}{\beta^d(1+|k|)^{\gamma}}(m+n+2)\mathcal{Q}_{m+n+2}(\beta|k|).
\end{align}
\end{Lemma}
\begin{proof} The proof is exactly as in \cite{smalltime} using our new bound in Lemma \ref{6.6.}. The idea is to split into two regions $|k'|\leq |k|/2$ and its compliment. In the ball, we have
\begin{equation}\nonumber
\frac{1}{(1+|k-k'|)^{\gamma}(1+|k'|)^{\gamma}}\leq\frac{\beta}{(1+|k|/2)^{\gamma}|\beta k'|},
\end{equation}
and we use Lemma \ref{6.6.} with $m$ replaced by $m-1$. In the compliment, we have
\begin{equation}\nonumber
\frac{1}{(1+|k-k'|)^{\gamma}(1+|k'|)^{\gamma}}\leq\frac{\beta}{(1+|k|/2)^{\gamma}|\beta(k- k')|},
\end{equation}
and we use Lemma \ref{6.6.} with $n$ replaced by $n-1$.
\end{proof}

\begin{Lemma}\label{6.8.} For any $\gamma \geq 2$ and $n\in \mathbb{N}-0$, we have 
\begin{align}\nonumber
|k|\int_{k'\in \mathbb{R}^d} &\frac{e^{-\beta(|k'|+|k-k'|)}}{(1+|k'|)^{\gamma}(1+|k-k'|)^{\gamma}}|\beta(k-k')|^n dk'\\ \nonumber
&\leq \frac{C_7\pi 2^{\gamma}e^{-\beta|k|}}{\beta^{d-1}(1+|k|)^{\gamma}}\left\{(n-1)!\mathcal{Q}_{n+1}(|q|)+\frac{3(n+1)!|q|^{2/3}}{2\beta ^{2/3}}\sum_{j=0}^{n+1}\frac{|q|^j}{j!}\right\}.
\end{align}
\end{Lemma}

\begin{proof} We again break into two integrals $\int_{|k'|\leq |k|/2}+\int_{|k'|\geq |k|/2}$. In the outer region, we have $(1+|k'|)^{-\gamma}\leq2^{\gamma}(1+|k|)^{-\gamma}$, and in the inner, we have $(1+|k-k'|)^{-\gamma}\leq 2^{\gamma}(1+|k|)^{-\gamma}$. We use this and $\gamma \geq 2$ for the first inequality and Lemma \ref{6.6.} for the second to get a bound for the outer region
\begin{align}\nonumber
|k|\int_{|k'|\geq |k|/2}& \frac{e^{-\beta(|k'|+|k-k'|)}}{(1+|k'|)^{\gamma}(1+|k-k'|)^{\gamma}}|\beta(k-k')|^n dk'\hspace{2.5 in}\\ \nonumber
&\leq \frac{2^{\gamma}e^{-\beta|k|}}{\beta^{d-1}(1+|k|)^{\gamma}}|q|\int_{q'\in \mathbb{R}^d}e^{|q|-|q'|-|q-q'|}|q-q'|^{n-2} dq'\\ \nonumber
&\leq\frac{C_7 \pi 2^{\gamma }e^{-\beta|k|}}{\beta^{d-1}(1+|k|)^{\gamma}}(n-1)!\mathcal{Q}_{n+1}(|q|).
\end{align}
\noindent In the inner region, we also use $(1+|k'|)^{-\gamma}\leq (|k'|)^{-2+2/3}$, a change to polar coordinates as in the proof of Lemma \ref{6.6.}, and integration by parts to get
\begin{align}\nonumber
|k|&\int_{|k'|\leq|k|/2}\frac{e^{-\beta(|k'|+|k-k'|)}}{(1+|k'|)^{\gamma}(1+|k-k'|)^{\gamma}}|\beta(k-k')|^n dk'\\ \nonumber
&\leq \frac{2^{\gamma}e^{-\beta|k|}}{\beta^{d-1+2/3}(1+|k|)^{\gamma}}|q|\int_{|q'|\leq|q|/2}e^{|q|-|q'|-|q-q'|}|q'|^{-2+2/3}|q-q'|^{n}dq'\\ \nonumber
&=\frac{  2^{\gamma}e^{-\beta|k|}}{\beta^{d-1+2/3}(1+|k|)^{\gamma}}\rho\int_0^{\rho/2}\int_0^{2\pi}e^{\rho-r-|\rho-re^{i\theta}|}|\rho-re^{i\theta}|^n r^{-2+2/3}rd\theta dr\\ \nonumber
&\leq \frac{2^{\gamma}e^{-\beta|k|}}{\beta^{d-1+2/3}(1+|k|)^{\gamma}}6\pi \rho\int_0^{\rho/2}r^{-1+2/3}P_n(|\rho-r|)dr\\ \nonumber
&\leq \frac{2^{\gamma}e^{-\beta|k|}}{\beta^{d-1+2/3}(1+|k|)^{\gamma}} 6\pi n!\rho^{1+2/3}\sum_{j=0}^n\frac{\rho^{j}}{j!}\int_0^{1}\tilde{r}^{-1+2/3}(1-\tilde{r})^j d\tilde{r}\\ \nonumber
&\leq \frac{2^{\gamma}e^{-\beta|k|}}{\beta^{d-1+2/3}(1+|k|)^{\gamma}} \frac{18}{2}\pi n!\rho^{2/3}\sum_{j=0}^n\frac{\rho^{j+1}}{j!}.
\end{align}
\end{proof}

\begin{Lemma}\label{6.9.} For any $\gamma \geq 1$ and nonnegative integers $l_1, l_2 \geq 0$, we have 
\begin{multline}\nonumber
|k|\int_{k'\in \mathbb{R}^d} \frac{e^{\beta(|k|-|k'|-|k-k'|)}}{(1+|k'|)^{\gamma}(1+|k-k'|)^{\gamma}}\mathcal{Q}_{2l_1}(\beta|k'|)\mathcal{Q}_{2l_2}(\beta|k-k'|)dk'\hspace{1.5 in}\\ \nonumber
\leq \frac{C_7 \pi 2^{\gamma}e^{-\beta|k|}}{3\beta^d(1+|k|)^{\gamma}}(2l_1+2l_2+1)(2l_1+2l_2+2)(2l_1+2l_2+3)\mathcal{Q}_{2l_1+2l_2+2}(\beta|k|).
\end{multline}
\end{Lemma}
\noindent The proof is exactly the same as in \cite{smalltime} with $K=\frac{C_7 \pi 2^{\gamma}e^{-\beta|k|}}{\beta^d(1+|k|)^{\gamma}}$. The idea of the proof is to use the definition of $\mathcal{Q}_{2l_1}$ and $\mathcal{Q}_{2l_2}$ with Lemma \ref{6.7.} to bound the left hand side by
\begin{align*}
K\sum_{j=0}^{2l_1+2l_2}2^{2l_1+2l_2+2-(j+2)}&(j+2)(j+1)\mathcal{Q}_{j+2}(|q|)\\
&\leq K\mathcal{Q}_{2l_1+2l_2+2}(|q|)\sum_{j=0}^{2l_1+2l_2}(j+1)(j+2)
\end{align*}
from which the result follows.

\begin{Lemma}\label{6.10.} If $\gamma\geq 2$ and $l\geq 0$, then
\begin{align}\nonumber
\frac{|k|}{(l+1)^{2/3}}\int_{k'\in\mathbb{R}^d}&\frac{e^{-\beta(|k'|+|k-k'|)}}{(1+|k'|)^{\gamma}(1+|k-k'|)^{\gamma}}\mathcal{Q}_{2l}(|\beta(k-k')|)dk'\\
&\leq \frac{C_1e^{-\beta|k|}}{(1+|k|)^{\gamma}}(2l+1)\mathcal{Q}_{2l+2}(\beta|k|),
\end{align}
\noindent where
\begin{equation}\nonumber
C_1=C_1(d)=6C_7\pi2^{\gamma}\beta^{-d+1/3}+C_7\pi2^{\gamma}\beta^{-d+1}+\frac{1}{2}C_0\beta^{-1}.
\end{equation}
\end{Lemma}

\begin{proof} The proof is again the same as in \cite{smalltime} except when Lemma 6.8. is invoked in \cite{smalltime} we use our Lemma \ref{6.8.}. The idea is to split into a few cases. When $l=0$, the claim holds with $C_1=\frac{1}{2}C_0\beta^{-1}$. For $l\geq 1$, we separate the constant term
\begin{equation}\nonumber
|k|\int_{k'\in\mathbb{R}^d}\frac{e^{-\beta(|k'|+|k-k'|)}}{(1+|k'|)^{\gamma}(1+|k-k'|)^{\gamma}}2^{2l}dk'
\leq\frac{C_0e^{-\beta|k|}}{2\beta(1+|k|)^{\gamma}}\mathcal{Q}_{2l+2}(\beta|k|).
\end{equation}
Then, we use Lemma \ref{6.8.} to bound the terms of 
\begin{equation}\nonumber
|k|\int_{k'\in\mathbb{R}^d}\frac{e^{-\beta(|k'|+|k-k'|)}}{(1+|k'|)^{\gamma}(1+|k-k'|)^{\gamma}}(\mathcal{Q}_{2l}(\beta|k|)-2^{2l})dk'
\end{equation}
and over bound the remaining sums to get the rest of the terms appearing in $C_1(d)$.
\end{proof}




\end{document}